\documentclass[11pt]{amsart}

\usepackage{indentfirst}
\usepackage{bm, mathrsfs, graphics,float,amssymb,amsmath,setspace,graphicx,amsthm,epstopdf,subfigure, enumerate, color, bbm, lineno}
\usepackage{thmtools, thm-restate}
\usepackage[hidelinks]{hyperref}
\usepackage[margin = 0.8in]{geometry}

\def\R{\mathbb{R}}

\def\p{\partial}

\def\vp{\varphi}
\def\E{\mathbb{E}}
\def\dd{\mathrm{d}}
\def\P{\mathbb{P}}

\newcommand{\blue}{}

\newcommand{\mb}{\mathbf}
\newcommand{\sL}{\mathcal{L}}

\providecommand{\bbs}[1]{\left(#1\right)}
\newcommand{\dqp}{\ud \mb q \ud \mb p}
\newcommand{\dqpz}{\ud \mb q \ud \mb p \ud \mb z}

\newcommand{\eps}{\varepsilon}

\newcommand{\ud}{\,\mathrm{d}}
\newcommand{\8}{\infty}

\numberwithin{equation}{section}

\newtheorem{lemma}{Lemma}[section]
\newtheorem{theorem}{Theorem}[section]
\newtheorem{proposition}{Proposition}[section]
\newtheorem{definition}{Definition}[section]

\newtheorem{remark}{Remark}[section]

\begin{document}

\title[Reversibility and linear response theory]{Some properties on the reversibility and the linear response theory of Langevin dynamics}

\author[Y. Gao]{Yuan Gao}
\address{Department of Mathematics, Purdue University, West Lafayette, IN}
\email{gao662@purdue.edu}

\author[J.-G. Liu]{Jian-Guo Liu}
\address{Department of Mathematics and Department of Physics, Duke University, Durham, NC}
\email{jliu@math.duke.edu}

\author[Z. Liu]{Zibu Liu}
\address{Department of Mathematics, Duke University, Durham, NC}
\email{zibu.liu@duke.edu}


\begin{abstract}
Linear response theory is a fundamental framework studying the macroscopic response of a physical system to an external perturbation. This paper focuses on the rigorous mathematical justification of linear response theory for Langevin dynamics. We give some equivalent characterizations for reversible overdamped/underdamped Langevin dynamics, which is the unperturbed reference system. Then we clarify sufficient conditions for the smoothness and exponential convergence to the invariant measure for the overdamped case. We also clarify those sufficient conditions for the underdamped case, which corresponds to hypoellipticity and hypocoercivity. Based on these, the asymptotic dependence of the response function on the small perturbation is proved in both finite and infinite time horizons. As applications, Green-Kubo relations and linear response theory for a generalized Langevin dynamics are also proved in a rigorous fashion.
\end{abstract}

\maketitle

\section{Introduction}
 
Linear response theory is a general framework for studying the behavior of a physical system under small external perturbations. While the microscopic fluctuations in a physical system are usually  complex and  sustained, linear response theory provides a way to characterize or predict the corresponding macroscopic behaviors of the system. For instance, the famous Einstein relation \cite{E1905} for Brownian motion fits within the framework of linear response theory, using the correlation of Brownian particles to predict macroscopic diffusion coefficients.

Linear response theory not only bridges the microscopic first physical principles with macroscopic properties (e.g., electrical, thermal, transport, and mechanical properties), but it also helps in understanding or predicting a nonequilibrium system, which deviates from equilibrium states, using only the information of the original equilibrium system. Thus, it plays a fundamental role in theoretical physics and statistical mechanics.

Given these important applications, linear response theory has been extensively studied, and various versions of response formulas for different physical contexts have been developed, cf. \cite{gallavotti1995dynamical, marconi2008fluctuation, campisi2011colloquium, seifert2012stochastic, dabelow2019irreversibility}.

In this paper, we do not attempt to propose new linear response relations. Instead, we focus on clarifying mathematical conditions and collecting further properties with quantitative estimates for the well-known linear response results for Langevin dynamics. Specifically, in a mathematically rigorous way, we study the asymptotic dependence of the solution to the perturbed Fokker-Planck equation concerning the small external force added to the original reversible Langevin dynamics.


We begin with an overdamped or underdamped Langevin dynamics, initially in a reversible form or with a detailed balanced invariant measure known as the Gibbs measure. Under mild assumptions on the conservative drift, the probability distribution will exponentially converge to the invariant measure as time tends to infinity. To study the nonequilibrium system perturbed by a small non-conservative external force in the long-time regime, the existence, smoothness, and exponential stability of the    perturbed invariant measure are essential preliminaries. These involve semigroup theory, Harris' theorem, and, particularly, hypoellipticity and hypocoercivity for the Fokker-Planck equation in the underdamped case.

In Section \ref{sec2}, we first provide the reversibility conditions in five equivalent forms for overdamped and underdamped Langevin dynamics, and generalized Langevin dynamics.   

In Section \ref{sec_LRT}, we first clarify the existence and uniqueness of positive invariant measures for perturbed irreversible Langevin dynamics. The smoothness and exponential convergence of the solution to the Fokker-Planck equation for the perturbed irreversible Langevin dynamics are also summarized, with proofs provided in the Appendix.
After these preparations, we give a rigorous verification of linear response theory for overdamped Langevin dynamics. We analyze the asymptotic behavior of the response function, which is defined as the difference (in weak form with respect to any observation test function) between the perturbed dynamics of the probability distribution and the initial data given by the unperturbed Gibbs measure $\rho_0(\mb q)$ 
\[
R(t,\eps;\varphi) = \dfrac{1}{\eps} \left( \int_{\R^d} \varphi(\mb q)\rho^{\eps}(\mb q,t)\dd\mb q - \int_{\R^d} \varphi(\mb q)\rho_0(\mb q)\dd\mb q \right).
\]

In Theorem \ref{thm:LRTO}, we prove the convergence of the response function $R(t,\eps;\varphi)$ for either fixed $t$ or fixed $\eps$, as well as the double limits for both $\lim_{\eps \to 0^+} \lim_{t \to +\infty}$ and $\lim_{t \to +\infty} \lim_{\eps \to 0^+}$. Moreover, we also obtain the uniform convergence of the response function for $\eps \leq \eps_0$ in time $t \in [0, +\infty)$. In the special case where {\blue the external perturbation is a conservative force $\nabla W$ for some potential $W$,} we derive the Green-Kubo formula, which connects the long-time behavior of the response function to the autocorrelation function over an infinite time horizon for the original unperturbed dynamics 

\[
\lim\limits_{\eps \to 0^+} \lim\limits_{t \to \infty} R(t,\eps;\mathcal{L}_0 g) = - \int_0^{+\infty} K_{\mathcal{L}_0 g, \mathcal{L}_0 W}(t) \ud t, 
\quad K_{AB}(t) := \int_{\R^d} (e^{t \mathcal{L}_0} A) B \rho_0 \dd\mb q,
\]
where $\mathcal{L}_0$ is the unperturbed generator.


In Section \ref{sec4}, we first summarize key well-posedness results and the hypoellipticity and hypocoercivity of the Fokker-Planck equation in the underdamped case. Villani's seminal work on hypocoercivity \cite{villani2009hypocoercivity} shows that exponential convergence to the invariant measure can still be achieved under mild conditions on the original potential. Given that we consider an external perturbation in a conservative form and with compact support, the perturbed Fokker-Planck equation in the underdamped case will still converge exponentially to the new invariant measure. Based on this, we prove parallel results for the response function $R(t,\eps; \varphi)$ under both fixed $t$ or $\eps$, and for the double limits behavior in Theorem \ref{thm:LRTU}. Lastly, in Section \ref{sec4.2}, as an application of Theorem \ref{thm:LRTU}, we derive the linear response theory for generalized Langevin dynamics with an exponential kernel using the corresponding augmented underdamped Langevin dynamics.


Our results on the linear response theory for stochastic dynamics extend the case where the external force is a general vector field depending on spatial variables. These results are parallel to those in the comprehensive book by \textsc{Pavliotis} \cite{Pavliotis}, which focuses on the case where the external force is independent of spatial variables. On the other hand, the abstract theorem by \textsc{Hairer-Majda} in \cite{hairer2010simple} provides a general methodology for justifying linear response theory for abstract Markov evolutions under several assumptions. Our work falls within the general framework of \cite{hairer2010simple}, but we offer self-contained conditions with detailed verification for Langevin dynamics and prove the uniform convergence of the response function for small external perturbations $\eps \leq \eps_0$ uniformly in time $t \in [0, +\infty)$. For more on linear response theory for deterministic dynamical systems, see \cite{ruelle2009review}.


The remaining sections of this paper are organized as follows:
In Section \ref{sec2}, we provide equivalent conditions for reversibility in overdamped, underdamped, and generalized Langevin dynamics.
In Section \ref{sec_LRT}, we prove the linear response theory and the Green-Kubo relation for overdamped Langevin dynamics.
In Section \ref{sec4}, we prove the linear response theory for underdamped Langevin dynamics and generalized Langevin dynamics.
All omitted proofs for reversibility, hypoellipticity, and hypocoercivity are given in Appendices \ref{sec:proof} and \ref{app:hy}.

\section{Equivalent conditions of reversibility of Langevin dynamics} \label{sec2}

Before studying the linear response theory for Langevin dynamics in the large-time regime, we need to first establish the well-posedness and stability of the invariant measure, which are built on the reversibility of the original unperturbed system. In this section, we provide several equivalent conditions for the reversibility of both overdamped and underdamped Langevin dynamics. Although these conditions are classic, some are less well-known but useful in practice. Therefore, we rigorously clarify these reversibility conditions for the reader's convenience.

\subsection{Definitions and preliminaries}

{\blue
We will study the reversibility of both overdamped, underdamped and generalized Langevin dynamics. In the underdamped and generalized Langevin dynamics case, some physical variables will have even or odd parities. A typical odd variable is velocity $(\mb{p})$, while typical even variables are position $(\mb{q})$, force $(\mb{f})$, and acceleration $(\mb{a})$. To preserve more flexibility, we allow different components of $\mb{x}$ to have different parities.

Another essential difference is the invariant measure (see \eqref{eq:req}) for the underdamped case, which involves both kinetic and potential energy, while the kinetic energy is neglected for the overdamped case. Below, we provide some basic definitions for reversibility.

\begin{definition}[Stationary process]
    A stochastic process $\{\mb{x}(t) \in \R^d, \, t \geq 0\}$ is stationary if for all $\tau, t_1, t_2, ..., t_n \in \R^+$, $n \in \mathbb{Z}^+$, and Borel sets $\mathcal{B}_j \in \R^d$, $j=1, 2, ..., n$, 
    \begin{align} \label{def:stationary}
        \P(\mb{x}(t_j) \in \mathcal{B}_j, 1 \leq j \leq n) = \P(\mb{x}(t_j + \tau) \in \mathcal{B}_j, 1 \leq j \leq n).
    \end{align}
    Taking $n = 1$, we deduce that the law of $\mb{x}(t)$ is invariant.
\end{definition}

\begin{definition}[Time-reversed process]
    Let $\{\mb{x}(t) \in \R^d, \, t \geq 0\}$ be a stochastic process. Fix time $T > 0$. The time-reversed process of $\mb{x}(t)$ (w.r.t. time $T$) is defined as 
    \begin{equation} \label{def:timereversed}
        \mb{x}^*(t) := (\eps_i x_i(T-t))_{1 \leq i \leq n},
    \end{equation}
    where
    \begin{equation*}
        \eps_i :=
        \begin{cases}
            1, & \text{if } x_i \text{ is an even variable}; \\
            -1, & \text{if } x_i \text{ is an odd variable}.
        \end{cases}
    \end{equation*}
\end{definition}
}
\begin{definition}[Reversibility]
    A stationary process $\{X_t \in \R^d, \, t \geq 0\}$ is reversible if for any $T > 0$, the time-reversed process $\{X^*_t, 0 \leq t \leq T\}$ (w.r.t. $T$) has the same finite-dimensional distribution as the original process $X_t$, i.e., for any $t_1, t_2, ..., t_n \in [0, T]$, $n \in \mathbb{Z}^+$, and Borel sets $\mathcal{B}_j \in \R^d, j=1, 2, ..., n$,
    \begin{align} \label{def:reversibility}
        \P(X_{t_j} \in \mathcal{B}_j, 1 \leq j \leq n) = \P(X^*_{t_j} \in \mathcal{B}_j, 1 \leq j \leq n).
    \end{align}
\end{definition}

\subsection{Reversibility of the overdamped Langevin dynamics}

We use $\mb{q}_t$ to indicate the stochastic process at time $t$, and we also use $\mb{q}$ as the displacement vector in $\R^d$, following the convention in mechanics. The overdamped Langevin equation is given by the following stochastic differential equation (SDE) 
\begin{align} \label{eq:OLE}
    \dd \mb{q}_t = \mb{b}(\mb{q}_t) \dd t + \bm{\sigma}(\mb{q}_t) \dd \mb{B}_t, \quad \mb{q}_{t=0} = \mb{q_0},
\end{align}
where $\mb{b} \in C^\infty(\R^d; \R^d), \mb{q}_t \in \R^d$, and $\bm{\sigma} \in C^\infty(\R^d; \R^{d \times d})$ is nonsingular. $\mb{B}_t$ is the $d$-dimensional standard Brownian motion, and $\mb{q}_{t=0} = \mb{q_0}$ is a random variable independent of $\mb{B}_t$. 

To ensure that \eqref{eq:OLE} admits a unique strong solution, we assume that coefficients $\mb{b}(\mb{q})$ and $\bm{\sigma}(\mb{q})$ satisfy certain conditions (see \eqref{eq:llc} and \eqref{eq:gc} below).

The corresponding Fokker-Planck equation and the Fokker-Planck operator $\mathcal{L}^*$ are given by 
\begin{align} \label{eq:FPOLE}
    \frac{\partial \rho(\mb{q}, t)}{\partial t} = (\mathcal{L}^* \rho)(\mb{q}, t), \quad \mathcal{L}^* \rho := -\nabla \cdot (\rho \mb{b}) + \frac{1}{2} \nabla^2 : (\rho \bm{\sigma} \bm{\sigma}^T).
\end{align}
The corresponding Kolmogorov backward equation is 
\begin{align} \label{eq:BOLE}
    \frac{\partial f(\mb{q}, t)}{\partial t} = (\mathcal{L} f)(\mb{q}, t), \quad \text{with generator } \mathcal{L} f := \mb{b} \cdot \nabla f + \frac{1}{2} (\bm{\sigma} \bm{\sigma}^T) : \nabla^2 f.
\end{align}
Define the probability flux $\mb{j}$ as 
\[
    \mb{j}(\rho) := -\frac{1}{2} \nabla \cdot (\rho \bm{\sigma} \bm{\sigma}^T) + \rho \mb{b}.
\]
Then the Fokker-Planck equation can be written as the continuity equation 
\begin{align}
    \frac{\partial \rho}{\partial t} + \nabla \cdot \mb{j} = 0.
\end{align}

According to \eqref{def:reversibility}, we have the following equivalent definition of reversibility.
\begin{lemma} \label{lmm:rev}
    Consider \eqref{eq:OLE} with initial density function $\rho_0(\mb{q})$. Then the system is reversible if and only if, for any $\vp_1, \vp_2 \in C_b^\infty(\R^d)$
    \begin{align} \label{eq:revs}
        \E[\vp_1(\mb{q}_t) \vp_2(\mb{q}_0) \mid \mb{q}_0 \sim \rho_0] = \E[\vp_1(\mb{q}_0) \vp_2(\mb{q}_t) \mid \mb{q}_0 \sim \rho_0].
    \end{align} 
\end{lemma}
(Proof of Lemma \ref{lmm:rev} can be found in Appendix \ref{sec:proof}.) We will use \eqref{eq:revs} to prove reversibility.

\begin{theorem}[Reversibility of the overdamped Langevin dynamics] \label{thm:rev1}
    Consider \eqref{eq:OLE} with initial probability density function $\rho_0(\mb{q}) > 0$. Suppose that $\bm{\sigma}$ is constant and nonsingular. Also, suppose that $\bm{\sigma}$ and $\mb{b}$ satisfy
    \begin{enumerate}[(i)]
        \item (Local Lipschitz continuity) For any $n \in \mathbb{Z}^+$, there exists $K_n > 0$ such that
        \begin{align} \label{eq:llc}
            |\mb{b(x)} - \mb{b(y)}| \leq K_n |\mb{x - y}|,
        \end{align}
        for any $\mb{x, y}$ such that $|\mb{x}|, |\mb{y}| \leq n$;
        \item (Monotone condition) There exists a constant $C > 0$ such that for any $\mb{x} \in \R^d$:
        \begin{align} \label{eq:gc}
            \mb{x}^T \mb{b}(\mb{x}) \leq C(1 + |\mb{x}|^2).
        \end{align}
    \end{enumerate}
    Then the following are equivalent
    \begin{enumerate}[(i)]
        \item (Reversibility) The stochastic process determined by \eqref{eq:OLE} is reversible in the sense of Definition \ref{def:reversibility};
        \item (Symmetry) For arbitrary $\varphi_1, \varphi_2 \in C_b^\infty(\R^d)$,
        \begin{align} \label{eq:rev}
            \int_{\R^d} (\mathcal{L} \varphi_1)(\mb{q}) \varphi_2(\mb{q}) \rho_0(\mb{q}) \dd \mb{q} = \int_{\R^d} (\mathcal{L} \varphi_2)(\mb{q}) \varphi_1(\mb{q}) \rho_0(\mb{q}) \dd \mb{q};
        \end{align}
        \item (Zero flux) The probability flux is zero for all $t \geq 0$, i.e., $\mb{j}(\rho(\cdot, t)) = 0$;
        \item (Potential condition) There exists $U: \R^d \to \R$ such that
        \begin{align}
            -\nabla U(\mb{q}) = 2 (\bm{\sigma} \bm{\sigma}^T)^{-1} \mb{b},
        \end{align}
        and $\rho_0(\mb{q}) = \frac{1}{Z} e^{-U(\mb{q})}$ with $\int_{\R^d} e^{-U(\mb{q})} \dd \mb{q} < \infty$.
    \end{enumerate}
\end{theorem}
{\blue We point out that assumptions \eqref{eq:llc} and \eqref{eq:gc} ensure the well-posedness of the SDE, but they are irrelevant to the reversibility.
We adopt \eqref{eq:gc} instead of the general linear growth condition as it covers more cases. For instance, $\mb{b}(\mb{x}) = (1 - |\mb{x}|^2)\mb{x}$ does not satisfy the linear growth condition but does satisfy \eqref{eq:gc}.}

\subsection{Reversibility of  underdamped Langevin equation}

Now consider the underdamped Langevin equation 
\begin{equation} \label{eq:ULE}
    \left\{
    \begin{split}
        \dd\mb{q} &= \mb{p}\dd t,\\
        \dd\mb{p} &= -\dfrac{1}{2} \bm{\sigma\sigma}^T \mb{p} \dd t + \mb{b}(\mb{q}) \dd t + \frac{\bm{\sigma}}{\sqrt{\beta}} \dd\mb{B}.
    \end{split}
    \right.
\end{equation}
Here $\mb{p}, \mb{q} \in \R^d$, $\bm{\sigma} \in \R^{d \times d}$ is a nonsingular constant matrix, and $\beta = 1/(k_B T)$ is the thermodynamic beta. The corresponding Fokker-Planck equation (also known as the kinetic Fokker-Planck equation) and the Fokker-Planck operator $\mathcal{L}^*$ are 
\begin{align} \label{eq:FPULE}
    \dfrac{\partial \rho(\mb{q}, \mb{p}, t)}{\partial t} = (\mathcal{L}^* \rho)(\mb{q}, \mb{p}, t), \quad \mathcal{L}^* \rho := -\mb{p} \cdot \nabla_{\mb{q}} \rho - \mb{b} \cdot \nabla_{\mb{p}} \rho + \dfrac{1}{2} \nabla_{\mb{p}} \cdot \left( \rho \bm{\sigma\sigma}^T \mb{p} + \frac{1}{\beta} \bm{\sigma\sigma}^T \nabla_{\mb{p}} \rho \right).
\end{align}
The corresponding Kolmogorov backward equation is 
\begin{align} \label{eq:BULE}
    \dfrac{\partial f(\mb{q}, \mb{p}, t)}{\partial t} = (\mathcal{L} f)(\mb{q}, \mb{p}, t), \quad \mathcal{L} f := \mb{p} \cdot \nabla_{\mb{q}} f + \mb{b} \cdot \nabla_{\mb{p}} f - \dfrac{1}{2} (\bm{\sigma\sigma}^T \mb{p}) \cdot \nabla_{\mb{p}} f + \dfrac{1}{2\beta} (\bm{\sigma\sigma}^T) : \nabla_{\mb{p}}^2 f.
\end{align}
We say that \eqref{eq:ULE} satisfies the fluctuation-dissipation theorem {\blue because the dissipative part corresponding to Ornstein–Uhlenbeck process in $\mathcal{L}^*$} can be rewritten as 
\begin{align*}
    \dfrac{1}{2} \nabla_{\mb{p}} \cdot \left( \rho \bm{\sigma\sigma}^T \mb{p} + \frac{1}{\beta} \bm{\sigma\sigma}^T \nabla_{\mb{p}} \rho \right)
    = \dfrac{1}{2} \nabla_{\mb{p}} \cdot \left( \bm{\sigma\sigma}^T \left( \rho \nabla_{\mb{p}} \left( \dfrac{\|\mb{p}\|^2}{2} + \dfrac{\log \rho}{\beta} \right) \right) \right).
\end{align*}

As with the overdamped Langevin equation, we consider an equivalent condition for reversibility using test functions. {\blue Recall that momentum $\mb{p}$ is an odd variable, so the equivalent characterization below also uses parity-reversed test functions  $\widetilde{\varphi}_2(\mb{p}) = \varphi_2(-\mb{p})$ and $\widetilde{\psi}_2(\mb{p}) = \psi_2(-\mb{p})$.}

\begin{lemma} \label{lmm:rev2}
    Consider \eqref{eq:ULE} with the initial density function $\rho_0(\mb{q})$. Then it is reversible if and only if, for any $\vp_1, \vp_2 \in C_0^\infty(\R^d)$,
    \begin{equation} \label{eq:rev2}
        \begin{aligned} 
            &\E\left[\vp_1(\mb{q}(0)) \vp_2(\mb{p}(0)) \psi_1(\mb{q}(t)) \psi_2(\mb{p}(t)) \,\big|\, (\mb{q}(0), \mb{p}(0)) \sim \rho_0 \right] \\
            &\quad = \E\left[\vp_1(\mb{q}(t)) \widetilde{\vp_2}(\mb{p}(t)) \psi_1(\mb{q}(0)) \widetilde{\psi_2}(\mb{p}(0)) \,\big|\, (\mb{q}(0), \mb{p}(0)) \sim \rho_0 \right].
        \end{aligned}
    \end{equation}
\end{lemma}

The proof of Lemma \ref{lmm:rev2} is given in Appendix \ref{sec:proof}.

Next, we provide the equivalence theorem for the reversibility of the underdamped Langevin dynamics.

\begin{theorem} \label{thm:revULE}
    (Reversibility of the underdamped Langevin) Consider \eqref{eq:ULE} with the initial probability density function $\rho_0(\mb{q}, \mb{p}) > 0$. Suppose that $\bm{\sigma}$ is a constant and nonsingular matrix, and that $\bm{\sigma}$ and $\mb{b}$ satisfy \eqref{eq:llc} and \eqref{eq:gc}. Then the following are equivalent 
    \begin{enumerate}[(i)]
        \item (Reversibility) The stochastic process determined by \eqref{eq:ULE} is reversible.
        \item (Symmetry) $\rho_0(\mb{q}, \mb{p}) = \rho_0(\mb{q}, -\mb{p})$ and, for arbitrary $\varphi_1(\mb{q}), \varphi_2(\mb{p}), \psi_1(\mb{q}), \psi_2(\mb{p}) \in C_0^\infty(\R^d)$ 
            \begin{align}\label{eq:thm2_ii}
                \iint \varphi_1(\mb{q}) \varphi_2(\mb{p}) \mathcal{L} \bbs{\psi_1(\mb{q}) \psi_2(\mb{p})} \rho_0(\mb{q}, \mb{p}) \, \dd\mb{q} \, \ud\mb{p} 
                = \iint \mathcal{L} \bbs{\varphi_1(\mb{q}) \widetilde{\varphi}_2(\mb{p})} \psi_1(\mb{q}) \widetilde{\psi}_2(\mb{p}) \rho_0(\mb{q}, \mb{p}) \, \dd\mb{q} \, \ud\mb{p},
            \end{align}
            where $\widetilde{\varphi}_2(\mb{p}) = \varphi_2(-\mb{p})$ and $\widetilde{\psi}_2(\mb{p}) = \psi_2(-\mb{p})$.
        \item (Potential condition) There exists $U: \R^d \to \R$ such that 
            \begin{align}
                -\nabla_{\mb{q}} U(\mb{q}) = \mb{b},
            \end{align}
            and
            \begin{equation} \label{eq:req}
                \rho_0(\mb{q}, \mb{p}) = \dfrac{1}{Z} e^{-\beta H(\mb{q}, \mb{p})}, \quad Z = \iint e^{-\beta H(\mb{q}, \mb{p})} \dqp < \infty,
            \end{equation}
            where 
            \begin{align}
                H(\mb{q}, \mb{p}) := \dfrac{|\mb{p}|^2}{2} + U(\mb{q}).
            \end{align}
        \item (Evenness in $\mb{p}$ variable) $\rho_0(\mb{q}, \mb{p}) = \rho_0(\mb{q}, -\mb{p})$ and $\rho_0$ is stationary, i.e., it solves $\mathcal{L}^* \rho_0 = 0$.
        \item (Separation of variables) $\rho_0(\mb{q}, \mb{p}) = U_1(\mb{q}) U_2(\mb{p})$, where $U_1 > 0$ and $U_2 > 0$, and $\rho_0$ is stationary, i.e., it solves $\mathcal{L}^* \rho_0 = 0$.
    \end{enumerate}
\end{theorem}
We remark here that the conclusions of Theorem \ref{thm:rev1} and Theorem \ref{thm:revULE} also hold for non-constant $\bm{\sigma}$. However, in this case, the SDE should be written in the backward Itô's integral sense \cite{kunitha1982backward}. For the overdamped Langevin equation 
\begin{align*}
    \dd \mb{q} = \mb{b}(\mb{q}) \dd t + \bm{\sigma}(\mb{q}) \hat{\dd} \mb{B},
\end{align*}
the Fokker-Planck equation is given by 
\begin{align*}
    \dfrac{\p \rho(\mb{q}, t)}{\p t} = (\mathcal{L}^* \rho)(\mb{q}, t), \quad \mathcal{L}^* \rho := \nabla \cdot \left( \dfrac{1}{2} \bm{\sigma} \bm{\sigma}^T \nabla \rho - \rho \mb{b} \right).
\end{align*}

For the underdamped Langevin equation, the SDE should be written as 
\begin{equation*}
    \left\{
    \begin{split}
        \dd \mb{q} &= \mb{p} \dd t, \\
        \dd \mb{p} &= -\dfrac{1}{2} \bm{\sigma \sigma}^T \mb{p} \dd t + \mb{b}(\mb{q}) \dd t + \frac{\bm{\sigma}}{\sqrt{\beta}} \hat{\dd} \mb{B},
    \end{split}
    \right.
\end{equation*}
with the Fokker-Planck equation 
\begin{align*}
    \dfrac{\p \rho(\mb{q}, \mb{p}, t)}{\p t} = (\mathcal{L}^* \rho)(\mb{q}, \mb{p}, t), \quad \mathcal{L}^* \rho := -\mb{p} \cdot \nabla_{\mb{q}} \rho - \mb{b} \cdot \nabla_{\mb{p}} \rho + \dfrac{1}{2} \nabla_{\mb{p}} \cdot \left( \rho \bm{\sigma\sigma}^T \mb{p} + \frac{1}{\beta} \bm{\sigma\sigma}^T \nabla_{\mb{p}} \rho \right).
\end{align*}

\subsection{The reversibility of generalized Langevin dynamics with memory}

Let $n \geq 1$ be a fixed integer. Now consider the following generalized Langevin equation with memory 
\begin{equation} \label{eq:GLE}
    \left\{
    \begin{split}
        \ddot{\mb{q}}(t) &= \mb{b}(\mb{q}) - \sum_{i=1}^n \mb{A}_i \mb{A}_i^T \int_0^t e^{-\alpha_i (t-s)} \dot{\mb{q}}(s) \dd s + \mb{A}_i \mb{f}_i(t), \\
        \dd \mb{f}_i(t) &= -\alpha_i \mb{f}_i \dd t + \sqrt{2 \beta^{-1} \alpha_i} \dd \mb{B}_i, \quad i = 1, 2, \dots, n.
    \end{split}
    \right.
\end{equation}
Here, for each $i = 1, 2, \dots, n$, $\mb{A}_i \in \R^{d \times d}$ is a constant matrix, $\alpha_i > 0$ is a constant, and $\mb{B}_i(t) \in \R^d$ is a standard Brownian motion. Additionally, $\mb{B}_i(t)$ for $i = 1, 2, \dots, n$ are independent. Reversibility can also be considered for \eqref{eq:GLE} if we reformulate it.

Let $\mb{z}_i(t) = -\mb{A}_i^T \int_0^t e^{-\alpha_i (t-s)} \dot{\mb{q}}(s) \dd s + \mb{f}_i(t)$ and $\dot{\mb{q}} = \mb{p}$. Then \eqref{eq:GLE} can be reformulated as 
\begin{equation} \label{eq:RGLE}
    \left\{
    \begin{split}
        \dd \mb{q} &= \mb{p} \dd t, \\
        \dd \mb{p} &= \left( \mb{b}(\mb{q}) + \sum_{i=1}^n \mb{A}_i \mb{z}_i \right) \dd t, \\
        \dd \mb{z}_i &= -(\alpha_i \mb{z}_i + \mb{A}_i^T \mb{p}) \dd t + \sqrt{2 \beta^{-1} \alpha_i} \dd \mb{B}_i, \quad i = 1, 2, \dots, n.
    \end{split}
    \right.
\end{equation}
We will consider \eqref{eq:RGLE} instead of \eqref{eq:GLE} for reversibility. The corresponding Fokker-Planck equation for \eqref{eq:RGLE} is given by 
\begin{equation} \label{eq:FGLE}
    \begin{aligned}
        \dfrac{\p \rho}{\p t} = \mathcal{L}^* \rho, \quad \mathcal{L}^* \rho := &-\mb{p} \cdot \nabla_{\mb{q}} \rho - \mb{b}(\mb{q}) \cdot \nabla_{\mb{p}} \rho - \sum_{i=1}^n \left( (\mb{A}_i \mb{z}_i) \cdot \nabla_{\mb{p}} \rho - \mb{A}_i^T \mb{p} \cdot \nabla_{\mb{z}_i} \rho \right) \\
        &+ \sum_{i=1}^n \alpha_i \nabla_{\mb{z}_i} \cdot \left( \mb{z}_i \rho + \dfrac{1}{\beta} \nabla_{\mb{z}_i} \rho \right).
    \end{aligned}
\end{equation}
The corresponding backward equation is given by 
\begin{equation} \label{eq:BGLE}
    \begin{aligned}
        \dfrac{\p f}{\p t} = \mathcal{L} f, \quad \mathcal{L} f := &\mb{p} \cdot \nabla_{\mb{q}} f + \mb{b}(\mb{q}) \cdot \nabla_{\mb{p}} f + \sum_{i=1}^n \left( (\mb{A}_i \mb{z}_i) \cdot \nabla_{\mb{p}} f - \mb{A}_i^T \mb{p} \cdot \nabla_{\mb{z}_i} f \right) \\
        &+ \sum_{i=1}^n \alpha_i \left( -\mb{z}_i \cdot \nabla_{\mb{z}_i} f + \dfrac{1}{\beta} \Delta_{\mb{z}_i} f \right).
    \end{aligned}
\end{equation}

We now state the equivalent theorem characterizing the reversibility of \eqref{eq:RGLE} in the case {\blue where $n = 1$ and $\mb{A}_1 = \mb{I} \in \R^{d \times d}$, i.e., 
\begin{equation} \label{eq:1GLE} 
\left\{\begin{split}
\dd{\mb{q}} &= \mb{p} \dd t, \\
\dd{\mb{p}} &= \left( -\nabla_{\mb{q}} V(\mb{q}) + \mb{z} \right) \dd t, \\
\dd{\mb{z}} &= -(\alpha \mb{z} + \mb{p}) \dd t + \sqrt{2 \beta^{-1} \alpha} \dd \mb{B}.
\end{split}\right.
\end{equation}
}
The proof is given in Appendix \ref{sec:proof}.

\begin{theorem} \label{thm:moR}
    (Reversibility of the generalized Langevin with memory) Consider \eqref{eq:RGLE} with initial probability density function $\rho_0(\mb{q}, \mb{p}, \mb{z}) > 0$ satisfying $\rho_0 > 0$. Suppose that $\bm{\sigma}$ is a constant and nonsingular matrix. Then the following are equivalent
    \begin{enumerate}[(i)]
        \item (Reversibility) The stochastic process determined by \eqref{eq:RGLE} is reversible.
        \item (Symmetry) $\rho_0(\mb{q}, \mb{p}, \mb{z}) = \rho_0(\mb{q}, -\mb{p}, \mb{z})$, and for arbitrary $\varphi_1(\mb{q}), \varphi_2(\mb{p}), \varphi_3(\mb{z}), \psi_1(\mb{q}), \psi_2(\mb{p}), \psi_3(\mb{z}) \in C_0^\infty(\R^d)$
        \begin{equation} \label{symLm}
            \begin{aligned}
                &\iiint \varphi_1(\mb{q}) \varphi_2(\mb{p}) \varphi_3(\mb{z}) \mathcal{L} \left( \psi_1(\mb{q}) \psi_2(\mb{p}) \psi_3(\mb{z}) \right) \rho_0(\mb{q}, \mb{p}, \mb{z}) \dqpz \\
                &= \iiint \mathcal{L} \left( \varphi_1(\mb{q}) \widetilde{\varphi}_2(\mb{p}) \varphi_3(\mb{z}) \right) \psi_1(\mb{q}) \widetilde{\psi}_2(\mb{p}) \psi_3(\mb{z}) \rho_0(\mb{q}, \mb{p}, \mb{z}) \dqpz,
            \end{aligned}
        \end{equation}
        where $\widetilde{\varphi}_2(\mb{p}) = \varphi_2(-\mb{p})$ and $\widetilde{\psi}_2(\mb{p}) = \psi_2(-\mb{p})$.
        \item (Potential condition) There exists $U: \R^d \to \R$ such that
        \begin{align}
            -\nabla_{\mb{q}} U(\mb{q}) = \mb{b},
        \end{align}
        and
        \begin{equation}
            \rho_0(\mb{q}, \mb{p}, \mb{z}) = \dfrac{1}{Z} e^{-\beta H(\mb{q}, \mb{p}, \mb{z})}, \quad Z = \iiint e^{-\beta H(\mb{q}, \mb{p}, \mb{z})} \dqpz < \infty,
        \end{equation}
        where
        \begin{align}
            H(\mb{q}, \mb{p}, \mb{z}) := U(\mb{q}) + \dfrac{|\mb{p}|^2}{2} + \dfrac{|\mb{z}|^2}{2}.
        \end{align}
        \item (Evenness in $\mb{p}$ variable) $\rho_0(\mb{q}, \mb{p}, \mb{z}) = \rho_0(\mb{q}, -\mb{p}, \mb{z})$ and $\rho_0$ is stationary, i.e., it solves $\mathcal{L}^* \rho_0 = 0$.
        \item (Separation of variables) $\rho_0(\mb{q}, \mb{p}, \mb{z}) = U_1(\mb{q}) U_2(\mb{p}) U_3(\mb{z})$, where $U_1, U_2, U_3 > 0$, and $\rho_0$ is stationary, i.e., it solves $\mathcal{L}^* \rho_0 = 0$.
    \end{enumerate}
\end{theorem}

\section{Rigorous verification of linear response theory (LRT) and the Green-Kubo formula for the overdamped Langevin}
\label{sec_LRT}

{\blue In this section, we rigorously verify the linear response theory (LRT) and the Green-Kubo relation for overdamped reversible Langevin dynamics with constant diffusion coefficients.   The linear response theory aims to study the asymptotic dependence of the solution $\rho^\eps$ of the perturbed Fokker-Planck equation \eqref{eq:FPpOL} in terms of the small external force. Recall the Gibbs measure $\rho_0(\mb{q})$ for the original overdamped Langevin dynamics. In the weak formulation, for any $\varphi \in C_c^{\infty}(\R^d)$, we define the response function as follows:
\begin{align} \label{def:delta}
    R(t,\eps;\varphi) := \dfrac{1}{\eps} \left( \int_{\R^d} \varphi(\mb{q}) \rho^{\eps}(\mb{q},t) \dd\mb{q} - \int_{\R^d} \varphi(\mb{q}) \rho_0(\mb{q}) \dd\mb{q} \right).
\end{align}

The term $\eps R(t, \eps; \varphi)$ represents the leading-order (i.e., $O(\eps)$) change of an observable $\varphi$ at time $t$ under the external perturbation $\eps \mb{M}$. Mathematically, LRT focuses on the behavior of $R(t, \eps; \varphi)$ as $\eps \to 0$ and $t \to \infty$. In Theorem \ref{thm:LRTO}, we will prove the convergence of $R(t, \eps; \varphi)$ for fixed $t$ or fixed $\eps$, as well as the double limits for both $\lim_{\eps \to 0^+} \lim_{t \to +\infty}$ and $\lim_{t \to +\infty} \lim_{\eps \to 0^+}$. The main conclusion is the convergence of the response function $R(t, \eps; \varphi)$ in terms of the small parameter for the external perturbation $\eps \leq \eps_0$, uniformly for $t \in [0, +\infty)$. In Section \ref{sec3.1}, we first give some preparations on the perturbed invariant measure including the hypoellipticity and the exponential convergence of the Fokker-Planck equation. Then we study the linear response theory for general external force in Section \ref{sec_LRT_over} and the Green-Kubo relation for conservative external force in Section \ref{sec3.2}.
}

\subsection{Invariant measure and the exponential convergence for irreversible perturbation}\label{sec3.1}
{\blue In this subsection, we consider an irreversible perturbation in a form of general external force $\eps \mb{M}$.  Before studying the linear response theory with respect to the external force $\eps \mb{M}$ in Section \ref{sec_LRT_over}, we first prepare some preliminary results, including estimates on the perturbed invariant measure and the well-posedness, hypoellipticity and exponential convergence of the Fokker-Planck equation.

Consider the following reversible overdamped Langevin system at equilibrium 
\begin{align} \label{eq:OL1}
    \dd \mb{q}_t = -\bm{\sigma \sigma}^T \nabla V(\mb{q}) \dd t + \sqrt{2} \bm{\sigma} \dd \mb{B}, \quad \mb{q}_{t=0} \sim \rho_0(\mb{q}) := \dfrac{1}{Z} e^{-V(\mb{q})}, \quad Z := \int_{\R^d} e^{-V(\mb{q})} \dd \mb{q}.
\end{align}
Here, $\bm{\sigma} \in \R^{d \times d}$ is constant and nonsingular.

Let $\mb{M} \in C_c^{\infty}(\R^d; \R^d)$. For any $\eps > 0$, suppose that at time $t=0$, an external force $\eps \mb{M}$ is added to the system, which yields the following perturbed SDE 
\begin{align} \label{eq:pertOL}
    \dd \mb{q}_t^{\eps} = \bm{\sigma \sigma}^T (-\nabla V(\mb{q}^{\eps}) + \eps \mb{M}(\mb{q}^{\eps})) \dd t + \sqrt{2} \bm{\sigma} \dd \mb{B},
\end{align}
with the initial distribution $
    \mb{q}_{t=0}^{\eps} \sim \rho_0.$
Notice that the initial data for the perturbed SDE is taken as the equilibrium $\rho_0$ for the unperturbed SDE.} 

\subsubsection{Invariant measure: existence, uniqueness, and positivity}

In this subsection, we clarify some known results on the existence, uniqueness, and positivity of the invariant measure, i.e., the stationary solution to the corresponding Fokker-Planck equation. For simplicity in notation related to the Fokker-Planck equation, we will use $2\bm{\sigma} \bm{\sigma}^T$ as the variance.

{\blue 
From this point on, we impose the following assumptions on the potential $V(\mb{x})$

\smallskip
\textbf{Assumption (I)} There exists $\alpha > 0$ such that
\begin{align} \label{ass:polytrap}
    |\nabla V(\mb{x})| \leq C_1 |\mb{x}|^{\alpha} + C_2, \quad \limsup_{|\mb{x}| \to \infty} \dfrac{-\mb{x} \cdot \nabla V(\mb{x})}{|\mb{x}|^{\alpha + 1}} =: \gamma_1 < 0.
\end{align}

\textbf{Assumption (II)} There exists $\lambda > 0$ such that
\begin{align} \label{ass:Poincare}
    \liminf_{|\mb{x}| \to \infty} \left( |\bm{\sigma}^T \nabla V(\mb{x})|^2 - 2 \mathrm{tr}(\bm{\sigma} \bm{\sigma}^T \nabla^2 V(\mb{x})) \right) =: \lambda > 0.
\end{align}

Assumption (I) ensures the uniqueness of the invariant measure and provides a decay estimate at infinity \cite[Theorem 3.4.3]{bogachev2015fokker}. According to \cite{dolbeault2015hypocoercivity}, Assumption (II) ensures that the measure $e^{-V(\mb{x})}$ satisfies Poincaré's inequality.
}

Notice that  \eqref{eq:pertOL} is irreversible, and the Gibbs measure $\rho_0(\mb{q})$ is no longer the invariant measure. In fact, even the existence of an invariant measure for an irreversible system is not trivial. 

{\blue 
Under Assumption (I), the potential is a trapping potential with superlinear growth, which guarantees that both equations \eqref{eq:OL1} and \eqref{eq:pertOL} admit a unique invariant measure with density functions $\rho_0$ and $\rho_\infty^{\varepsilon}$, respectively. In Lemma \ref{lmm:uniqueness}, we establish the existence, uniqueness, and far-field decay estimates for the invariant measure under Assumption (I). Furthermore, if exponential convergence to the invariant measure over long times is considered, we will also require Assumption (II). In Proposition \ref{prop:expTV}, we recall the result of exponential convergence for the invariant measure $\rho_{\infty}^\eps$ in $L^1(\mathbb{R}^d)$, which is crucial for obtaining the linear response theory.

For the underdamped case, we restrict our study to perturbations in a potential form. Assumptions (I) and (II) are also essential for the well-posedness of the invariant measure and the exponential convergence of the perturbed Fokker-Planck equation in $L^2(1/\rho_{\infty}^\eps)$; see Section \ref{sec4}. 
}

\begin{lemma} \label{lmm:uniqueness}
    (Existence, uniqueness, and positivity of the invariant measure) Let $V(\mb{q}) \in C^\infty(\R^d; \R)$ satisfy Assumption (I), and let $\mb{M} \in C_c^{\infty}(\R^d; \R^d)$. For all $\varepsilon \in [0,1)$, consider the SDE \eqref{eq:pertOL}. Then it admits a unique invariant measure with density $\rho_\infty^{\varepsilon}$ that satisfies 
    \begin{align} \label{eq:est}
        e^{-K_1 (|\mb{q}|^{\alpha+1} + 1)} \leq \rho_\infty^{\varepsilon} \leq e^{-K_2 (|\mb{q}|^{\alpha+1} + 1)},
    \end{align}
    where constants $K_1, K_2 > 0$ are uniform in $\varepsilon$.
\end{lemma}

The existence and estimate \eqref{eq:est} can be found in \cite[Theorem 3.4.3]{bogachev2015fokker}, and uniqueness is ensured by the following lemma from \cite[Theorem 4.1.6]{bogachev2015fokker}.

\begin{lemma}[Theorem 4.1.6 in \cite{bogachev2015fokker}] \label{prop:uniqueness}
    Suppose that 
    \begin{align} \label{ass:smoothness}
        \mb{b} \in C^{\infty}(\R^d; \R^d), \quad \bm{\sigma} \in C^{\infty}(\R^d; \R^{d \times d}), \quad \bm{\sigma} \bm{\sigma}^T \text{ is strictly elliptic}.
    \end{align}
    Let $\rho \in C^{\infty}(\R^d)$ solve $\mathcal{L}^* \rho = 0$, where $\mathcal{L}^*$ is defined in \eqref{eq:FPOLE}. Assume that $\rho > 0$ and $\rho \in L^1(\R^d)$. If
    \begin{align} \label{eq:int}
        \dfrac{\|\bm{\sigma \sigma}^T\|}{1 + \|\mb{x}\|^2} \in L^1(\rho \dd \mb{x}), \quad \dfrac{\|\mb{b}\|}{1 + \|\mb{x}\|} \in L^1(\rho \dd \mb{x}) 
    \end{align}
    holds, then $\rho$ is the unique non-zero solution of $\mathcal{L}^* u = 0$ such that $u \in L^1(\R^d)$ and $u \geq 0$.
\end{lemma}

\subsubsection{Well-posedness of the Fokker-Planck equation}
Suppose the density of $\mb{q}_t^{\eps}$ is given by $\rho^{\eps}(\mb{q}, t)$. Then $\rho^{\eps}(\mb{q}, t)$ satisfies the following Fokker-Planck equation 
\begin{align} \label{eq:FPpOL}
    \dfrac{\partial \rho^{\eps}(\mb{q}, t)}{\partial t} = \mathcal{L}^*_{\eps} \rho^{\eps}(\mb{q}, t) := \nabla \cdot \left( \bm{\sigma \sigma}^T \left( (\nabla V(\mb{q}) - \eps \mb{M}(\mb{q})) \rho^{\eps}(\mb{q}, t) + \nabla \rho^{\eps}(\mb{q}, t) \right) \right), 
    \quad \rho^{\eps}(\mb{q}, 0) = \rho_0(\mb{q}).
\end{align} 

In fact, $\mathcal{L}^*_{\eps}$ generates a strongly continuous semigroup of contractions, which ensures the well-posedness of \eqref{eq:FPOLE}.

\begin{lemma} \label{prop:contraction}
    Let $V(\mb{q}) \in C^\infty(\R^d; \R)$ satisfy Assumption (I), and let $\mb{M} \in C_c^{\infty}(\R^d; \R^d)$. Consider \eqref{eq:FPpOL}. Then 
    \begin{enumerate}[(i)]
        \item $\mathcal{L}^*_{\eps}$ in \eqref{eq:FPpOL} generates a strongly continuous semigroup of contractions in $L^2(1/\rho_\infty^{\eps})$;
        \item \eqref{eq:FPpOL} admits a unique solution $\rho^{\varepsilon}(\mb{q}, t) \in C^1([0, T], L^2(1/\rho_\infty^\varepsilon))$ for any $T > 0$.
    \end{enumerate}
\end{lemma}
The proof of Lemma \ref{prop:contraction} is provided in Appendix \ref{sec:proof}. We denote the semigroup generated by $\mathcal{L}^*_{\eps}$ as $\left\{e^{t \mathcal{L}^*_{\eps}}\right\}_{t \geq 0}$.

\subsubsection{Hypoellipticity}

Hypoellipticity implies the smoothness of the solution to the Fokker-Planck equation, allowing us to perform integration by parts in subsequent proofs without concern. We prove hypoellipticity by applying H\"ormander's celebrated result \cite[Theorem 1.1]{hormander1967hypoelliptic}.

Recall that a linear differential operator $P$ with $C^{\infty}$ coefficients in $\R^d$ (or an open subset of $\R^d$) is called hypoelliptic if for every distribution $u$ in $\mathcal D'(\R^d)$, we have 
\begin{align*}
    \mathrm{sing} \, \mathrm{supp} \, u = \mathrm{sing} \, \mathrm{supp} \, Pu.
\end{align*}  
Here, $\mathrm{sing} \, \mathrm{supp} \, u$ denotes the singular support of $u$ 
\begin{align*}
    \mathrm{sing} \, \mathrm{supp} \, u = \R^d \setminus \left\{ x \in \R^d : u \, \text{is smooth near } x \right\}.
\end{align*}

Consider a linear differential operator $P$ with $C^{\infty}$ coefficients, which can be written as:
\begin{align*}
    P = \sum_{j=1}^r \mb{X}_j^2 + \mb{X}_0 + c,
\end{align*}
where $\mb{X}_0, ..., \mb{X}_r$ denote first-order homogeneous differential operators in $\R^d \times (0, \infty)$ with smooth coefficients, and $c \in C^{\infty}(\R^d \times (0, \infty))$. For example, the heat operator $\Delta - \dfrac{\partial}{\partial t}$ in $\R^d \times (0, \infty)$ can be recast as
\begin{align*}
    \Delta - \dfrac{\partial}{\partial t} = \sum_{i=1}^d (\mb{X}_i)^2 + \mb{X}_0,
\end{align*}
where 
\begin{align*}
    \mb{X}_i = \dfrac{\partial}{\partial x_i}, \quad i = 1, 2, ..., d, \quad \mb{X}_0 = -\dfrac{\partial}{\partial t}.
\end{align*}

H\"ormander’s theorem \cite[Theorem 1.1]{hormander1967hypoelliptic} relates the Lie algebra generated by $\mb{X}_i, \, i = 0, 1, 2, ..., r$ to the hypoellipticity of $P$.

\begin{theorem}[Hypoellipticity, Theorem 1.1 in \cite{hormander1967hypoelliptic}] \label{thm:hypoellipticity}
    Consider 
    \begin{align*}
        P = \sum_{j=1}^r \mb{X}_j^2 + \mb{X}_0 + c,
    \end{align*}
    where $\mb{X}_0, ..., \mb{X}_r$ denote first-order homogeneous differential operators in $\R^d \times (0, \infty)$ with smooth coefficients, and $c \in C^{\infty}(\R^d \times (0, \infty))$. If, at any given point $(\mb{q}, t) \in \R^d \times (0, \infty)$,
    \begin{align*}
        \mathrm{span} \left\{ \mb{X}_{j_1}, [\mb{X}_{j_1}, \mb{X}_{j_2}], [\mb{X}_{j_1}, [\mb{X}_{j_2}, \mb{X}_{j_3}]], ..., [\mb{X}_{j_1}, [\mb{X}_{j_2}, [\mb{X}_{j_3}, ..., \mb{X}_{j_k}]]] \right\} = \R^{d+1},
    \end{align*}
    where $j_i = 0, 1, 2, ..., r$, then $P$ is hypoelliptic.
\end{theorem}

Using Theorem \ref{thm:hypoellipticity}, we obtain the smoothness of $\rho^{\varepsilon}(\mb{q}, t)$.

\begin{lemma} \label{lmm:smoothOL}
    Let $\rho^\eps(\mb{q}, t)$ be the unique solution of \eqref{eq:FPpOL}. Then it is smooth in $\R^d \times (0, \infty)$.
\end{lemma}

The hypoellipticity for the Fokker-Planck equation corresponding to overdamped Langevin dynamics is trivial. However, we include the proof using H\"ormander’s hypoellipticity theorem in Appendix \ref{app:hy} to facilitate comparison with the underdamped case.

\subsubsection{Exponential convergence}

The convergence of $\rho^{\eps}(\mb{q},t)$ to $\rho_\infty^\eps(\mb{q})$ in $L^2(1/\rho_\infty^{\eps})$ cannot be derived by Poincare's inequality because \eqref{eq:pertOL} is irreversible. However, we still have exponential convergence in total variation, as shown in \cite[Theorem B]{ji2019convergence}.

\begin{proposition}[Convergence in $L^1(\R^d)$] \label{prop:expTV}
    Suppose Assumptions (I) and (II) hold. For any $\eps \in (0, 1)$, consider \eqref{eq:FPpOL}. Then there exist constants $\eps_0, C > 0$, and $r > 0$ that depend only on $V, \mb{M}$, and $\bm{\sigma}$, such that for all $\eps \in [0, \eps_0)$, we have 
    \begin{align*}
        \|\rho^{\eps}(\cdot,t) - \rho_\infty^\eps\|_{L^1(\R^d)} \leq Ce^{-rt}.
    \end{align*}
\end{proposition}

The proof of Proposition \ref{prop:expTV} employs a version of Harris's theorem by Hairer and Majda \cite{hairer2010convergence}. For completeness, the proof is provided in Appendix \ref{app:hy}.

\subsection{Linear response theory (LRT)}\label{sec_LRT_over}

Recall \eqref{def:delta} and perturbed SDE \eqref{eq:pertOL}. We will study the behaviors of the response function \eqref{def:delta}.
Let $\widetilde{\rho^{\eps}}(\mb{q},t) = \rho^{\eps}(\mb{q},t) - \rho_0$. Then
\begin{align*}
    R(t, \eps; \varphi) := \dfrac{1}{\eps} \int_{\R^d} \varphi(\mb{q}) \widetilde{\rho^{\eps}}(\mb{q}, t) \dd\mb{q}.
\end{align*}
The function $\widetilde{\rho^{\eps}}(\mb{q}, t)$ is smooth and satisfies 
\begin{align} \label{eq:FPp}
    \dfrac{\partial \widetilde{\rho^{\eps}}(\mb{q}, t)}{\partial t} = \mathcal{L}^*_{\eps} \widetilde{\rho^{\eps}}(\mb{q}, t) + \eps \nabla \cdot (\rho_0 \bm{\sigma} \bm{\sigma}^T \mb{M}), \quad \widetilde{\rho^{\eps}}(\mb{q}, 0) = 0.
\end{align}

Using the dissipative property of the semigroup, we can derive the following estimate for $\widetilde{\rho^{\eps}}(\mb{q}, t)$.

\begin{lemma} \label{lmm:eps}
    Consider $\widetilde{\rho^{\eps}}(\mb{q}, t)$ in \eqref{eq:FPp}. Then 
    \begin{enumerate}[(i)]
        \item There exists a constant $C > 0$, which depends only on $\mb{M}$ and $V$, such that 
        \begin{align} \label{eq:L2est}
            \|\widetilde{\rho^{\eps}}(\cdot, t)\|_{L^1} \leq C \eps,
        \end{align}
        for all $t \geq 0$.
        
        \item There exists a constant $C > 0$, which depends only on $\mb{M}$ and $V$, such that 
        \begin{align} \label{eq:L2est_L2}
            \|\widetilde{\rho^{\eps}}(\cdot, t)\|_{L^2(1/\rho_\infty^\eps)} \leq C \eps t,
        \end{align}
        for all $t \geq 0$.
            Furthermore, if $\mb{M} = \nabla W$ is in gradient form, we also have 
        \begin{align} \label{xxxx}
            \|\widetilde{\rho^{\eps}}(\cdot, t)\|_{L^2(1/\rho_\infty^\eps)} \leq C \eps.
        \end{align}
    \end{enumerate}
\end{lemma}

\begin{proof}
    Since $\mb{M} \in C_c^{\infty}(\R^d; \R^d)$, by Duhamel's principle, we have 
    \begin{align}
        \widetilde{\rho^{\eps}}(\mb{q}, t) = \eps \int_0^t e^{(t-s) \mathcal{L}_{\eps}^*} \left( \nabla \cdot (\rho_0 \bm{\sigma} \bm{\sigma}^T \mb{M}) \right) \dd s.
    \end{align}
    Here, $e^{s \mathcal{L}_{\eps}^*} \varphi$ represents the solution to \eqref{eq:FPpOL} at time $s$ with initial value $\varphi$.

    First, since $\mu := \nabla \cdot (\rho_0 \bm{\sigma} \bm{\sigma}^T \mb{M})$ has zero Lebesgue integral, the positive part $\mu^+$ and negative part $\mu^-$ satisfy $\mu = \mu^+ + \mu^-$, and $\int \mu^+(x) \dd x = \int \mu^-(x) \dd x =: c_0$.
    Therefore, the exponential convergence in $L^1$ norm from Proposition \ref{prop:expTV} implies 
    \begin{align*}
        \eps \int_0^t & \int |e^{(t-s) \mathcal{L}_{\eps}^*} (\nabla \cdot (\rho_0 \bm{\sigma} \bm{\sigma}^T \mb{M}))| \dd x \dd s \\
        &= \eps \int_0^t \int |e^{(t-s) \mathcal{L}_{\eps}^*} (\mu^+ - c_0 \rho^\eps_\infty + c_0 \rho^\eps_\infty - \mu^-)| \dd x \dd s \\
        &= \eps \int_0^t \int |e^{(t-s) \mathcal{L}_{\eps}^*} \mu^+ - c_0 \rho^\eps_\infty| \dd x \dd s + \eps \int_0^t \int |e^{(t-s) \mathcal{L}_{\eps}^*} \mu^- - c_0 \rho^\eps_\infty| \dd x \dd s \\
        &\leq c \eps \int_0^t e^{-r(t-s)} \dd s = \frac{c \eps}{r} (1 - e^{-rt}).
    \end{align*}

    Second, since $\nabla \cdot (\rho_0 \mb{M}) \in D(\mathcal{L}^*_{\eps})$, by the contraction property of $\mathcal{L}^*_{\eps}$ from Lemma \ref{prop:contraction}, we have 
    \begin{align*}
        \|\widetilde{\rho^{\eps}}(\cdot, t)\|_{L^2(1/\rho_\infty^\eps)} = \eps \int_0^t \|e^{(t-s) \mathcal{L}_{\eps}^*} (\nabla \cdot (\rho_0 \bm{\sigma} \bm{\sigma}^T \mb{M}))\|_{L^2(1/\rho_\infty^\eps)} \dd s \leq \eps t \|\nabla \cdot (\rho_0 \bm{\sigma} \bm{\sigma}^T \mb{M})\|_{L^2(1/\rho_\infty^\eps)}.
    \end{align*}
        From \eqref{eq:est}, we know that $\|\nabla \cdot (\rho_0 \bm{\sigma} \bm{\sigma}^T \mb{M})\|_{L^2(1/\rho_\infty^\eps)}$ can be uniformly bounded since $\mb{M}$ is compactly supported. This proves \eqref{eq:L2est_L2}.

    Furthermore, the exponential convergence in $L^2(1/\rho^\eps_\infty)$ norm implies 
    \begin{align*}
        \eps \int_0^t & \int |e^{(t-s) \mathcal{L}_{\eps}^*} (\nabla \cdot (\rho_0 \bm{\sigma} \bm{\sigma}^T \mb{M}))|^2 / \rho^\eps_\infty \dd x \dd s \\
        & = \eps\int_0^t\int  |e^{(t-s)\mathcal{L}_{\eps}^*}(\mu^+- c_0\rho^\eps_\8+ c_0\rho^\eps_\8 -\mu^-) |^2 /\rho^\eps_{\8}\ud x\dd s \\
        &\leq 2 \eps \int_0^t \int |e^{(t-s) \mathcal{L}_{\eps}^*} \mu^+ - c_0 \rho^\eps_\infty|^2 / \rho^\eps_\infty \dd x \dd s + 2 \eps \int_0^t \int |e^{(t-s) \mathcal{L}_{\eps}^*} \mu^- - c_0 \rho^\eps_\infty|^2 / \rho^\eps_\infty \dd x \dd s \\
        &\leq c \eps \int_0^t e^{-r(t-s)} \dd s = \frac{c \eps}{r} (1 - e^{-rt}).
    \end{align*}
\end{proof}

Denote the Fokker-Planck operator for the reversible part as
\begin{equation}\label{Lsys}
\mathcal{L}_0^* \rho:= \nabla\cdot(\bm{\sigma\sigma}^T(\rho \nabla V+\nabla \rho)) = \nabla \cdot (\rho_0 \bm{\sigma\sigma}^T \nabla \frac{\rho}{\rho_0} ).
\end{equation} 
Then Fokker-Planck equation \eqref{eq:FPpOL} which could be reformulated as
	\begin{align*}
		\dfrac{\p \rho^{\eps}(\mb q,t)}{\p t} = \mathcal{L^*_{\eps}}\rho^{\eps}(\mb q,t) = \mathcal{L}_0^*\rho^{\eps}(\mb q,t) -\eps\nabla\cdot(\rho^{\eps}\bm\sigma\bm\sigma^T\mb M).
	\end{align*}

We now consider the limit behavior of $R(t, \eps; \varphi)$.
\begin{theorem} \label{thm:LRTO}
    Suppose that Assumptions (I) and (II) hold. Let $\mb{M}(\mb{q}) \in C_c^{\infty}(\R^d; \R^d)$, and let $\rho^{\eps}(\mb{q}, t)$ be the law of $\mb{q}^{\eps}_t$ in \eqref{eq:pertOL}. Let $\rho_\infty^\eps(\mb{q})$ be the invariant measure of \eqref{eq:pertOL}. For some $\varphi \in C_c^{\infty}(\R^d)$, consider $R(t, \eps; \varphi)$ defined in \eqref{def:delta}. Then 
    \begin{enumerate}[(i)]
        \item  (Convergence as $\eps \to 0^+$)
        For any given $t > 0$,
        \begin{align} \label{eq:LRT}
            \lim_{\eps \to 0^+} R(t, \eps; \varphi) = \int_0^t \int_{\R^d} \left[ \bm{\sigma \sigma}^T \mb{M} \cdot (\nabla(e^{s \mathcal{L}_0} \varphi)) \right] \rho_0 \dd \mb{q} \dd s.
        \end{align}
        Moreover, the limit in \eqref{eq:LRT} holds uniformly for all $t > 0$, i.e., for any $\eta > 0$, there exists $\eps_0$ such that for any $0 < \eps < \eps_0$ and $t > 0$ 
        \begin{align*}
            \left| R(t, \eps; \varphi) - \int_0^t \int_{\R^d} \left[ \bm{\sigma \sigma}^T \mb{M} \cdot (\nabla(e^{s \mathcal{L}_0} \varphi)) \right] \rho_0 \dd \mb{q} \dd s \right| < \eta.
        \end{align*}

        \item (Convergence as $\eps \to 0^+$, then $t \to \infty$) The following limit exists 
        \begin{align}
            \lim_{t \to \infty} \lim_{\eps \to 0^+} R(t, \eps; \varphi) = \int_0^\infty \int_{\R^d} \left[ \bm{\sigma \sigma}^T \mb{M} \cdot (\nabla(e^{s \mathcal{L}_0} \varphi)) \right] \rho_0 \dd \mb{q} \dd s,
        \end{align}
        and the convergence in $t$ is exponentially fast, i.e., there exist constants $C > 0$ and $r > 0$, which depend on $\mb{M}$ and $V$, such that 
        \begin{align}
            \left| \lim_{\eps \to 0^+} R(t, \eps; \varphi) - \int_0^\infty \int_{\R^d} \left[ \bm{\sigma \sigma}^T \mb{M} \cdot (\nabla(e^{s \mathcal{L}_0} \varphi)) \right] \rho_0 \dd \mb{q} \dd s \right| \leq C e^{-rt},
        \end{align}
        holds for all $t > 0$.

        \item (Convergence as $t \to \infty$) For any $\eps > 0$,
        \begin{align}
            \lim_{t \to \infty} R(t, \eps; \varphi) = \dfrac{1}{\eps} \left( \int_{\R^d} \varphi(\mb{q}) \rho^{\eps}_\infty(\mb{q}) \dd \mb{q} - \int_{\R^d} \varphi(\mb{q}) \rho_0(\mb{q}) \dd \mb{q} \right).
        \end{align}

        \item (Convergence as $t \to \infty$, then $\eps \to 0^+$) The following limit exists 
        \begin{align}
            \lim_{\eps \to 0^+} \lim_{t \to \infty} R(t, \eps; \varphi) = \int_0^\infty \int_{\R^d} \left[ \bm{\sigma \sigma}^T \mb{M} \cdot (\nabla(e^{s \mathcal{L}_0} \varphi)) \right] \rho_0 \dd \mb{q} \dd s,
        \end{align}
        or equivalently 
        \begin{align} \label{eq:GKformlua}
            \lim_{\eps \to 0^+} \dfrac{1}{\eps} \left( \int_{\R^d} \varphi(\mb{q}) \rho^{\eps}_\infty(\mb{q}) \dd \mb{q} - \int_{\R^d} \varphi(\mb{q}) \rho_0(\mb{q}) \dd \mb{q} \right) = \int_0^\infty \int_{\R^d} \left[ \bm{\sigma \sigma}^T \mb{M} \cdot (\nabla(e^{s \mathcal{L}_0} \varphi)) \right] \rho_0 \dd \mb{q} \dd s.
        \end{align}
    \end{enumerate}
\end{theorem}

Before proving the theorem, we provide a necessary estimate that will be frequently used later.

\begin{lemma} \label{lmm:est2}
  Under Assumption (II), there exist constants $C > 0$ and $r > 0$ that depend on $\varphi$, $\mb{M}$, and $V$, such that for any $\varphi \in C_c^{\infty}(\R^d)$, we have 
    \begin{align}
        \|\nabla (e^{t\mathcal{L}_0}\varphi)\|_{L^2(\rho_0)} \leq C e^{-r t}, \quad  
        \|\nabla (e^{t\mathcal{L}_0}\varphi)\|_{L^{\infty}(\mathrm{supp}(\mb{M}))} \leq C e^{-r t}.
    \end{align}
\end{lemma}
This lemma ensures that the integral 
\begin{align*}
    \int_0^\infty \int_{\R^d} \left[\bm{\sigma} \bm{\sigma}^T \mb{M} \cdot (\nabla(e^{s\mathcal{L}_0} \varphi)) \right] \rho_0 \, \dd\mb{q} \, \dd{s}
\end{align*}
appearing in Theorem \ref{thm:LRTO} converges 
\begin{align*}
    \left| \int_0^\infty \int_{\R^d} \left[ \bm{\sigma} \bm{\sigma}^T \mb{M} \cdot (\nabla(e^{s \mathcal{L}_0} \varphi)) \right] \rho_0 \, \dd\mb{q} \, \dd{s} \right| 
    \leq C \int_0^\infty \|\nabla(e^{t\mathcal{L}_0} \varphi)\|_{L^{\infty}(\mathrm{supp}(\mb{M}))} \, \dd{s} 
    \leq C' \int_0^\infty e^{-r s} \, \dd{s} < \infty,
\end{align*}
where $C$ and $C'$ are constants.

\begin{proof}[Proof of Lemma \ref{lmm:est2}]
    Define $\bar{\varphi} := \int_{\R^d} \varphi \rho_0 \, \dd\mb{q}$. Then, we know that $\psi(\mb{q}, t) := e^{t \mathcal{L}_0} (\varphi - \bar{\varphi})$ solves the backward equation 
    \begin{align}\label{ts}
        \dfrac{\partial \psi(\mb{q}, t)}{\partial t} = \mathcal{L}_0 \psi(\mb{q}, t), \quad \psi(\mb{q}, 0) = \varphi(\mb{q}) - \bar{\varphi}.
    \end{align}
    {\blue Notice that $\int_{\R^d} \psi(\mb{q}, 0) \rho_0 \, \dd\mb{q} = 0$. Multiplying  \eqref{ts} by $\rho_0$ and integrating in $[0,t]\times \R^d$, we have for any $t\geq 0$,
    \begin{align}\label{sloC}
        \int_{\R^d} \psi(\mb{q}, t) \rho_0 \, \dd\mb{q} = 0.
    \end{align}
    }
    By Assumption (II), Poincare's inequality holds, and we have exponential convergence 
    \begin{align} \label{exp_decay}
        \|\psi(\cdot, t)\|_{L^2(\rho_0)} \leq C e^{-r t}.
    \end{align}
    Furthermore, let $\eta := \mathcal{L}_0 \psi(\mb{q}, t) = \mathcal{L}_0 e^{t \mathcal{L}_0} (\varphi - \bar{\varphi}) = e^{t \mathcal{L}_0} \mathcal{L}_0 \varphi$. Then $\eta$ solves 
    \begin{align*}
        \dfrac{\partial \eta(\mb{q}, t)}{\partial t} = \mathcal{L}_0 \eta(\mb{q}, t), \quad \eta(\mb{q}, 0) = \mathcal{L}_0 \varphi(\mb{q}).
    \end{align*}
    Since $\int_{\R^d} \mathcal{L}_0 \varphi \rho_0 \, \dd\mb{q} = 0$, we again have exponential convergence 
    \begin{align*}
        \|\mathcal{L}_0 \psi(\cdot, t)\|_{L^2(\rho_0)} \leq C e^{-r t}.
    \end{align*}

    Using the definition of $\mathcal{L}_0$, we compute 
    \begin{align*}
        \int_{\R^d} \rho_0 |\bm{\sigma}^T \nabla (e^{t \mathcal{L}_0} \varphi)|^2 \, \dd\mb{q} 
        &= -\int_{\R^d} \rho_0 \psi(\mb{q}, t) \mathcal{L}_0 \psi(\mb{q}, t) \, \dd\mb{q} \\
        &\leq \|\psi(\cdot, t)\|_{L^2(\rho_0)} \|\mathcal{L}_0 \psi(\cdot, t)\|_{L^2(\rho_0)} \\
        &\leq C e^{-r t}.
    \end{align*}
   Since $\bm{\sigma}$ is nonsingular and constant, we have 
    \begin{align*}
        \|\nabla (e^{t \mathcal{L}_0} \varphi)\|_{L^2(\rho_0)} \leq C e^{-r t}.
    \end{align*}

    Next, we estimate the $L^{\infty}$ norm. We first observe 
    \begin{align*}
        \|\psi(\cdot, t)\|_{L^{\infty}(\R^d)} \leq \|\psi(\cdot, 0)\|_{L^{\infty}(\R^d)}, \quad 
        \|\mathcal{L}_0 \psi(\cdot, t)\|_{L^{\infty}(\R^d)} \leq \|\mathcal{L}_0 \psi(\cdot, 0)\|_{L^{\infty}(\R^d)}.
    \end{align*}
   Then, by the interior estimate from \cite[Theorem 9.11]{gilbarg2015elliptic}, for any $1 < p < \infty$, we have 
    \begin{align} \label{GT}
        \|\psi(\cdot, t)\|_{W^{2, p}(\mathrm{supp}(\mb{M}))} 
        \leq C \left[\|\psi(\cdot, t)\|_{L^p(B)} + \|\mathcal{L} \psi(\cdot, t)\|_{L^p(B)} \right].
    \end{align}
    Here, $B$ is a compact ball such that $\mathrm{supp}(\mb{M}) \subset \subset B$. Since $\rho_0(\mb{q}) = e^{-V(\mb{q})} \geq \frac{1}{c} > 0$ for $\mb{q} \in B$, for some $p > d$, we have 
    \begin{align*}
        \|\psi(\cdot, t)\|_{L^p(B)}^p 
        &\leq \|\psi(\cdot, t)\|_{L^{\infty}(\R^d)}^{p-2} \cdot \int_B c \rho_0 |\psi(\mb{q}, t)|^2 \, \dd\mb{q} 
        \leq c \|\psi(\cdot, 0)\|_{L^{\infty}(\R^d)}^{p-2} \|\psi(\cdot, t)\|_{L^2(\rho_0)}^2, \\
        \|\mathcal{L}_0 \psi(\cdot, t)\|_{L^p(B)}^p 
        &\leq \|\mathcal{L}_0 \psi(\cdot, t)\|_{L^{\infty}(\R^d)}^{p-2} \cdot \int_B c \rho_0 |\mathcal{L}_0 \psi(\mb{q}, t)|^2 \, \dd\mb{q} 
        \leq c \|\mathcal{L}_0 \psi(\cdot, 0)\|_{L^{\infty}(\R^d)}^{p-2} \|\mathcal{L}_0 \psi(\cdot, t)\|_{L^2(\rho_0)}^2.
    \end{align*}
    Combining this with \eqref{GT}, we obtain 
    \begin{align*}
        \|\psi(\cdot, t)\|_{W^{2, p}(\mathrm{supp}(\mb{M}))} \leq C e^{-r t}.
    \end{align*}
    By the embedding $W^{1, p}(B) \subset L^{\infty}(B)$, we conclude 
    \begin{align*}
        \|\nabla(e^{t \mathcal{L}_0} \varphi)\|_{L^{\infty}(\mathrm{supp}(\mb{M}))} 
        \leq \|\nabla(e^{t \mathcal{L}_0} \varphi)\|_{W^{1, p}(\mathrm{supp}(\mb{M}))} 
        \leq \|\psi(\cdot, t)\|_{W^{2, p}(\mathrm{supp}(\mb{M}))} \leq C e^{-r t}.
    \end{align*}
   This completes the proof of the lemma.
\end{proof}

Now we can proceed to prove Theorem \ref{thm:LRTO}.

\begin{proof}[Proof of Theorem \ref{thm:LRTO}]
(i) This is the key step of the proof, and parts (ii) $\sim$ (iv) will follow from it. Recall that $\rho^{\eps}(\mb{q}, t)$ solves the Fokker-Planck equation \eqref{eq:FPpOL}, which can be reformulated as 
\begin{align*}
    \frac{\partial \rho^{\eps}(\mb{q}, t)}{\partial t} 
    = \mathcal{L}^*_{\eps} \rho^{\eps}(\mb{q}, t) 
    = \mathcal{L}_0^* \rho^{\eps}(\mb{q}, t) - \eps \nabla \cdot (\rho^{\eps} \bm{\sigma} \bm{\sigma}^T \mb{M}).
\end{align*}
Thus, by Duhamel's principle, we have 
\begin{align} \label{eq:help9}
    \rho^{\eps}(\mb{q}, t) 
    = e^{t \mathcal{L}_0^*} \rho_0 
    - \eps \int_0^t e^{(t-s) \mathcal{L}_0^*} \left[\nabla \cdot \left(\rho^{\eps}(\mb{q}, s) \bm{\sigma} \bm{\sigma}^T \mb{M}(\mb{q}) \right) \right] \dd s.
\end{align}
Substituting \eqref{eq:help9} into \eqref{def:delta}, we rewrite the response function as 
\begin{align*}
    R(\eps, t; \varphi) 
    = - \int_{\R^d} \int_0^t \varphi(\mb{q}) e^{(t-s) \mathcal{L}_0^*} \left[ \nabla \cdot \left(\rho^{\eps}(\mb{q}, s) \bm{\sigma} \bm{\sigma}^T \mb{M}(\mb{q}) \right) \right] \dd s \, \dd \mb{q}.
\end{align*}
Since $\varphi \in C_c^{\infty}(\R^d)$, the term $\varphi(\mb{q}) e^{(t-s) \mathcal{L}_0^*} \left[\nabla \cdot (\rho^{\eps}(\mb{q}, s) \bm{\sigma} \bm{\sigma}^T \mb{M}(\mb{q}))\right]$ is smooth and has compact support in $\R^d \times [0, t]$, ensuring it is bounded in this domain. By Fubini’s theorem, we get 
\begin{align*}
    R(\eps, t; \varphi) 
    = - \int_0^t \int_{\R^d} \varphi(\mb{q}) e^{(t-s) \mathcal{L}_0^*} \left[\nabla \cdot (\rho^{\eps}(\mb{q}, s) \bm{\sigma} \bm{\sigma}^T \mb{M}(\mb{q})) \right] \dd \mb{q} \, \dd s.
\end{align*}
Integration by parts gives 
\begin{align}
    R(\eps, t; \varphi) 
    = \int_0^t \int_{\R^d} \left[\bm{\sigma} \bm{\sigma}^T \mb{M} \cdot \nabla (e^{(t-s) \mathcal{L}_0} \varphi) \right] \rho^{\eps}(\mb{q}, s) \dd \mb{q} \, \dd s.
\end{align}

Now, for a fixed time $T_0 > 0$ (to be chosen later), for $t > T_0$, we have 
\begin{align*}
   &\left| R (\eps, t; \varphi) - \int_0^t \int_{\R^d} \left[\bm{\sigma} \bm{\sigma}^T \mb{M} \cdot \nabla(e^{(t-s) \mathcal{L}_0} \varphi)\right] \rho_0(\mb{q}, s) \dd \mb{q} \, \dd s \right| \\
    &\quad = \left| \int_0^{T_0} \int_{\R^d} \left[\bm{\sigma} \bm{\sigma}^T \mb{M} \cdot \nabla(e^{(t-s) \mathcal{L}_0} \varphi)\right] (\rho^{\eps}(\mb{q}, s) - \rho_0) \dd \mb{q} \, \dd s \right. \\
    &\qquad \left. + \int_{T_0}^t \int_{\R^d} \left[\bm{\sigma} \bm{\sigma}^T \mb{M} \cdot \nabla(e^{(t-s) \mathcal{L}_0} \varphi)\right] (\rho^{\eps}(\mb{q}, s) - \rho_0) \dd \mb{q} \, \dd s \right|.
\end{align*}
This can be bounded by the sum of  three terms 
\begin{align*}
    \mathrm{I} &= \int_0^{T_0} \|\bm{\sigma} \bm{\sigma}^T \mb{M} \cdot \nabla(e^{(t-s) \mathcal{L}_0} \varphi)\|_{L^2(\rho_0)} \left\|\frac{\widetilde{\rho^{\eps}}(\cdot, s)}{\rho_0} \right\|_{L^2(\rho_0, \mathrm{supp}(\mb{M}))} \dd s, \\
    \mathrm{II} &= \left| \int_{T_0}^t \int_{\R^d} \left[\bm{\sigma} \bm{\sigma}^T \mb{M} \cdot \nabla(e^{(t-s) \mathcal{L}_0} \varphi)\right] (\rho^{\eps}(\mb{q}, s) - \rho_\infty^\eps(\mb{q})) \dd \mb{q} \, \dd s \right|, \\
    \mathrm{III} &= \left| \int_{T_0}^t \int_{\R^d} \left[\bm{\sigma} \bm{\sigma}^T \mb{M} \cdot \nabla(e^{(t-s) \mathcal{L}_0} \varphi)\right] (\rho_0(\mb{q}) - \rho_\infty^\eps(\mb{q})) \dd \mb{q} \, \dd s \right|.
\end{align*}
By Lemma \ref{lmm:eps} and Lemma \ref{lmm:est2}, we know 
\begin{align} \label{eq:I}
    \mathrm{I} \leq C \eps \int_0^{T_0} e^{(s-t) r} s \, \dd s = \eps C_1 (T_0 + 1) e^{r(T_0 - t)},
\end{align}
where $C_1$ is a constant depending on $\mb{M}, \varphi, \bm{\sigma}$, and $V$.
By Proposition \ref{prop:expTV} and Lemma \ref{lmm:est2}, we know 
\begin{equation} \label{eq:II}
\begin{aligned}
\mathrm{II} &\leq \int_{T_0}^t\|\bm\sigma\bm\sigma^T\mb M\cdot\nabla(e^{(t-s)\mathcal{L}_0}\varphi)\|_{L^{\infty}(\mathrm{supp}(\mb M))}\|\rho^{\eps}(\mb q,s)-\rho_\infty^\eps(\mb q)\|_{L^1(\R^d)}\dd s\\
&\leq C'\int_{T_0}^te^{r(s-t)}\cdot e^{-r_1s}\dd s\\
&\leq C_2e^{-r_1T_0}.
\end{aligned}
\end{equation}
where $C'$ and $C_2$ are constants and $r_1 > 0$ depends on $\mb{M}, \varphi, \bm{\sigma}$, and $V$.

Finally, by taking $W(\mb q)=|\mb q|^2$ in \cite[Proposition 3.7.4]{bogachev2015fokker}, $\rho_\infty^{\eps}$ converges to $\rho_0$ as $\eps \to 0^+$ in $L^1(\R^d)$. Thus, by Lemma \ref{lmm:est2}, as $\eps \to 0^+$, we have 
\begin{equation} \label{eq:III} 
\begin{aligned}
\mathrm{III}&\leq \|\rho^{\eps}_\infty-\rho_0\|_{L^1(\R^d)}\cdot\int_{T_0}^t\|\bm\sigma\bm\sigma^T\mb M\cdot\nabla(e^{(t-s)\mathcal{L}_0}\varphi)\|_{L^{\infty}(\mathrm{supp}(\mb M))}\dd s\\
&\leq C''\|\rho^{\eps}_\infty-\rho_0\|_{L^1(\R^d)}\cdot\int_{T_0}^te^{r(s-t)}\dd s\\
&\leq C_3\|\rho^{\eps}_\infty-\rho_0\|_{L^1(\R^d)}\to 0.
\end{aligned}		
\end{equation}
Here, $C_3$ depends on $\mb{M}, \varphi, \bm{\sigma}$, and $V$.

To control the error within any $\delta > 0$, take $T_0$ sufficiently large such that 	$\mathrm{II}\leq C_2e^{-r_1T_0}\leq \delta/3$. Then, choose $\eps_0$ so that $\mathrm{I}\leq C_1\eps_0(T_0+1)\leq \delta/3$ for any $\eps \leq \eps_0$. Finally, select $\eps_1 \leq \eps_0$  so that $\mathrm{III}\leq C_3\|\rho^{\eps}_\infty-\rho_0\|_{L^1(\R^d)}\leq \delta/3$ for all $\eps \leq \eps_1$. 	
Thus, for any $\eps < \eps_1$ and $t > T_0$, by \eqref{eq:I}, \eqref{eq:II}, and \eqref{eq:III}, we have 
\begin{align}
    \left| R(\eps, t; \varphi) - \int_0^t \int_{\R^d} \left[\bm{\sigma} \bm{\sigma}^T \mb{M} \cdot \nabla(e^{(t-s) \mathcal{L}_0} \varphi)\right] \rho_0(\mb{q}) \dd \mb{q} \, \dd s \right| \leq\mathrm{I}+\mathrm{II}+\mathrm{III}  \leq \delta.
\end{align}
For $t \in (0, T_0]$, use \eqref{eq:I} to bound the difference 
\begin{align}
    \left| R(\eps, t; \varphi) - \int_0^t \int_{\R^d} \left[\bm{\sigma} \bm{\sigma}^T \mb{M} \cdot \nabla(e^{(t-s) \mathcal{L}_0} \varphi)\right] \rho_0(\mb{q}) \dd \mb{q} \, \dd s \right| \leq\mathrm{I} \leq \delta / 3.
\end{align}

Thus, the convergence is uniform in $t$, and by a change of variables, we obtain 
\begin{align}
    \int_0^t \int_{\R^d} \left[\bm{\sigma} \bm{\sigma}^T \mb{M} \cdot \nabla(e^{(t-s) \mathcal{L}_0} \varphi)\right] \rho_0(\mb{q}) \dd \mb{q} \, \dd s
    = \int_0^t \int_{\R^d} \left[\bm{\sigma} \bm{\sigma}^T \mb{M} \cdot \nabla(e^{s \mathcal{L}_0} \varphi)\right] \rho_0(\mb{q}) \dd \mb{q} \, \dd s.
\end{align}
This completes the proof of (i).

Next, we prove (ii). By Lemma \ref{lmm:est2}, we know that for $T_2 > T_1 > 0$ 
\begin{align} \label{eq:est2}
\int_{T_1}^{T_2} & \int_{\R^d}[\bm\sigma\bm\sigma^T\mb M\cdot\nabla(e^{s\mathcal{L}_0}\varphi)]\rho_0(\mb q)\dd\mb q\dd s \nonumber \\
& \leq \int_{T_1}^{T_2}\|\bm\sigma\bm\sigma^T\mb M\cdot\nabla(e^{s\mathcal{L}_0}\varphi)\|_{L^{\infty}(\mathrm{supp}(\mb M))}\dd s\leq C(e^{-rT_1}-e^{-rT_2}).
\end{align}
Here $r$ and $C$ are constants in Lemma \ref{lmm:est2} that depend on $\mb M, V, \bm\sigma$ and $\varphi$. 
Therefore, we know that the limit 
\begin{align}
    \lim_{t \to \infty} \int_0^t \int_{\R^d} \left[\bm{\sigma} \bm{\sigma}^T \mb{M} \cdot \nabla(e^{s \mathcal{L}_0} \varphi)\right] \rho_0(\mb{q}) \dd \mb{q} \, \dd s
\end{align}
exists, and by letting $T_2 \to \infty$ in \eqref{eq:est2}, we obtain exponentially fast convergence, proving (ii).

For (iii), use Proposition \ref{prop:expTV} 
\begin{align*}
    \left| \int_{\R^d} \varphi(\mb{q}) (\rho^{\eps}(\mb{q}, t) - \rho_\infty^\eps) \dd \mb{q} \right|
    \leq \|\varphi\|_{L^{\infty}(\R^d)} \cdot \|\rho^{\eps}(\mb{q}, t) - \rho_\infty^\eps\|_{L^1(\R^d)} \leq C e^{-r t}.
\end{align*}
Thus, for any given $\eps$, we have 
\begin{align*}
    \lim_{t \to \infty} R(\eps, t; \varphi) = \frac{1}{\eps} \left( \int_{\R^d} \varphi(\mb{q}) (\rho_\infty^\eps(\mb{q}) - \rho_0(\mb{q})) \dd \mb{q} \right).
\end{align*}
This proves (iii).

Finally, we prove (iv). By the uniform convergence in (i), for any $\eta > 0$, there exists $\eps_0$ such that for all $t > 0$ and $\eps \in (0, \eps_0)$ 
\begin{align*}
    \left| R(t, \eps; \varphi) - \int_0^t \int_{\R^d} \left[\bm{\sigma} \bm{\sigma}^T \mb{M} \cdot \nabla(e^{s \mathcal{L}_0} \varphi)\right] \rho_0 \dd \mb{q} \, \dd s \right| < \eta.
\end{align*}
Taking the limit $t \to \infty$, we get 
\begin{align*}
    \left| \lim_{t \to \infty} R(t, \eps; \varphi) - \int_0^\infty \int_{\R^d} \left[\bm{\sigma} \bm{\sigma}^T \mb{M} \cdot \nabla(e^{s \mathcal{L}_0} \varphi)\right] \rho_0 \dd \mb{q} \, \dd s \right| < \eta.
\end{align*}
This implies 
\begin{align*}
    \lim_{\eps \to 0^+} \lim_{t \to \infty} R(\eps, t; \varphi)
    = \int_0^\infty \int_{\R^d} \left[\bm{\sigma} \bm{\sigma}^T \mb{M} \cdot \nabla(e^{s \mathcal{L}_0} \varphi)\right] \rho_0 \dd \mb{q} \, \dd s,
\end{align*}
or equivalently 
\begin{align*}
    \lim_{\eps \to 0^+} \frac{1}{\eps} \left( \int_{\R^d} \varphi(\mb{q}) \rho_\infty^\eps(\mb{q}) \dd \mb{q} - \int_{\R^d} \varphi(\mb{q}) \rho_0(\mb{q}) \dd \mb{q} \right)
    = \int_0^\infty \int_{\R^d} \left[\bm{\sigma} \bm{\sigma}^T \mb{M} \cdot \nabla(e^{s \mathcal{L}_0} \varphi)\right] \rho_0 \dd \mb{q} \, \dd s.
\end{align*}
\end{proof}
\subsection{The Green-Kubo relation}\label{sec3.2}
Compared to the general linear response formula describing the behavior of the response function $R(\eps,t;\varphi)$ with respect to an external force, the Green-Kubo relation is a special case where the limiting response function is explicitly computed via the stationary auto-correlation function. This auto-correlation function depends only on the correlation of the original unperturbed solution at different times, so it can be used to ``predict" the averaged response for a reversible system after applying a conservative force as the perturbation.

Precisely, consider a special case of linear response theory: the perturbation is also of potential form, i.e., there exists $W \in C_c^{\infty}(\R^d, \R^+)$ such that $\mb{M} = \nabla W$. In this case, the perturbed SDE is also of reversible form, and the invariant measure of \eqref{eq:pertOL} is given by 
\begin{align} \label{eq:rhoe}
	\rho^{\eps}_\infty = \dfrac{1}{Z_\eps} e^{-V(\mb{q}) + \eps W(\mb{q})}, \quad Z_\eps := \int_{\R^d} e^{-V(\mb{q}) + \eps W(\mb{q})} \dd \mb{q}.
\end{align}

By Theorem \ref{thm:LRTO}, taking two special functions as $\varphi = \mathcal{L}_0 g$ and $\mb{M} = \nabla W$, we can rigorously verify the Green-Kubo relation \eqref{eq:GK1}. Specifically, we take a special class of test functions $\varphi$ satisfying $\int \varphi \rho_0 \dd \mb{q} = 0$, and then one can uniquely solve $g$  from $\varphi = \mathcal{L}_0 g$ up to a constant. {\blue 
For this special case, we provide an intuitive proof of the Green-Kubo relation using the semigroup property. Indeed, observe that
\begin{align*}
	\int_{\R^d} \left[g - e^{t \mathcal{L}_0} g \right] (\mathcal{L}_0 W) \rho_0 \dd \mb{q} &= -\int_{\R^d} \int_0^t \frac{\partial}{\partial s} (e^{s \mathcal{L}_0} g) \dd s \, (\mathcal{L}_0 W) \rho_0 \dd \mb{q} \\
	&= -\int_{\R^d} \int_0^t e^{s \mathcal{L}_0} \mathcal{L}_0 g \, \dd s \, (\mathcal{L}_0 W) \rho_0 \dd \mb{q}.
\end{align*}

Since $\int_{\R^d} (\mathcal{L}_0 W) \rho_0 \dd \mb{q} = 0$, and using the result from the proof of \eqref{sloC} with $\psi(\mb{q}, 0) = \mathcal{L}_0 W$, we obtain for $\bar{g}:=\int_{\R^d} g \rho_0 \dd \mb{q}$, that $\bar{g} \int_{\R^d} e^{t \mathcal{L}_0} (\mathcal{L}_0 W) \rho_0 \dd \mb{q} = 0$. Thus, we have
$$
\int_{\R^d} e^{t \mathcal{L}_0} g (\mathcal{L}_0 W) \rho_0 \dd \mb{q} = \int_{\R^d} e^{t \mathcal{L}_0} (g - \bar{g}) (\mathcal{L}_0 W) \rho_0 \dd \mb{q} \leq \left( \int |e^{t \mathcal{L}_0} (g - \bar{g})|^2 \rho_0 \dd \mb{q} \right)^{\frac{1}{2}} \left( \int |\mathcal{L}_0 W|^2 \rho_0 \dd \mb{q} \right)^{\frac{1}{2}} \to 0
$$
as $t \to +\infty$ due to \eqref{exp_decay}. Therefore, the following identity holds:
\begin{equation} \label{324}
	-\int_0^\infty \int_{\R^d} \left[ e^{s \mathcal{L}_0} (\mathcal{L}_0 g) \right] (\mathcal{L}_0 W) \rho_0 \dd \mb{q} \, \dd s = \int_{\R^d} g (\mathcal{L}_0 W) \rho_0 \dd \mb{q} = \int_{\R^d} (\mathcal{L}_0 g) W \rho_0 \dd \mb{q},
\end{equation}
due to the symmetry of $\mathcal{L}_0$.
}

%
%
%
In the theorem below, we provide an alternative proof starting from the linear response theory in Theorem \ref{thm:LRTO}.

\begin{theorem}
	(The Green-Kubo relation) Suppose that Assumptions (I) and (II) hold. Let $W \in C_c^{\infty}(\R^d; \R^d)$ and $\mb{M} = \nabla W$ in \eqref{eq:pertOL}. For any $g \in C_c^{\infty}$, we have 
	\begin{align} \label{eq:GK1}
		\lim_{\eps \to 0^+} \lim_{t \to \infty} R(t, \eps; \mathcal{L}_0 g) = \int_{\R^d} W(\mb{q}) (\mathcal{L}_0 g)(\mb{q}) \rho_0(\mb{q}) \dd \mb{q} = -\int_0^\infty \int_{\R^d} \left[ e^{s \mathcal{L}_0} (\mathcal{L}_0 g) \right] (\mathcal{L}_0 W) \rho_0 \dd \mb{q} \, \dd s.
	\end{align}
\end{theorem}
Using the conventional notation for the stationary auto-correlation function 
\begin{align*}
	K_{AB}(t) := \int_{\R^d} \left( e^{t \mathcal{L}_0} A \right) B \, \rho_0 \dd \mb{q},
\end{align*}
the above Green-Kubo relation \eqref{eq:GK1} is interpreted as 
\begin{align}
	\lim_{\eps \to 0^+} \lim_{t \to \infty} R(t, \eps; \mathcal{L}_0 g) = -\int_0^{+\infty} K_{\mathcal{L}_0 g, \mathcal{L}_0 W}(t) \dd t.
\end{align}

\begin{proof}
Recall the linear response relation \eqref{eq:GKformlua}. We will recast the L.H.S. and R.H.S. of \eqref{eq:GKformlua} respectively, which in the special case $\varphi = \mathcal{L}_0 g$ and $\mb{M} = \nabla W$ will deduce the so-called Green-Kubo relation \eqref{eq:GK1}. This relation is also known as the fluctuation-dissipation theorem by Gallavotti-Cohen \cite{gallavotti1995dynamical}.

For a given observable $\varphi \in C_c^{\infty}(\R^d)$ and a potential $U \in C^{\infty}(\R^d)$ such that $e^{-U} \in L^1(\R^d)$, define 
\begin{align} \label{eq:functional}
	f(U) := \frac{\int_{\R^d} \varphi(\mb{q}) e^{-U(\mb{q})} \dd \mb{q}}{\int_{\R^d} e^{-U(\mb{q})} \dd \mb{q}}.
\end{align}
Then, the L.H.S. of \eqref{eq:GKformlua} should be viewed as the Gateaux derivative at $U = V$ in the direction $W$, i.e.,
\begin{align}
	\lim_{\eps \to 0^+} \frac{f(V + \eps W) - f(V)}{\eps} = \lim_{\eps \to 0^+} \frac{1}{\eps} \left( \int_{\R^d} \varphi(\mb{q}) \rho^{\eps}_\infty(\mb{q}) \dd \mb{q} - \int_{\R^d} \varphi(\mb{q}) \rho_0(\mb{q}) \dd \mb{q} \right).
\end{align}
By the Dominated Convergence Theorem, direct computation yields 
\begin{equation} \label{eq:LHS}
\begin{aligned}
	\lim_{\eps \to 0^+} \frac{f(V + \eps W) - f(V)}{\eps} &= \int_{\R^d} \varphi(\mb{q}) \left. \frac{\partial}{\partial \eps} \left( \frac{e^{-V(\mb{q}) + \eps W(\mb{q})}}{\int_{\R^d} e^{-V(\mb{q}) + \eps W(\mb{q})} \dd \mb{q}} \right) \right|_{\eps = 0} \dd \mb{q} \\
	&= \int_{\R^d} \varphi(\mb{q}) W(\mb{q}) \rho_0 \dd \mb{q} - \left( \int_{\R^d} \varphi(\mb{q}) \rho_0 \dd \mb{q} \right) \left( \int_{\R^d} W(\mb{q}) \rho_0 \dd \mb{q} \right).
\end{aligned}
\end{equation}
For the case $\varphi = \mathcal{L}_0 g$, we have $\int_{\R^d} (\mathcal{L}_0 g) \rho_0 \dd \mb{q} = 0$, and thus the first equality in \eqref{eq:GK1} holds.

Now, we reformulate the R.H.S. of \eqref{eq:GKformlua}. Substituting $\mb{M} = \nabla W$ into \eqref{eq:GKformlua} yields 
\begin{equation} \label{eq:RHS}
\begin{aligned}
	\int_0^\infty \int_{\R^d} \left[ \bm{\sigma} \bm{\sigma}^T \mb{M} \cdot \nabla(e^{s \mathcal{L}_0} \varphi) \right] \rho_0 \dd \mb{q} \, \dd s &= -\int_0^\infty \int_{\R^d} \left( e^{s \mathcal{L}_0} \varphi \right) \nabla \cdot \left( \rho_0 \bm{\sigma} \bm{\sigma}^T \nabla W \right) \dd \mb{q} \, \dd s \\
	&= -\int_0^\infty \int_{\R^d} \left( e^{s \mathcal{L}_0} \varphi \right) (\mathcal{L}_0 W) \rho_0 \dd \mb{q} \, \dd s,
\end{aligned}
\end{equation}
where the last equality is due to \eqref{Lsys}.
By Theorem \ref{thm:LRTO}, we know that \eqref{eq:RHS} and \eqref{eq:LHS} are equal. Thus, for $g \in C_c^{\infty}(\R^d)$ with $\varphi = \mathcal{L}_0 g$, we derive \eqref{eq:GK1}.
\end{proof}

We point out that literature (for instance, \cite[Chapter 9.3]{Pavliotis}) also rewrites \eqref{eq:GK1} as 
\begin{align} \label{eq:GK2}
\dfrac{1}{2} \int_{\R^d} \Gamma(g,W)(\mb{q}) \rho_0(\mb{q}) \dd \mb{q} &= \int_0^\infty \E\left[(\mathcal{L}_0 g)(\mb{q}_s) (\mathcal{L}_0 W)(\mb{q}_0)\right] \dd s \nonumber \\
&= \int_0^\infty \int_{\R^d} \E^{\mb{q}}\left[(\mathcal{L}_0 g)(\mb{q}_s)\right] (\mathcal{L}_0 W) \rho_0 \dd \mb{q} \dd s \nonumber \\
&= \int_0^\infty \int_{\R^d} e^{s \mathcal{L}_0} (\mathcal{L}_0 g) (\mathcal{L}_0 W) \rho_0 \dd \mb{q} \dd s.
\end{align}
Here, $\mb{q}_s, s \geq 0$ is a trajectory of \eqref{eq:OL1}, and $\Gamma(g,h)$ is the carré du champ operator defined as 
\begin{align}
	\Gamma(g,h) := \mathcal{L}_0(gh) - g \mathcal{L}_0 h - h \mathcal{L}_0 g.
\end{align}
Equations \eqref{eq:GK1} and \eqref{eq:GK2} are regarded as the Green-Kubo relation.

In the following, we provide an example to compute the diffusion coefficients via the Green-Kubo relation. In fact, \eqref{eq:GK2} can be rewritten as 
\begin{align} \label{eq:GK3}
	\int_{\R^d} \rho_0(\mb{q}) (\nabla g)^\top \bm{\sigma} \bm{\sigma}^\top \nabla W \dd \mb{q} = \int_0^\infty \E\left[(\mathcal{L}_0 g)(\mb{q}_s) (\mathcal{L}_0 W)(\mb{q}_0)\right] \dd s.
\end{align}
Though we assume $W$ and $g$ are in $C_c^{\infty}(\R^d)$, by a density argument, \eqref{eq:GK2} or \eqref{eq:GK3} also holds for $W(\mb{q}) = q_j, g(\mb{q}) = q_i$ where $i,j = 1,2,...,d$. Thus, we derive the diffusion coefficients 
\begin{align*}
	(\bm{\sigma} \bm{\sigma}^\top)_{ij} = \int_0^\infty \E\left[(\bm{\sigma}_i \bm{\sigma} \nabla V)(\mb{q}_s) (\bm{\sigma}_j \bm{\sigma} \nabla V)(\mb{q}_0)\right] \dd s,
\end{align*}
where $\bm{\sigma}_j, j = 1,2,...,d$ are the $j$-th row of $\bm{\sigma}$.

\begin{remark}
	As mentioned in \eqref{324}, the Green-Kubo relation could be verified by using the exponential convergence of $e^{s \mathcal{L}_0} \varphi$. This method is independent of linear response theory (see \cite[Result 9.1]{Pavliotis}). The way we verified the Green-Kubo formula is to reveal the intrinsic relationship between it and the linear response theory. Indeed,    Hairer-Majda  directly termed \eqref{eq:GKformlua} as the Green-Kubo relation in \cite{hairer2010simple}.
\end{remark}

\begin{remark}
We also remark that the Green-Kubo relation when studying $\rho^\eps_{\infty}$, relaxed to $\rho_0$ after canceling the perturbation force, can be used to justify the Onsager regression hypothesis \cite{onsager1931reciprocal1, onsager1931reciprocal2}; see also the review article \cite{marconi2008fluctuation}. That is, using the following relation, one can compute (in a reversible way) the original unperturbed auto-correlation function 
\begin{align*}
\lim_{\eps \to 0^+} \dfrac{1}{\eps} \left(\int_{\R^d} \varphi(\mb{q}) \rho_0(\mb{q}) \dd \mb{q} - \int_{\R^d} \varphi(\mb{q}) \rho^{\eps}_\infty(\mb{q}) \dd \mb{q}\right) &= \lim_{\eps \to 0^+} \int_0^\infty \int_{\R^d} \left[e^{s \mathcal{L}_\eps} (\mathcal{L}_\eps g)\right] (\mathcal{L}_\eps W) \rho^{\eps}_\infty(\mb{q}) \dd \mb{q} \dd s \\
&=  \int_0^\infty \int_{\R^d} \left[e^{s \mathcal{L}_0} (\mathcal{L}_0 g)\right] (\mathcal{L}_0 W) \rho_0(\mb{q}) \dd \mb{q} \dd s.
\end{align*}
\end{remark}

\section{The Linear Response Theory for Underdamped Langevin and Generalized Langevin Dynamics}\label{sec4}

{\blue Unlike the overdamped case, the degenerate ellipticity of the underdamped Langevin equation prevents the direct analogy of many results from the overdamped Langevin case. In particular, few results are known for irreversible underdamped Langevin dynamics. Fortunately, Villani's remarkable work on hypocoercivity \cite{villani2009hypocoercivity} has facilitated the analysis of the reversible case, showing that exponential convergence still holds under mild conditions on the potential.

In this section, we consider underdamped Langevin dynamics and study the linear response theory {\blue for the case where the external force perturbation is in a conservative form \(\eps \nabla W\).} The main differences between the underdamped Langevin dynamics and the overdamped case are hypocoercivity and hypoellipticity. Therefore, we will first clarify hypocoercivity and hypoellipticity before proving parallel results for the behavior of the response function \(R(\eps, t; \varphi)\).}

\subsection{Hypocoercivity and hypoellipticity of the kinetic Fokker-Planck equation}

Suppose that \(V(\mb{q}) \in C^{\infty}(\R^d)\) is positive and satisfies Assumptions (I) and (II). Consider the following underdamped Langevin dynamics 
\begin{equation} \label{eq:pertUL}
\left\{
\begin{split}
	&\dd\mb{q}^\eps = \mb{p}^\eps \dd t, \\
	&\dd\mb{p}^{\eps} = -{\bm{\sigma\sigma}}^T\mb{p}^\eps \dd t - \nabla_{\mb{q}} [V(\mb{q}^\eps) - \eps W(\mb{q}^\eps)] \dd t + \sqrt {2} \bm{\sigma} \dd \mb{B}, \\
	&(\mb{q}^{\eps}(0), \mb{p}^{\eps}(0)) \sim \rho_0 := \dfrac{1}{Z} e^{-|\mb{p}|^2/2 - V(\mb{q})}, \quad Z := \int_{\R^{2d}} e^{-|\mb{p}|^2/2 - V(\mb{q})} \dd\mb{q} \dd\mb{p},
\end{split}
\right.
\end{equation}
{\blue where \(W(\mb{q}) \in C_c^{\infty}(\R^d)\) is the potential as in Section \ref{sec3.2} such that the force perturbation is also of the potential form \(\eps \nabla W\).}

Let \(\rho^{\eps}(\mb{q}, \mb{p}, s)\) be the law of \((\mb{q}^{\eps}_s, \mb{p}^{\eps}_s)\). It satisfies the following Fokker-Planck equation 
\begin{align} \label{eq:FPpUL}
\dfrac{\p \rho^{\eps}(\mb{q}, \mb{p}, t)}{\p t} = \mathcal{L^*_{\eps}} \rho(\mb{q}, \mb{p}, t), \quad \rho^{\eps}(\mb{q}, \mb{p}, 0) = \rho_0,
\end{align}
where \(\mathcal{L^*_{\eps}}\) is the perturbed Fokker-Planck operator 
\begin{align}
\mathcal{L^*_{\eps}} \rho := -\mb{p} \cdot \nabla_{\mb{q}} \rho + \nabla_{\mb{q}} [V(\mb{q}) - \eps W(\mb{q})] \cdot \nabla_{\mb{p}} \rho + \nabla_{\mb{p}} \cdot (\rho \bm{\sigma\sigma}^T \mb{p} + \bm{\sigma\sigma}^T \nabla_{\mb{p}} \rho).
\end{align}
The invariant measure of \eqref{eq:FPpUL} is the Gibbs measure 
\begin{align} \label{eq:invp}
\rho_\infty^\eps := \dfrac{1}{Z_\eps} e^{-|\mb{p}|^2/2 - V(\mb{q}) + \eps W(\mb{q})}, \quad Z_\eps := \int_{\R^{2d}} e^{-|\mb{p}|^2/2 - V(\mb{q}) + \eps W(\mb{q})} \dd\mb{q} \dd\mb{p}.
\end{align}

LRT is interested in the limit behavior of the response function given an observable \(\varphi \in C_c^{\infty}(\R^{2d})\), i.e.,
\begin{align} \label{def:deltaU}
R(\eps,t;\varphi) := \dfrac{1}{\eps}\left( \int_{\R^{2d}} \varphi(\mb{q}, \mb{p}) \rho^\eps(\mb{q}, \mb{p}, t) \dd\mb{q} \dd\mb{p} - \int_{\R^{2d}} \varphi(\mb{q}, \mb{p}) \rho_0(\mb{q}, \mb{p}) \dd\mb{q} \dd\mb{p} \right).
\end{align}
As in \eqref{eq:BULE}, the unperturbed backward operator \(\mathcal{L}_0\) is given by 
\begin{align} \label{eq:BO}
\mathcal{L}_0 f := \mb{p} \cdot \nabla_{\mb{q}} f - \nabla_{\mb{q}} V(\mb{q}) \cdot \nabla_{\mb{p}} f - (\bm{\sigma\sigma}^T \mb{p}) \cdot \nabla_{\mb{p}} f + (\bm{\sigma\sigma}^T) : \nabla^2_{\mb{p}} f,
\end{align}
and the unperturbed Fokker-Planck operator \(\mathcal{L}_0^*\) is given by 
\begin{align} \label{eq:UO}
\mathcal{L}_0^* \rho := -\mb{p} \cdot \nabla_{\mb{q}} \rho + \nabla_{\mb{q}} V(\mb{q}) \cdot \nabla_{\mb{p}} \rho + \nabla_{\mb{p}} \cdot (\rho \bm{\sigma\sigma}^T \mb{p} + \bm{\sigma\sigma}^T \nabla_{\mb{p}} \rho).
\end{align}

We now consider the equation of \(p^\eps(\mb{q}, \mb{p}, t) = \dfrac{\rho^\eps(\mb{q}, \mb{p}, t)}{\rho^\eps_\infty(\mb{q}, \mb{p})}\), which is 
\begin{align*} 
\dfrac{\p p^\eps}{\p t} = \mathcal{L}_{\eps,1} p^\eps, \quad p^\eps(\mb{q}, \mb{p}, 0) = \dfrac{\rho_0}{\rho_\infty^\eps},
\end{align*}
where
\begin{align}\label{eq:newFP} 
\mathcal{L}_{\eps,1} f := -\mb{p} \cdot \nabla_{\mb{q}} f + (\nabla_{\mb{q}} V(\mb{q}) - \eps \nabla_{\mb{q}} W(\mb{q})) \cdot \nabla_{\mb{p}} f - (\bm{\sigma\sigma}^T \mb{p}) \cdot \nabla_{\mb{p}} f + (\bm{\sigma\sigma}^T) : \nabla^2_{\mb{p}} f.
\end{align}
We point out that, unlike the overdamped Langevin dynamics, \(\mathcal{L}_{\eps,1}\) is not exactly the same as the generator, since the first two terms in \(\mathcal{L}_0\) form the Liouville operator.

We can prove that \(\mathcal{L}_{\eps,1}\) generates a strongly continuous semigroup of contraction on \(L^2(\rho_\infty^\eps)\).

\begin{lemma}\label{lem:4.1semi} 
	Let  \(V(\mb{q}) \in C^\infty(\R^d;\R)\) satisfy Assumption (I) and \(W(\mb{q}) \in C_c^{\infty}(\R^d;\R)\). Consider equation \eqref{eq:FPpUL}. Then 
	\begin{enumerate}[(i)]
		\item \(\rho_0/\rho_\infty^\varepsilon \in H^k(\rho_\infty^\varepsilon), k = 1, 2, \dots\)
		\item \(\mathcal{L}_{\eps,1}\) in \eqref{eq:newFP} generates a strongly continuous semigroup of contractions in \(L^2(\rho_{\infty}^{\eps})\);
		\item Equation \eqref{eq:FPpUL} admits a unique solution \(\rho^{\varepsilon}(\mb{q}, \mb{p}, t) \in C^1([0, T]; L^2(1/\rho_\infty^\varepsilon))\) for any \(T > 0\).
	\end{enumerate}
\end{lemma}

By H\"ormander's hypoellipticity Theorem \ref{thm:hypoellipticity}, we can prove that the unique solution of \eqref{eq:FPpUL} is smooth (proof is in Appendix \ref{app:hy}).

\begin{lemma} \label{lmm:smoothUL}
	Let \(\rho^\eps(\mb{q}, \mb{p}, t)\) be the unique solution of \eqref{eq:FPpUL}. Then it is smooth in \(\R^{2d} \times (0, \infty)\).
\end{lemma}

\subsection{Linear response theory for underdamped Langevin}

According to \cite{dolbeault2015hypocoercivity}, for any \(\varphi \in C_c^{\infty}(\R^{2d})\), the solution to the backward equation
\begin{align}
\dfrac{\p \varphi(\mb{q}, \mb{p}, t)}{\p t} = \mathcal{L}_0 \varphi(\mb{q}, \mb{p}, t), \quad \varphi(\mb{q}, \mb{p}, 0) = \varphi(\mb{q}, \mb{p}),
\end{align}
converges exponentially fast in \(L^2(\rho_0)\), i.e., 
\begin{align} \label{eq:expc}
\|\varphi(\cdot, \cdot, t) - \bar{\varphi}\|_{L^2(\rho_0)} \leq C e^{-\lambda t}.
\end{align}
Here \(\bar{\varphi} := \int_{\R^{2d}} \varphi(\mb{q}, \mb{p}) \rho_0(\mb{q}, \mb{p}) \dd\mb{q} \dd\mb{p}\), \(C > 0\) is a constant depending on \(V, \bm{\sigma}\), and \(\varphi\), and \(r > 0\) is a constant depending on \(V\) and \(\bm{\sigma}\). 

Since \(W\) is of compact support, the conditions of hypocoercivity for \(V\) in Assumption (II) still apply to the perturbed Fokker-Planck equation \cite{dolbeault2015hypocoercivity}. Hence, exponential convergence also holds 
\begin{align} \label{eq:expu}
	\|\rho^\eps(\cdot, \cdot, t) - \rho_\infty^\eps(\cdot, \cdot)\|_{L^2(1/\rho_\infty^\eps)} \leq e^{-r t} \|\rho_0 - \rho^\eps_\infty\|_{L^2(1/\rho_\infty^\eps)}.
\end{align}
Here \(r > 0\) depends only on \(V\) and \(\bm{\sigma}\), but not on \(\eps\); \(C_1 > 0\) depends only on \(V\) and \(\bm{\sigma}\). The reason they do not depend on \(\eps\) is that the perturbation from \(W\) is smooth and of compact support. 

Using these exponential convergence results (as in Lemma \ref{lmm:est2}), we can rigorously verify the linear response theory. 

\begin{theorem} \label{thm:LRTU}
	Let \(V(\mb{q}) \in C^{\infty}(\R^d)\) be positive and satisfy Assumptions (I) and (II). Suppose \(W(\mb{q}) \in C_c^{\infty}(\R^d)\). Let \(\rho^{\eps}(\mb{q}, t)\) be the law of \((\mb{q}^{\eps}_t, \mb{p}^\eps_t)\) in \eqref{eq:pertUL}, and let \(\rho_\infty^\eps(\mb{q})\) be the invariant measure of \eqref{eq:pertUL}. For some \(\varphi \in C_c^{\infty}(\R^{2d})\), consider \(R(t, \eps; \varphi)\) defined in \eqref{def:deltaU}. Then 
	\begin{enumerate}[(i)]
		\item (Convergence as \(\eps \to 0^+\)) For any given \(t > 0\),
		\begin{align} \label{eq:LRTU}
		\lim_{\eps \to 0^+} R(t, \eps; \varphi) = \int_0^t \int_{\R^{2d}} [\nabla_{\mb{q}} W \cdot (\nabla_{\mb{p}}(e^{s \mathcal{L}_0} \varphi))] \rho_0 \dd\mb{q} \dd\mb{p} \dd s.
		\end{align} 
		Moreover, the limit in \eqref{eq:LRTU} holds uniformly for all \(t > 0\); i.e., for any \(\eta > 0\), there exists \(\eps_0\) such that for any \(0 < \eps < \eps_0\) and \(t > 0\),
		\begin{align*}
		\left| R(t, \eps; \varphi) - \int_0^t \int_{\R^{2d}} [\nabla_{\mb{q}} W \cdot (\nabla_{\mb{p}}(e^{s \mathcal{L}_0} \varphi))] \rho_0 \dd\mb{q} \dd\mb{p} \dd s \right| < \eta;
		\end{align*}
		\item (Convergence as \(\eps \to 0^+\), then \(t \to \infty\)) The following limit exists 
		\begin{align} \label{eq:tlim}
		\lim_{t \to \infty} \lim_{\eps \to 0^+} R(t, \eps; \varphi) = \int_0^\infty \int_{\R^{2d}} [\nabla_{\mb{q}} W \cdot (\nabla_{\mb{p}}(e^{s \mathcal{L}_0} \varphi))] \rho_0 \dd\mb{q} \dd\mb{p} \dd s,
		\end{align}
		and the convergence in \(t\) is exponentially fast. That is, there exist constants \(C > 0\) and \(r > 0\) (depending on \(\mb{M}\) and \(V\)) such that
		\begin{align}
		\left| \lim_{\eps \to 0^+} R(t, \eps; \varphi) - \int_0^\infty \int_{\R^{2d}} [\nabla_{\mb{q}} W \cdot (\nabla_{\mb{p}}(e^{s \mathcal{L}_0} \varphi))] \rho_0 \dd\mb{q} \dd\mb{p} \dd s \right| \leq C e^{-r t}
		\end{align}
		holds for all \(t > 0\);
		\item (Convergence as \(t \to \infty\)) For any \(\eps > 0\),
		\begin{align}
		\lim_{t \to \infty} R(t, \eps; \varphi) = \dfrac{1}{\eps} \left( \int_{\R^{2d}} \varphi(\mb{q}) \rho^{\eps}_\infty(\mb{q}) \dd\mb{q} \dd\mb{p} - \int_{\R^{2d}} \varphi(\mb{q}) \rho_0(\mb{q}, t) \dd\mb{q} \dd\mb{p} \right);
		\end{align}
		\item (Convergence as \(t \to \infty\), then \(\eps \to 0^+\)) The following limit exists 
		\begin{align}
		\lim_{\eps \to 0^+} \lim_{t \to \infty} R(t, \eps; \varphi) = \int_0^\infty \int_{\R^{2d}} [\nabla_{\mb{q}} W \cdot (\nabla_{\mb{p}}(e^{s \mathcal{L}_0} \varphi))] \rho_0 \dd\mb{q} \dd\mb{p} \dd s,
		\end{align}
		or equivalently,
		\begin{align} \label{eq:GKU}
		\lim_{\eps \to 0^+} \dfrac{1}{\eps} \left( \int_{\R^{2d}} \varphi(\mb{q}) \rho^{\eps}_\infty(\mb{q}) \dd\mb{q} \dd\mb{p} - \int_{\R^{2d}} \varphi(\mb{q}) \rho_0(\mb{q}) \dd\mb{q} \dd\mb{p} \right) = \int_0^\infty \int_{\R^{2d}} [\nabla_{\mb{q}} W \cdot (\nabla_{\mb{p}}(e^{s \mathcal{L}_0} \varphi))] \rho_0 \dd\mb{q} \dd\mb{p} \dd s.
		\end{align}
	\end{enumerate}
\end{theorem}

\begin{proof}
	We first prove (i). By Lemma \ref{lmm:smoothUL}, \(\rho^{\eps}(\mb{q}, \mb{p}, t)\) is smooth. By Duhamel's principle, we have
	\begin{align}
		\rho^{\eps}(\mb{q}, \mb{p}, t) = e^{t\mathcal{L}_0^*} \rho_0 - \eps \int_0^t e^{(t-s)\mathcal{L}_0^*} (\nabla_{\mb{q}} W(\mb{q}) \cdot \nabla_{\mb{p}} \rho^{\eps}(\mb{q}, \mb{p}, s)) \dd s.
	\end{align}
	Remember that \(e^{t\mathcal{L}_0^*} \rho_0 = \rho_0\), so
	\begin{align*}
		R(t, \eps; \varphi) &= \dfrac{1}{\eps} \left( \int_{\R^{2d}} \varphi(\mb{q}, \mb{p}) \rho^\eps(\mb{q}, \mb{p}, t) \dd\mb{q} \dd\mb{p} - \int_{\R^{2d}} \varphi(\mb{q}, \mb{p}) \rho_0(\mb{q}, \mb{p}) \dd\mb{q} \dd\mb{p} \right)\\
		&= - \int_{\R^{2d}} \int_0^t \varphi(\mb{q}, \mb{p}) e^{(t-s)\mathcal{L}_0^*} (\nabla_{\mb{q}} W(\mb{q}) \cdot \nabla_{\mb{p}} \rho^{\eps}(\mb{q}, \mb{p}, s)) \dd s \dd\mb{q} \dd\mb{p}.
	\end{align*}
	Since \(\varphi(\mb{q}, \mb{p}) e^{(t-s)\mathcal{L}_0^*} (\nabla_{\mb{q}} W(\mb{q}) \cdot \nabla_{\mb{p}} \rho^{\eps}(\mb{q}, \mb{p}, s))\) is a smooth function with compact support in \(\R^{2d} \times [0, t]\), Fubini's theorem gives
	\begin{align*}
		R(t, \eps; \varphi) = - \int_0^t \int_{\R^{2d}} \varphi(\mb{q}, \mb{p}) e^{(t-s)\mathcal{L}_0^*} (\nabla_{\mb{q}} W(\mb{q}) \cdot \nabla_{\mb{p}} \rho^{\eps}(\mb{q}, \mb{p}, s)) \dd\mb{q} \dd\mb{p} \dd s.
	\end{align*}
	Integration by parts then yields
	\begin{align*}
		R(t, \eps; \varphi) = \int_0^t \int_{\R^{2d}} [\nabla_{\mb{q}} W \cdot (\nabla_{\mb{p}} (e^{(t-s)\mathcal{L}_0} \varphi))] \rho^\eps(\mb{q}, \mb{p}, s) \dd\mb{q} \dd\mb{p} \dd s.
	\end{align*}
	By \eqref{eq:expu}, we know
	\begin{align*}
		\|\rho^\eps(\cdot, \cdot, s) - \rho_0(\cdot, \cdot)\|_{L^2(1/\rho_\infty^\eps)} &\leq \|\rho^\eps(\cdot, \cdot, s) - \rho_\infty^\eps(\cdot, \cdot)\|_{L^2(1/\rho_\infty^\eps)} + \|\rho_0 - \rho_\infty^\eps\|_{L^2(1/\rho_\infty^\eps)} \\
		&\leq (1 + e^{-r s}) \|\rho_0 - \rho_\infty^\eps\|_{L^2(1/\rho_\infty^\eps)} \\
		&\leq 2 \|\rho_0 - \rho_\infty^\eps\|_{L^2(1/\rho_\infty^\eps)}.
	\end{align*}
	Thus, by Cauchy-Schwarz's inequality, we have
	\begin{equation} \label{eq:help11}
		\begin{aligned}
		&\left| R(t, \eps; \varphi) - \int_0^t \int_{\R^{2d}} [\nabla_{\mb{q}} W \cdot (\nabla_{\mb{p}} (e^{s \mathcal{L}_0} \varphi))] \rho_0 \dd\mb{q} \dd\mb{p} \dd s \right| \\
		&\quad = \left| \int_0^t \int_{\R^{2d}} [\nabla_{\mb{q}} W \cdot (\nabla_{\mb{p}} (e^{(t-s)\mathcal{L}_0} \varphi))] (\rho^\eps(\mb{q}, \mb{p}, s) - \rho_0(\mb{q}, \mb{p})) \dd\mb{q} \dd\mb{p} \dd s \right| \\
		&\quad \leq \int_0^t \|\nabla_{\mb{q}} W \cdot (\nabla_{\mb{p}} (e^{(t-s)\mathcal{L}_0} \varphi))\|_{L^2(\rho^\eps_\infty)} \|\rho^\eps(\mb{q}, \mb{p}, s) - \rho_0(\mb{q}, \mb{p})\|_{L^2(1/\rho_\infty^\eps)} \dd s \\
		&\quad \leq 2 \|\rho_0 - \rho_\infty^\eps\|_{L^2(1/\rho_\infty^\eps)} \cdot \int_0^t \|\nabla_{\mb{q}} W \cdot (\nabla_{\mb{p}} (e^{s \mathcal{L}_0} \varphi))\|_{L^2(\rho^\eps_\infty)} \dd s.
		\end{aligned} 
	\end{equation}
	Because \(W\) has compact support, there exists a constant \(C' > 0\), depending only on \(V, W\), and \(\bm{\sigma}\), such that
	\begin{align}
		\|\nabla_{\mb{q}} W \cdot (\nabla_{\mb{p}} (e^{s \mathcal{L}_0} \varphi))\|_{L^2(\rho^\eps_\infty)} \leq C \|\nabla_{\mb{p}} (e^{s \mathcal{L}_0} \varphi)\|_{L^2(\rho^\eps_\infty)} \leq C' \|\nabla_{\mb{p}} (e^{s \mathcal{L}_0} \varphi)\|_{L^2(\rho_0)}.
	\end{align}
	Notice that \(\nabla_{\mb{p}} (e^{s \mathcal{L}_0} \varphi) = \nabla_{\mb{p}} (e^{s \mathcal{L}_0} (\varphi - \bar{\varphi}))\). Thus, by \eqref{eq:expc}, we have
	\begin{align*}
		\|\bm{\sigma}^T \nabla_{\mb{p}} (e^{s \mathcal{L}_0} \varphi)\|_{L^2(\rho_0)}^2 &= - \int_{\R^{2d}} (e^{s \mathcal{L}_0} (\varphi - \bar{\varphi})) (\mathcal{L}_0 e^{s \mathcal{L}_0} (\varphi - \bar{\varphi})) \rho_0 \dd\mb{q} \dd\mb{p} \\
		&\leq \|e^{s \mathcal{L}_0} (\varphi - \bar{\varphi})\|_{L^2(\rho_0)} \|\mathcal{L}_0 e^{s \mathcal{L}_0} (\varphi - \bar{\varphi})\|_{L^2(\rho_0)} \\
		&\leq C'' e^{-2 \lambda t}.
	\end{align*}
	Here \(C'' > 0\) is a constant depending on \(\varphi, V\), and \(\bm{\sigma}\), and \(\lambda\) is the constant in \eqref{eq:expc}. Since \(\bm{\sigma}\) is non-singular, we have
	\begin{align} \label{eq:exp3}
		\|\nabla_{\mb{p}} (e^{s \mathcal{L}_0} \varphi)\|_{L^2(\rho_0)}^2 \leq C_1 e^{-r t}.
	\end{align}
	Here \(C_1\) is a constant depending on \(\varphi, V\), and \(\bm{\sigma}\). Returning to \eqref{eq:help11}, we obtain
	\begin{align*}
		\left| R(t, \eps; \varphi) - \int_0^t \int_{\R^{2d}} [\nabla_{\mb{q}} W \cdot (\nabla_{\mb{p}} (e^{s \mathcal{L}_0} \varphi))] \rho_0 \dd\mb{q} \dd\mb{p} \dd s \right| &\leq 2 C_2 \|\rho_0 - \rho_\infty^\eps\|_{L^2(1/\rho_\infty^\eps)} \cdot \int_0^t e^{-r s} \dd s \\
		&\leq C_3 \|\rho_0 - \rho_\infty^\eps\|_{L^2(1/\rho_\infty^\eps)}.
	\end{align*}
	Here \(C_3 > 0\) is a constant depending on \(\varphi, V, W\), and \(\bm{\sigma}\). Since \(\rho_0\) and \(\rho_\infty^\eps\) are the Gibbs measures given in \eqref{eq:pertUL} and \eqref{eq:invp} respectively, and \(W\) is of compact support, the dominated convergence theorem directly implies
	\begin{align*}
		\|\rho_0 - \rho_\infty^\eps\|_{L^2(1/\rho_\infty^\eps)} \to 0
	\end{align*}
	as \(\eps \to 0^+\). Since \(C_3 \|\rho_0 - \rho_\infty^\eps\|_{L^2(1/\rho_\infty^\eps)}\) is independent of \(t\) and approaches 0 as \(\eps \to 0^+\), (i) is proved, and the convergence is uniform in \(t\).
	
	Now we proceed to prove (ii). By \eqref{eq:exp3} and Cauchy-Schwarz's inequality, for any \(T_1, T_2 > 0\), we have
	\begin{align*}
		\left| \int_{T_1}^{T_2} \int_{\R^{2d}} [\nabla_{\mb{q}} W \cdot (\nabla_{\mb{p}} (e^{s \mathcal{L}_0} \varphi))] \rho_0 \dd\mb{q} \dd\mb{p} \dd s \right| &\leq \|\nabla_{\mb{q}} W\|_{L^2(\rho_0)} \cdot \int_{T_1}^{T_2} e^{-s r} \dd s \\
		&\leq C (e^{-T_1 r} - e^{-T_2 r}).
	\end{align*}
	Here \(C\) is a constant depending on \(V, W\), and \(\bm{\sigma}\). Thus, the limit in \eqref{eq:tlim} exists and converges exponentially fast in \(t\).
	
	Next, we prove (iii). This follows immediately from \eqref{eq:expu}. For any given \(\eps > 0\), we have
	\begin{align*}
		& \left| \dfrac{1}{\eps} \left( \int_{\R^{2d}} \varphi(\mb{q}, \mb{p}) \rho^{\eps}_\infty(\mb{q}, \mb{p}) \dd\mb{q} \dd\mb{p} - \int_{\R^{2d}} \varphi(\mb{q}, \mb{p}) \rho^{\eps}(\mb{q}, \mb{p}, t) \dd\mb{q} \dd\mb{p} \right) \right| \\
		&\quad \leq \dfrac{1}{\eps} \|\varphi\|_{L^2(1/\rho^\eps_\infty)} \|\rho^{\eps}(\cdot, \cdot, t) - \rho_\infty^\eps(\cdot, \cdot)\|_{L^2(1/\rho^\eps_\infty)} \\
		&\quad \leq C_1 e^{-r_1 t} / \eps.
	\end{align*}
	Here \(C_1 > 0\) is a constant depending on \(\varphi, V, W\), and \(\bm{\sigma}\), and \(r_1 > 0\) is in \eqref{eq:expu}. Thus, the limit as \(t \to \infty\) exists for each \(\eps > 0\).
	
	Finally, we prove (iv). By (i), we know that for any \(\eta > 0\), there exists \(\eps_0 > 0\) such that for all \(\eps \in (0, \eps_0)\), we have
	\begin{align*}
		\left| R(t, \eps; \varphi) - \int_0^t \int_{\R^{2d}} [\nabla_{\mb{q}} W \cdot (\nabla_{\mb{p}} (e^{s \mathcal{L}_0} \varphi))] \rho_0 \dd\mb{q} \dd\mb{p} \dd s \right| < \eta.
	\end{align*}
	Passing to the limit as \(t \to \infty\), we obtain
	\begin{align*}
		\left| \lim_{t \to \infty} R(t, \eps; \varphi) - \int_0^\infty \int_{\R^{2d}} [\nabla_{\mb{q}} W \cdot (\nabla_{\mb{p}} (e^{s \mathcal{L}_0} \varphi))] \rho_0 \dd\mb{q} \dd\mb{p} \dd s \right| \leq \eta.
	\end{align*}
	Thus, by definition,
	\begin{align*}
		\lim_{\eps \to 0^+} \lim_{t \to \infty} R(t, \eps; \varphi) = \int_0^\infty \int_{\R^{2d}} [\nabla_{\mb{q}} W \cdot (\nabla_{\mb{p}} (e^{s \mathcal{L}_0} \varphi))] \rho_0 \dd\mb{q} \dd\mb{p} \dd s.
	\end{align*}
	By (iii), this also implies \eqref{eq:GKU}. This completes the proof.
\end{proof}

\subsection{Linear Response for the Generalized Langevin with Memory}\label{sec4.2}

The generalized Langevin equation with a general algebraic memory kernel was introduced in \cite{kubo1966fluctuation, mori1965continued}. Recent works \cite{kou2004generalized, li2017fractional} have extended the fluctuation-dissipation relation for both overdamped and underdamped generalized Langevin dynamics.

When the memory kernel takes an exponential form, the generalized Langevin dynamics can be transformed into classical underdamped Langevin dynamics by introducing an additional variable \cite{kupferman2004fractional, pavliotis2010asymptotic}. As an application of Theorem \ref{thm:LRTU}, we consider the generalized Langevin equation with an exponential memory kernel, i.e., \eqref{eq:GLE}  
\begin{equation*} 
\left\{\begin{split}
\ddot{\mb{q}}_t &= -\nabla_{\mb{q}} V(\mb{q}) - \sum_{i=1}^n \mb{A}_i \mb{A}_i^T \int_{0}^t e^{-\alpha_i (t-s)} \dot{\mb{q}}_s \dd s + \mb{A}_i \mb{f}_t^i, \\
\dd\mb{f}_t^i &= -\alpha_i \mb{f}_t^i \dd t + \sqrt{2\beta^{-1} \alpha_i} \dd \mb{B}_t^i, \quad i=1, 2, \dots, n.
\end{split}\right.
\end{equation*}

Here, \(V \in C^{\infty}(\R^d)\) satisfies Assumptions (I) and (II) with \(\bm{\sigma} = \mb{I}_d\), and for each \(i=1, 2, \dots, n\), \(\mb{A}_i \in \R^{d \times d}\) is a constant matrix, \(\alpha_i > 0\) is a constant, and \(\mb{B}_t^i \in \R^d\) is a standard Brownian motion. The \(\mb{B}_t^i\)'s are independent. Furthermore, we assume that the Hessian of \(V\) is uniformly bounded, i.e., there exists a constant \(C > 0\) such that for all \(\mb{q} \in \R^d\),
\begin{align} \label{eq:quad}
\|\nabla^2 V(\mb{q})\| \leq C.
\end{align}
Here, the norm refers to the Frobenius norm.

Let \(\mb{z}_t^i = -\mb{A}_i^T \int_0^t e^{-\alpha_i (t-s)} \dot{\mb{q}}_s \dd s + \mb{f}_t^i\) and \(\dot{\mb{q}} = \mb{p}\). Then \eqref{eq:GLE} can be reformulated as \eqref{eq:RGLE}, i.e.,
\begin{equation*} 
\left\{\begin{split}
\dd{\mb{q}} &= \mb{p} \dd t, \\
\dd{\mb{p}} &= \left(-\nabla_{\mb{q}} V(\mb{q}) + \sum_{i=1}^n \mb{A}_i \mb{z}_i \right) \dd t, \\
\dd{\mb{z}}_i &= -(\alpha_i \mb{z}_i + \mb{A}_i^T \mb{p}) \dd t + \sqrt{2 \beta^{-1} \alpha_i} \dd \mb{B}_i, \quad i=1, 2, \dots, n.
\end{split}\right.
\end{equation*}

For simplicity and without loss of generality, we consider the case \(n=1\) and \(\mb{A}_1 = \mb{I}_d\), i.e. \eqref{eq:1GLE}. The conclusions in this section can be easily extended to the general case.

Now based on \eqref{eq:1GLE}, we impose a compact and smooth perturbation on the potential term, i.e., for \(W \in C_c^{\infty}(\R^d)\), and consider the following perturbed SDE 
\begin{equation} \label{eq:pGLE}
\left\{
\begin{aligned}
&\dd{\mb{q}^\eps} = \mb{p}^\eps \dd t, \\
&\dd{\mb{p}^\eps} = \left(-\nabla_{\mb{q}} V(\mb{q}^\eps) + \eps \nabla_{\mb{q}} W(\mb{q}^\eps) + \mb{z}^\eps \right) \dd t, \\
&\dd{\mb{z}^\eps} = -(\alpha \mb{z}^\eps + \mb{p}^\eps) \dd t + \sqrt{2 \beta^{-1} \alpha} \dd \mb{B}, \\
&(\mb{q}^\eps, \mb{p}^\eps, \mb{z}^\eps) \sim \rho_0.
\end{aligned}
\right.
\end{equation}

The invariant measure of \eqref{eq:pGLE} is given by 
\begin{align}
\rho_\infty^\eps (\mb{q}, \mb{p}, \mb{z}) := \dfrac{1}{Z^\eps} \exp(-\beta H(\mb{q}, \mb{p}, \mb{z}) + \eps \beta W(\mb{q})), \quad Z^\eps := \int_{\R^{3d}} e^{-\beta H(\mb{q}, \mb{p}, \mb{z}) + \eps \beta W(\mb{q})} \dd\mb{q} \dd\mb{p} \dd\mb{z} < \infty.
\end{align}

Let \(\rho^\eps(\mb{q}, \mb{p}, \mb{z}, t)\) be the law of \((\mb{q}^\eps(t), \mb{p}^\eps(t), \mb{z}^\eps(t))\) in \eqref{eq:1GLE}. Then \(\rho^\eps(\mb{q}, \mb{p}, \mb{z}, t)\) satisfies the following perturbed Fokker-Planck equation with initial value \(\rho^\eps(\mb{q}, \mb{p}, \mb{z}, 0) = \rho_0(\mb{q}, \mb{p}, \mb{z})\) 
\begin{equation}
\begin{aligned}
\dfrac{\p \rho^\eps}{\p t} = \mathcal{L}_\eps^* \rho, \quad \mathcal{L}_\eps^* \rho := & -\mb{p} \cdot \nabla_{\mb{q}} \rho + \nabla_{\mb{q}} (V(\mb{q}) - \eps W(\mb{q})) \cdot \nabla_{\mb{p}} \rho + (\mb{p} \cdot \nabla_{\mb{z}} \rho - \mb{z} \cdot \nabla_{\mb{p}} \rho) \\
& + \alpha \nabla_{\mb{z}} \cdot \left(\mb{z} \rho + \dfrac{1}{\beta} \nabla_{\mb{z}} \rho \right).
\end{aligned}
\end{equation}

Consider \(\varphi \in C_c^{\infty}(\R^{3d})\). The linear response theory focuses on the limit behavior of
\begin{align} \label{def:deltaG}
R(\eps, t; \varphi) := \dfrac{1}{\eps} \left( \int_{\R^{3d}} \varphi(\mb{q}, \mb{p}, \mb{z}) \rho^\eps(\mb{q}, \mb{p}, \mb{z}, t) \dd \mb{q} \dd \mb{p} \dd \mb{z} - \int_{\R^{3d}} \varphi(\mb{q}, \mb{p}, \mb{z}) \rho_0(\mb{q}, \mb{p}, \mb{z}) \dd \mb{q} \dd \mb{p} \dd \mb{z} \right).
\end{align}
The unperturbed backward operator \(\mathcal{L}_0\) is given by 
\begin{align}
\mathcal{L}_0 f := \mb{p} \cdot \nabla_{\mb{q}} f - \nabla_{\mb{q}} V(\mb{q}) \cdot \nabla_{\mb{p}} f - (\mb{p} \cdot \nabla_{\mb{z}} f - \mb{z} \cdot \nabla_{\mb{p}} f) - \alpha \mb{z}^T \nabla_{\mb{z}} f + \dfrac{\alpha}{\beta} \Delta_{\mb{z}} f.
\end{align}
The backward equation with initial value \(\varphi\) is then written as 
\begin{align} \label{eq:BGLE1}
\dfrac{\p \varphi(\mb{q}, \mb{p}, \mb{z}, t)}{\p t} = \mathcal{L}_0 \varphi(\mb{q}, \mb{p}, \mb{z}, t), \quad \varphi(\mb{q}, \mb{p}, \mb{z}, 0) = \varphi(\mb{q}, \mb{p}, \mb{z}).
\end{align}

According to the hypocoercivity result in \cite{hairer2008ballistic}, under Assumptions (I), (II), and \eqref{eq:quad} on the smooth potential \(V\), the solution to \eqref{eq:BGLE1} converges exponentially in the weighted Sobolev space \(H^1(\rho_0)\) 
\begin{align} \label{eq:expGLE}
\left\|\varphi(\cdot, \cdot, \cdot, t) - \int_{\R^{3d}} \varphi(\mb{q}, \mb{p}, \mb{z}) \rho_0 \dd \mb{q} \dd \mb{p} \dd \mb{z} \right\|_{H^1(\rho_0)} \leq e^{-\lambda t} \left\|\varphi(\cdot, \cdot, \cdot, 0) - \int_{\R^{3d}} \varphi(\mb{q}, \mb{p}, \mb{z}) \rho_0 \dd \mb{q} \dd \mb{p} \dd \mb{z} \right\|_{H^1(\rho_0)}.
\end{align}

We emphasize that \eqref{eq:expGLE} is stronger than the result used in the underdamped Langevin case, which focuses on exponential convergence in \(L^2(\rho_0)\), while \eqref{eq:expGLE} concerns \(H^1(\rho_0)\). With \eqref{eq:expGLE}, we can rigorously verify the LRT for the generalized Langevin dynamics.
\begin{theorem} \label{thm:LRTG}
	Let \(V(\mb{q}) \in C^{\infty}(\R^d)\) be positive, satisfying Assumptions (I), (II), and \eqref{eq:quad}. Suppose \(W(\mb{q}) \in C_c^{\infty}(\R^d)\). Let \(\rho^{\eps}(\mb{q}, \mb{p}, \mb{z}, t)\) be the law of \((\mb{q}^{\eps}(t), \mb{p}^\eps(t), \mb{z}^{\eps}(t))\) in \eqref{eq:pGLE}, and let \(\rho_\infty^\eps(\mb{q}, \mb{p}, \mb{z})\) be the invariant measure of \eqref{eq:pGLE}. For some \(\varphi \in C_c^{\infty}(\R^{3d})\), consider \(R(t, \eps; \varphi)\) defined in \eqref{def:deltaG}. Then 
	\begin{enumerate}[(i)]
		\item (Convergence as \(\eps \to 0^+\)) For any given \(t > 0\),
		\begin{align} \label{eq:LRTG}
		\lim_{\eps \to 0^+} R(t, \eps; \varphi) = \int_0^t \int_{\R^{3d}} [\nabla_{\mb{q}} W \cdot (\nabla_{\mb{p}} (e^{s \mathcal{L}_0} \varphi))] \rho_0 \dd \mb{q} \dd \mb{p} \dd \mb{z} \dd s.
		\end{align}
		Moreover, the limit in \eqref{eq:LRTG} holds uniformly for all \(t > 0\); i.e., for any \(\eta > 0\), there exists \(\eps_0\) such that for any \(0 < \eps < \eps_0\) and \(t > 0\),
		\begin{align*}
		\left| R(t, \eps; \varphi) - \int_0^t \int_{\R^{3d}} [\nabla_{\mb{q}} W \cdot (\nabla_{\mb{p}} (e^{s \mathcal{L}_0} \varphi))] \rho_0 \dd \mb{q} \dd \mb{p} \dd \mb{z} \dd s \right| < \eta.
		\end{align*}
		
		\item (Convergence as \(\eps \to 0^+\) then \(t \to \infty\)) The following limit exists 
		\begin{align} \label{eq:tlimG}
		\lim_{t \to \infty} \lim_{\eps \to 0^+} R(t, \eps; \varphi) = \int_0^\infty \int_{\R^{3d}} [\nabla_{\mb{q}} W \cdot (\nabla_{\mb{p}} (e^{s \mathcal{L}_0} \varphi))] \rho_0 \dd \mb{q} \dd \mb{p} \dd \mb{z} \dd s,
		\end{align}
		and the convergence in \(t\) is exponentially fast; i.e., there exist constants \(C > 0\) and \(r > 0\) that depend on \(W\) and \(V\) such that
		\begin{align}
		\left| \lim_{\eps \to 0^+} R(t, \eps; \varphi) - \int_0^\infty \int_{\R^{3d}} [\nabla_{\mb{q}} W \cdot (\nabla_{\mb{p}} (e^{s \mathcal{L}_0} \varphi))] \rho_0 \dd \mb{q} \dd \mb{p} \dd \mb{z} \dd s \right| \leq C e^{-rt}
		\end{align}
		holds for all \(t > 0\).
		
		\item (Convergence as \(t \to \infty\)) For any \(\eps > 0\),
		\begin{align}
		\lim_{t \to \infty} R(t, \eps; \varphi) = \dfrac{1}{\eps} \left( \int_{\R^{3d}} \varphi(\mb{q}, \mb{p}, \mb{z}) \rho^{\eps}_\infty(\mb{q}, \mb{p}, \mb{z}) \dd \mb{q} \dd \mb{p} \dd \mb{z} - \int_{\R^{3d}} \varphi(\mb{q}, \mb{p}, \mb{z}) \rho_0(\mb{q}, \mb{p}, \mb{z}) \dd \mb{q} \dd \mb{p} \dd \mb{z} \right).
		\end{align}
		
		\item (Convergence as \(t \to \infty\) then \(\eps \to 0^+\)) The following limit exists 
		\begin{align}
		\lim_{\eps \to 0^+} \lim_{t \to \infty} R(t, \eps; \varphi) = \int_0^\infty \int_{\R^{3d}} [\nabla_{\mb{q}} W \cdot (\nabla_{\mb{p}} (e^{s \mathcal{L}_0} \varphi))] \rho_0 \dd \mb{q} \dd \mb{p} \dd \mb{z} \dd s,
		\end{align}
		or equivalently,
		\begin{equation} \label{eq:GKG}
		\begin{aligned} 
	 \lim_{\eps \to 0^+}& \dfrac{1}{\eps} \left( \int_{\R^{3d}} \varphi(\mb{q}, \mb{p}, \mb{z}) \rho^{\eps}_\infty(\mb{q}, \mb{p}, \mb{z}) \dd \mb{q} \dd \mb{p} \dd \mb{z} - \int_{\R^{3d}} \varphi(\mb{q}, \mb{p}, \mb{z}) \rho_0(\mb{q}, \mb{p}, \mb{z}) \dd \mb{q} \dd \mb{p} \dd \mb{z} \right) \\
		& = \int_0^\infty \int_{\R^{3d}} [\nabla_{\mb{q}} W \cdot (\nabla_{\mb{p}} (e^{s \mathcal{L}_0} \varphi))] \rho_0 \dd \mb{q} \dd \mb{p} \dd \mb{z} \dd s.
		\end{aligned}
		\end{equation}
	\end{enumerate}
\end{theorem}

\begin{proof}
	The proof follows closely from the case of the underdamped Langevin dynamics, so we provide only a sketch.

	For (i), as in the proof of Theorem \ref{thm:LRTU}, by Duhamel's principle, we have 
	\begin{align}
	 R(t,\eps;\varphi) = -\int_{\R^{3d}} \int_0^t \varphi(\mb{q}, \mb{p}, \mb{z}) e^{(t-s)\mathcal{L}_0^*} \left( \nabla_{\mb{q}} W(\mb{q}) \cdot \nabla_{\mb{p}} \rho^\eps(\mb{q}, \mb{p}, \mb{z}) \right) \dd \mb{q} \dd \mb{p} \dd \mb{z} \dd s.
	\end{align}
	Using integration by parts and Fubini's theorem, this becomes 
	\begin{align}
	 R(t,\eps;\varphi) = \int_0^t \int_{\R^{3d}} \left[ \nabla_{\mb{q}} W(\mb{q}) \cdot \nabla_{\mb{p}} \left( e^{(t-s)\mathcal{L}_0} \varphi \right) \right] \rho^\eps(\mb{q}, \mb{p}, \mb{z}, s) \dd \mb{q} \dd \mb{p} \dd \mb{z} \dd s.
	\end{align}	
	The variational structure of the Fokker-Planck equation gives the bound 
	\begin{align*}
	\|\rho^\eps(\cdot, \cdot, \cdot, s) - \rho_0(\cdot, \cdot, \cdot)\|_{L^2(1/\rho^\eps_\infty)} &\leq \|\rho^\eps(\cdot, \cdot, \cdot, s) - \rho^\eps_\infty(\cdot, \cdot, \cdot)\|_{L^2(1/\rho^\eps_\infty)} \\
	& \quad + \|\rho_0(\cdot, \cdot, \cdot) - \rho^\eps_\infty(\cdot, \cdot, \cdot)\|_{L^2(1/\rho^\eps_\infty)} \\
	&\leq 2 \|\rho_0(\cdot, \cdot, \cdot) - \rho^\eps_\infty(\cdot, \cdot, \cdot)\|_{L^2(1/\rho^\eps_\infty)}.
	\end{align*}
	Using Cauchy-Schwarz's inequality, we obtain 
	\begin{align*}
	& \left| R(t, \eps; \varphi) - \int_0^t \int_{\R^{3d}} \left[ \nabla_{\mb{q}} W(\mb{q}) \cdot \nabla_{\mb{p}} \left( e^{(t-s) \mathcal{L}_0} \varphi \right) \right] \rho_0(\mb{q}, \mb{p}, \mb{z}) \dd \mb{q} \dd \mb{p} \dd \mb{z} \dd s \right| \\
	&\quad \leq 2 \|\rho_0(\cdot, \cdot, \cdot) - \rho^\eps_\infty(\cdot, \cdot, \cdot)\|_{L^2(1/\rho^\eps_\infty)} \cdot \int_0^t \left\| \nabla_{\mb{q}} W(\mb{q}) \cdot \nabla_{\mb{p}} \left( e^{s \mathcal{L}_0} \varphi \right) \right\|_{L^2(\rho^\eps_\infty)} \dd s.
	\end{align*}
	Since \(W\) is of compact support, by \eqref{eq:expGLE}, we have 
	\begin{align*}
	\int_0^t \left\| \nabla_{\mb{q}} W(\mb{q}) \cdot \nabla_{\mb{p}} \left( e^{s \mathcal{L}_0} \varphi \right) \right\|_{L^2(\rho_\infty^\eps)} \dd s \leq C \int_0^t e^{-\lambda s} \dd s \leq C'.
	\end{align*}
	Here, \(C'\) is a constant independent of \(t\) and \(\eps\). Hence,
	\begin{align*}
	\left| R(t, \eps; \varphi) - \int_0^t \int_{\R^{3d}} \left[ \nabla_{\mb{q}} W(\mb{q}) \cdot \nabla_{\mb{p}} \left( e^{(t-s) \mathcal{L}_0} \varphi \right) \right] \rho_0(\mb{q}, \mb{p}, \mb{z}) \dd \mb{q} \dd \mb{p} \dd \mb{z} \dd s \right| \\
	\leq 2C' \|\rho_0(\cdot, \cdot, \cdot) - \rho^\eps_\infty(\cdot, \cdot, \cdot)\|_{L^2(1/\rho^\eps_\infty)},
	\end{align*}
	and the right-hand side tends to 0 as \(\eps \to 0\). Thus, \eqref{eq:LRTG} holds, and the convergence is uniform in \(t\).

	For (ii), by \eqref{eq:expGLE}, for any \(0 < T_1 < T_2\), we have 
	\begin{align*}
	\left| \int_{T_1}^{T_2} \int_{\R^{3d}} \left[ \nabla_{\mb{q}} W(\mb{q}) \cdot \nabla_{\mb{p}} \left( e^{(t-s) \mathcal{L}_0} \varphi \right) \right] \rho_0(\mb{q}, \mb{p}, \mb{z}) \dd \mb{q} \dd \mb{p} \dd \mb{z} \dd s \right| \\
	\leq \|\nabla_{\mb{q}} W\|_{L^2(\rho_0)} \cdot \int_{T_1}^{T_2} e^{-\lambda s} \dd s \leq C(e^{-T_1 \lambda} - e^{-T_2 \lambda}).
	\end{align*}
	Here, \(C\) is a constant depending on \(V, W\), and \(\alpha\). Thus, \eqref{eq:tlimG} holds, and the convergence in \(t\) is exponentially fast.

	Next, for (iii), this is a consequence of convergence in \(L^1\). For any given \(\eps > 0\), we have 
	\begin{align*}
	&  \left| \dfrac{1}{\eps} \left( \int_{\R^{3d}} \varphi(\mb{q}, \mb{p}, \mb{z}) \rho^{\eps}_\infty(\mb{q}, \mb{p}, \mb{z}) \dd \mb{q} \dd \mb{p} \dd \mb{z} - \int_{\R^{3d}} \varphi(\mb{q}, \mb{p}, \mb{z}) \rho^\eps(\mb{q}, \mb{p}, \mb{z}, t) \dd \mb{q} \dd \mb{p} \dd \mb{z} \right) \right| \\
	&\quad  \leq \dfrac{1}{\eps} \|\varphi\|_{L^\infty} \|\rho^\eps(\cdot, \cdot, \cdot, t) - \rho_\infty^\eps(\cdot, \cdot, \cdot)\|_{L^1} \\
	&\quad  \leq C_1 e^{-r_\eps t} / \eps.
	\end{align*}
	Here, \(r_\eps > 0\) and \(C_1 > 0\) are constants depending on \(\eps\). Thus, we have convergence as \(t \to \infty\).

	Finally, for (iv), by (i), for any \(\eta > 0\), there exists \(\eps_0 > 0\) such that for all \(\eps \in (0, \eps_0)\), we have 
	\begin{align*}
	\left| R(t, \eps; \varphi) - \int_0^t \int_{\R^{3d}} [\nabla_{\mb{q}} W \cdot (\nabla_{\mb{p}} (e^{s \mathcal{L}_0} \varphi))] \rho_0 \dd \mb{q} \dd \mb{p} \dd \mb{z} \dd s \right| < \eta.
	\end{align*}
	Passing the limit \(t \to \infty\), we have 
	\begin{align*}
	\left| \lim_{t \to \infty} R(t, \eps; \varphi) - \int_0^\infty \int_{\R^{3d}} [\nabla_{\mb{q}} W \cdot (\nabla_{\mb{p}} (e^{s \mathcal{L}_0} \varphi))] \rho_0 \dd \mb{q} \dd \mb{p} \dd \mb{z} \dd s \right| \leq \eta.
	\end{align*}
	Thus, we conclude 
	\begin{align*}
	\lim_{\eps \to 0^+} \lim_{t \to \infty} R(t, \eps; \varphi) = \int_0^\infty \int_{\R^{3d}} [\nabla_{\mb{q}} W \cdot (\nabla_{\mb{p}} (e^{s \mathcal{L}_0} \varphi))] \rho_0 \dd \mb{q} \dd \mb{p} \dd \mb{z} \dd s.
	\end{align*}
	By (iii), this also implies \eqref{eq:GKG}. This concludes the proof.
\end{proof}

\section*{Compliance with Ethical Standards}
The authors declare no conflict of interest.

\section*{Acknowledgment}
Yuan Gao was partially supported by NSF under Award DMS-2204288. Jian-Guo Liu and Zibu Liu were partially supported by NSF under Award DMS-2106988.

\appendix

\section{Omitted Proofs for Reversibility Results}\label{sec:proof}
This appendix provides some preliminary definitions and the omitted proofs for the reversibility of overdamped and underdamped Langevin dynamics.

\begin{proof}[Proof of Lemma \ref{lmm:rev}]
	Suppose that \eqref{eq:OLE} with the initial density function $\rho_0(\mb{q})$ describes a reversible process. According to Definition \ref{def:reversibility}, $\rho_0$ must be a stationary measure. Consider the distribution of the time-reversed process $\mb{q}^*(t)$ with respect to a fixed time \(T\) (see Definition \ref{def:reversibility})
	\begin{align}
	\mb{q}^*_t = \mb{q}_{T-t}.
	\end{align}
	Because $\rho_0(\mb{q})$ is stationary, the time-reversed process $\mb{q}^*_t$ has the same distribution as the original process. Thus, for any $\varphi_1, \varphi_2 \in C_b^{\infty}(\mathbb{R}^d)$, reversibility implies 
	\begin{align}
	\E[\varphi_1(\mb{q}_t)\varphi_2(\mb{q}_0) \mid \mb{q}_0 \sim \rho_0] = \E[\varphi_1(\mb{q}^*_t)\varphi_2(\mb{q}^*_0) \mid \mb{q}_0 \sim \rho_0].
	\end{align}	
	Using the Markov property and the invariance of $\rho_0$, for any \(0 \leq t \leq T\), we have 
	\begin{equation}
	\begin{aligned}
	\E[\varphi_1(\mb{q}^*_t)\varphi_2(\mb{q}^*_0) \mid \mb{q}_0 \sim \rho_0] &= \E[\varphi_1(\mb{q}_{T-t})\varphi_2(\mb{q}_T) \mid \mb{q}_0 \sim \rho_0] \\
	&= \E[\varphi_1(\mb{q}_{T-t})\varphi_2(\mb{q}_T) \mid \mb{q}_{T-t} \sim \rho_0] \\
	&= \E[\varphi_1(\mb{q}_0)\varphi_2(\mb{q}_t) \mid \mb{q}_0 \sim \rho_0].
	\end{aligned}
	\end{equation}
	Thus, we obtain the symmetry 
	\begin{align}
	\E[\varphi_1(\mb{q}_t)\varphi_2(\mb{q}_0) \mid \mb{q}_0 \sim \rho_0] = \E[\varphi_1(\mb{q}_0)\varphi_2(\mb{q}_t) \mid \mb{q}_0 \sim \rho_0].
	\end{align}
	
	Conversely, if for arbitrary $\varphi_1$ and $\varphi_2$, equation \eqref{eq:revs} holds, we can apply smooth functions to approximate simple functions. By the monotone convergence theorem, this allows us to conclude reversibility.
\end{proof}

\begin{proof}[Proof of Theorem \ref{thm:rev1}]
	
We will prove the theorem using the following logical steps 
\begin{align*}
(i) \implies (ii) \implies (iv) \implies (iii) \implies (iv) \implies (i).
\end{align*}
		
\textbf{Step 1.} \((i) \implies (ii)\).

Denote the solution of \eqref{eq:OLE} as \(\mb{q}_t\). For arbitrary \(\varphi_1, \varphi_2 \in C_b^{\infty}(\mathbb{R}^d)\), we have 
\begin{equation}\label{conEE}
\begin{aligned}
\dfrac{\E[\varphi_1(\mb{q}_t)\varphi_2(\mb{q}_0)] - \E[\varphi_1(\mb{q}_0)\varphi_2(\mb{q}_0)]}{t} 
&= \dfrac{\E[(\varphi_1(\mb{q}_t)-\varphi_1(\mb{q}_0))\varphi_2(\mb{q}_0)]}{t} \\
&= \dfrac{\E[\E[(\varphi_1(\mb{q}_t)-\varphi_1(\mb{q}_0))\varphi_2(\mb{q}_0) \mid \mb{q}_0]]}{t} \\
&= \dfrac{\E[\varphi_2(\mb{q}_0)\E[(\varphi_1(\mb{q}_t) - \varphi_1(\mb{q}_0)) \mid \mb{q}_0]]}{t} \\
&= \int_{\mathbb{R}^d} \varphi_2(\mb{q}) \dfrac{\E[(\varphi_1(\mb{q}_t)-\varphi_1(\mb{q}_0)) \mid \mb{q}_0 = \mb{q}]}{t} \rho_0(\mb{q}) \, \mathrm{d}\mb{q}.
\end{aligned}
\end{equation}
Similarly, we have 
\begin{align}
\dfrac{\E[\varphi_2(\mb{q}_t)\varphi_1(\mb{q}_0)] - \E[\varphi_2(\mb{q}_0)\varphi_1(\mb{q}_0)]}{t}
= \int_{\mathbb{R}^d} \varphi_1(\mb{q}) \dfrac{\E[(\varphi_2(\mb{q}_t)-\varphi_2(\mb{q}_0)) \mid \mb{q}_0 = \mb{q}]}{t} \rho_0(\mb{q}) \, \mathrm{d}\mb{q}.
\end{align}
By Definition \ref{def:reversibility}, we know 
\begin{align*}
\E[\varphi_1(\mb{q}_t)\varphi_2(\mb{q}_0)] = \E[\varphi_1(\mb{q}_0)\varphi_2(\mb{q}_t)].
\end{align*}
Therefore, we have 
\begin{align}
\int_{\mathbb{R}^d} \varphi_2(\mb{q}) \dfrac{\E[(\varphi_1(\mb{q}_t) - \varphi_1(\mb{q}_0)) \mid \mb{q}_0 = \mb{q}]}{t} \rho_0(\mb{q}) \, \mathrm{d}\mb{q} 
= \int_{\mathbb{R}^d} \varphi_1(\mb{q}) \dfrac{\E[(\varphi_2(\mb{q}_t) - \varphi_2(\mb{q}_0)) \mid \mb{q}_0 = \mb{q}]}{t} \rho_0(\mb{q}) \, \mathrm{d}\mb{q}.
\end{align}
Passing the limit as \(t \to 0^+\), using the definition of the generator, we obtain 
\begin{align*}
\int_{\mathbb{R}^d} (\mathcal{L}\varphi_1)(\mb{q}) \varphi_2(\mb{q}) \rho_0(\mb{q}) \, \mathrm{d}\mb{q}
= \int_{\mathbb{R}^d} (\mathcal{L}\varphi_2)(\mb{q}) \varphi_1(\mb{q}) \rho_0(\mb{q}) \, \mathrm{d}\mb{q},
\end{align*}
which is exactly \eqref{eq:rev}.
		
\textbf{Step 2.} \((ii) \implies (iv)\).

Rewrite \eqref{eq:rev} as 
\begin{align*}
\int_{\mathbb{R}^d} \frac{1}{\rho_0(\mb{q})} \rho_0(\mb{q}) \varphi_1(\mb{q}) \mathcal{L}^*(\rho_0(\mb{q}) \varphi_2(\mb{q})) \, \mathrm{d}\mb{q} 
= \int_{\mathbb{R}^d} \frac{1}{\rho_0(\mb{q})} \rho_0(\mb{q}) \varphi_2(\mb{q}) \mathcal{L}^*(\rho_0(\mb{q}) \varphi_1(\mb{q})) \, \mathrm{d}\mb{q}.
\end{align*}
This holds for arbitrary \(\varphi_1, \varphi_2 \in C_0^{\infty}(\mathbb{R}^d)\). Hence, for arbitrary \(\phi_1(\mb{q}), \phi_2(\mb{q}) \in C_0^{\infty}(\mathbb{R}^d)\), we get 
\begin{align*}
\int_{\mathbb{R}^d} \frac{1}{\rho_0(\mb{q})} \phi_1(\mb{q}) \mathcal{L}^*(\phi_2(\mb{q})) \, \mathrm{d}\mb{q} 
= \int_{\mathbb{R}^d} \frac{1}{\rho_0(\mb{q})} \phi_2(\mb{q}) \mathcal{L}^*(\phi_1(\mb{q})) \, \mathrm{d}\mb{q}.
\end{align*}
Let \(1/\rho_0(\mb{q}) = e^{U(\mb{q})}\) for some \(U(\mb{q}): \mathbb{R}^d \to \mathbb{R}\). Substituting \eqref{eq:FPOLE}, we have 
\begin{align*}
\int_{\mathbb{R}^d} e^{U(\mb{q})} \phi_1(\mb{q}) \left[-\nabla \cdot (\phi_2 \mb{b}) + \frac{1}{2} \nabla^2 : (\phi_2 \bm{\sigma \sigma}^T)\right] \, \mathrm{d}\mb{q} 
= \int_{\mathbb{R}^d} e^{U(\mb{q})} \phi_2(\mb{q}) \left[-\nabla \cdot (\phi_1 \mb{b}) + \frac{1}{2} \nabla^2 : (\phi_1 \bm{\sigma \sigma}^T)\right] \, \mathrm{d}\mb{q}.
\end{align*}
Integration by parts yields 
\begin{equation} \label{eq:help}
\begin{aligned}
 \int_{\mathbb{R}^d} & e^{U(\mb{q})} \phi_2(\mb{q}) \nabla \phi_1(\mb{q}) \cdot \mb{b}(\mb{q}) + \frac{1}{2} \nabla^2 (e^{U(\mb{q})} \phi_1(\mb{q})) : (\phi_2 \bm{\sigma \sigma}^T) \, \mathrm{d}\mb{q} \\
&= \int_{\mathbb{R}^d} e^{U(\mb{q})} \phi_1(\mb{q}) \nabla \phi_2(\mb{q}) \cdot \mb{b}(\mb{q}) + \frac{1}{2} \nabla^2 (e^{U(\mb{q})} \phi_2(\mb{q})) : (\phi_1 \bm{\sigma \sigma}^T) \, \mathrm{d}\mb{q}.
\end{aligned}
\end{equation}
Next, using the fact that \(\bm{\sigma \sigma}^T\) is constant, we obtain 
\begin{align*}
\nabla^2(e^{U(\mb{q})} \phi_1(\mb{q})) : (\phi_2 \bm{\sigma \sigma}^T) - \nabla^2(e^{U(\mb{q})} \phi_2(\mb{q})) : (\phi_1 \bm{\sigma \sigma}^T) \\
= 2[\bm{\sigma \sigma}^T \nabla (e^{U(\mb{q})})] \cdot (\phi_2(\mb{q}) \nabla \phi_1(\mb{q}) - \phi_1(\mb{q}) \nabla \phi_2(\mb{q})) \\
+ e^{U(\mb{q})} \nabla \cdot (\bm{\sigma \sigma}^T (\phi_2(\mb{q}) \nabla \phi_1(\mb{q}) - \phi_1(\mb{q}) \nabla \phi_2(\mb{q}))).
\end{align*}
Substituting this into \eqref{eq:help} and integrating by parts again, we find 
\begin{align*}
\int_{\mathbb{R}^d} e^{U(\mb{q})} (\phi_2(\mb{q}) \nabla \phi_1(\mb{q}) - \phi_1(\mb{q}) \nabla \phi_2(\mb{q})) \cdot \left(\mb{b} + \frac{1}{2} \bm{\sigma \sigma}^T \nabla U(\mb{q})\right) \, \mathrm{d}\mb{q} = 0.
\end{align*}
By Lemma \ref{lmm:aux} and the smoothness of \(\mb{b}\) and \(\rho_0\), we obtain 
\begin{align*}
\mb{b} + \frac{1}{2} \bm{\sigma \sigma}^T \nabla U(\mb{q}) = 0.
\end{align*}
Since \(\bm{\sigma}\) is nonsingular, we have 
\begin{align}
-\nabla U(\mb{q}) = 2(\bm{\sigma \sigma}^T)^{-1} \mb{b}.
\end{align}
Thus, \(\rho_0(\mb{q}) = e^{-U(\mb{q})}\), which proves \((iv)\).

\textbf{Step 3.} \((iv) \implies (iii)\). 

If \((iv)\) holds, we can derive that \(\mathcal{L}\) generates a strongly continuous semigroup on \(L^2(\rho_0 \, \mathrm{d}\mb{q})\) by Hille-Yosida's theorem (see proof of Lemma \ref{Lmm}). Therefore, the diffusion process is well-defined, and \(\mb{x}(t)\) admits a probability density \(\rho(\mb{q},t)\), which is the unique local solution of \eqref{eq:FPOLE} with initial value \(\rho(\mb{q},0) = \rho_0\) (see \cite{buhler2018functional}).

Notice 
\begin{align*}
\mathcal{L}^* \rho_0 &= \frac{1}{Z} \nabla \cdot \left(-e^{-U(\mb{q})} \mb{b}(\mb{q}) + \frac{1}{2} \bm{\sigma \sigma}^T \nabla e^{-U(\mb{q})}\right) \\
&= -\frac{1}{Z} \nabla \cdot \left(e^{-U(\mb{q})} \frac{1}{2} \bm{\sigma \sigma}^T \nabla U(\mb{q}) - e^{-U(\mb{q})} \frac{1}{2} \bm{\sigma \sigma}^T \nabla U(\mb{q})\right) \\
&= 0.
\end{align*}
Thus, the unique local solution of \eqref{eq:FPOLE} with initial value \(\rho_0(\mb{q})\) is exactly \(\rho(\mb{q},t) = \rho_0(\mb{q})\), which is actually a global solution. Hence 
\begin{align}
\mb{j}(\rho(t)) = \mb{j}(\rho_0) = \mb{b} \rho_0 - \frac{1}{2} \bm{\sigma \sigma}^T \nabla \rho_0 = \mb{0}.
\end{align} 
Therefore, \(\mb{j}(\rho) = \mb{0}\), i.e., the probability flux is zero.

\textbf{Step 4.} \((iii) \implies (iv)\). 

Since \(\mb{j}(\rho(t)) = 0\), we know that \(\rho_0\) satisfies 
\begin{align*}
\mb{b}(\mb{q}) \rho_0(\mb{q}) - \frac{1}{2} \bm{\sigma \sigma}^T \nabla \rho_0(\mb{q}) = 0.
\end{align*}
Since \(\rho_0(\mb{q}) > 0\), we have 
\begin{align*}
\mb{b}(\mb{q}) = \frac{\bm{\sigma \sigma}^T \nabla \log(\rho_0(\mb{q}))}{2}.
\end{align*}
Thus, \((iv)\) is proved by setting \(U(\mb{q}) = -\log(\rho_0(\mb{q}))\). 

\textbf{Step 5.} \((iv) \implies (i)\).

From the proof of Lemma \ref{Lmm} and Hille-Yosida's theorem, we know that \(\mathcal{L}\) generates a strongly continuous semigroup on \(L^2(\rho_0 \, \mathrm{d}\mb{q})\), denoted as \(e^{t\mathcal{L}}\), for \(t \geq 0\). Therefore, for any \(\varphi_1 \in C_0^{\infty}(\mathbb{R}^d)\), \((e^{t\mathcal{L}} \varphi_1)(\mb{q}, t)\) is the solution of \eqref{eq:BOLE} with initial value \(f(\mb{q}, 0) = \varphi_1(\mb{q})\).

Direct calculation yields that for arbitrary \(f\) and \(g \in C_0^{\infty}(\mathbb{R}^d)\) 
\begin{align*}
\int_{\mathbb{R}^d} \rho_0(\mb{q}) g(\mb{q}) (\mathcal{L}f)(\mb{q}) \, \mathrm{d}\mb{q} 
&= \frac{1}{Z} \int_{\mathbb{R}^d} e^{-U(\mb{q})} g(\mb{q}) \left(\mb{b} \cdot \nabla f(\mb{q}) + \frac{1}{2} \bm{\sigma \sigma}^T : \nabla^2 f(\mb{q})\right) \, \mathrm{d}\mb{q} \\
&= \frac{1}{Z} \int_{\mathbb{R}^d} g(\mb{q}) \left(\frac{\bm{\sigma \sigma}^T \nabla e^{-U(\mb{q})}}{2} \cdot \nabla f(\mb{q}) + e^{-U(\mb{q})} \frac{1}{2} \bm{\sigma \sigma}^T : \nabla^2 f(\mb{q})\right) \, \mathrm{d}\mb{q} \\
&= \frac{1}{2Z} \int_{\mathbb{R}^d} g(\mb{q}) \nabla \cdot (e^{-U(\mb{q})} \bm{\sigma \sigma}^T \nabla f(\mb{q})) \, \mathrm{d}\mb{q} \\
&= -\frac{1}{2Z} \int_{\mathbb{R}^d} e^{-U(\mb{q})} (\nabla g(\mb{q}))^T \bm{\sigma \sigma}^T \nabla f(\mb{q}) \, \mathrm{d}\mb{q}.
\end{align*}
The last equality is symmetric with respect to \(f\) and \(g\), thus \(\mathcal{L}\) is symmetric in \(C_0^{\infty}(\mathbb{R}^d) \subset L^2(\rho_0 \, \mathrm{d}\mb{q})\). By \cite[Corollary 7.3.2]{buhler2018functional}, the dual semigroup of \(e^{t\mathcal{L}}\) in \(L^2(\rho_0 \, \mathrm{d}\mb{q})\) (denoted as \(\widetilde{e^{t\mathcal{L}}}\)) is generated by the dual operator of \(\mathcal{L}\) in \(L^2(\rho_0 \, \mathrm{d}\mb{q})\) (denoted as \(\widetilde{\mathcal{L}}\)).

Since \(\mathcal{L}\) is symmetric on \(C_0^{\infty}(\mathbb{R}^d)\), we know \(\widetilde{e^{t\mathcal{L}}} \varphi = e^{t\mathcal{L}} \varphi\) for any \(\varphi \in C_0^{\infty}(\mathbb{R}^d)\). Thus, for arbitrary \(\varphi_1, \varphi_2 \in C_0^{\infty}(\mathbb{R}^d)\), we have 
\begin{align*}
\int_{\mathbb{R}^d} \rho_0(\mb{q}) \varphi_2(\mb{q}) (e^{t\mathcal{L}} \varphi_1)(\mb{q}, t) \, \mathrm{d}\mb{q} 
= \int_{\mathbb{R}^d} \rho_0(\mb{q}) \varphi_1(\mb{q}) (\widetilde{e^{t\mathcal{L}}} \varphi_2)(\mb{q}, t) \, \mathrm{d}\mb{q}
= \int_{\mathbb{R}^d} \rho_0(\mb{q}) \varphi_1(\mb{q}) (e^{t\mathcal{L}} \varphi_2)(\mb{q}, t) \, \mathrm{d}\mb{q}.
\end{align*}
Thus, \(\E[\varphi_1(\mb{q}_0)\varphi_2(\mb{q}_t) \mid \mb{q}(0) \sim \rho_0] = \E[\varphi_2(\mb{q}_0)\varphi_1(\mb{q}_t) \mid \mb{q}(0) \sim \rho_0]\), which proves reversibility.

\end{proof}

\begin{proof}[Proof of Lemma \ref{lmm:rev2}]
	
First, if the process with initial density function \(\rho_0(\mb{q}, \mb{p})\) is a reversible process, then \(\rho_0\) is a stationary measure. Consider the distribution of the time-reversed process \((\mb{q}^*(t), \mb{p}^*(t))\) with respect to a fixed time \(T\), as defined in Definition \ref{def:reversibility}. Since \(\mb{q}\) is even and \(\mb{p}\) is odd, we have 
\begin{align}
	\mb{q}^*_t = \mb{q}_{T-t}, \quad \mb{p}^*_t = -\mb{p}_{T-t}.
\end{align}
Since \(\rho_0\) is invariant, it follows that 
\begin{align}
	(\mb{q}^*_t, \mb{p}^*_t) = (\mb{q}_{T-t}, -\mb{p}_{T-t}) \sim \rho_0(\mb{q}, -\mb{p}),
\end{align}
i.e., the distribution of the time-reversed process has density \(\rho_0(\mb{q}, -\mb{p})\).

Let \(\widetilde{\rho_0}(\mb{q}, \mb{p}) = \rho_0(\mb{q}, -\mb{p})\). For any \(\varphi_1, \varphi_2, \psi_1, \psi_2 \in C_b^{\infty}(\mathbb{R}^d)\), the reversibility condition implies 
\begin{align*}
	&\E[\varphi_1(\mb{q}_0) \varphi_2(\mb{p}_0) \psi_1(\mb{q}_t) \psi_2(\mb{p}_t) \mid (\mb{q}_0, \mb{p}_0) \sim \rho_0] \\
	&\quad = \E[\varphi_1(\mb{q}^*_0) \varphi_2(\mb{p}^*_0) \psi_1(\mb{q}^*_t) \psi_2(\mb{p}^*_t) \mid (\mb{q}_0, \mb{p}_0) \sim \rho_0].
\end{align*}

Let \(\widetilde{\varphi_2}(\mb{p}) = \varphi_2(-\mb{p})\) and \(\widetilde{\psi_2}(\mb{p}) = \psi_2(-\mb{p})\). Using the Markov property and the invariance of \(\rho_0\), for any \(t \in [0, T]\), we have 
\begin{align*}
	\E[\varphi_1(\mb{q}^*_0) \varphi_2(\mb{p}^*_0) \psi_1(\mb{q}^*_t) \psi_2(\mb{p}^*_t) \mid (\mb{q}_0, \mb{p}_0) \sim \rho_0] &= \E[\varphi_1(\mb{q}_T) \varphi_2(-\mb{p}_T) \psi_1(\mb{q}_{T-t}) \psi_2(-\mb{p}_{T-t}) \mid (\mb{q}_0, \mb{p}_0) \sim \rho_0] \\
	&= \E[\varphi_1(\mb{q}_T) \widetilde{\varphi_2}(\mb{p}_T) \psi_1(\mb{q}_{T-t}) \widetilde{\psi_2}(\mb{p}_{T-t}) \mid (\mb{q}_0, \mb{p}_0) \sim \rho_0] \\
	&= \E[\varphi_1(\mb{q}_T) \widetilde{\varphi_2}(\mb{p}_T) \psi_1(\mb{q}_{T-t}) \widetilde{\psi_2}(\mb{p}_{T-t}) \mid (\mb{q}_{T-t}, \mb{p}_{T-t}) \sim \rho_0] \\
	&= \E[\varphi_1(\mb{q}_t) \widetilde{\varphi_2}(\mb{p}_t) \psi_1(\mb{q}_0) \widetilde{\psi_2}(\mb{p}_0) \mid (\mb{q}_0, \mb{p}_0) \sim \rho_0].
\end{align*}
Therefore, equation \eqref{eq:rev2} holds.

Conversely, if equation \eqref{eq:rev2} holds for arbitrary \(\varphi_1\) and \(\varphi_2\), then by using smooth functions to approximate simple functions, one can conclude reversibility through the monotone convergence theorem. Hence, we employ \eqref{eq:rev2} to prove reversibility instead of checking the definition directly.

\end{proof}

\begin{lemma} \label{lmm:aux}
	For any \(\bm{\phi} = (\phi_1, \phi_2, \dots, \phi_d) \in C_0^{\infty}(\mathbb{R}^d; \mathbb{R}^d)\) with \(\phi_i \geq 0, i = 1, 2, \dots, d\), there exist \(f, g \in C_0^{\infty}(\mathbb{R}^d)\) such that 
	\[
		f \nabla g - g \nabla f = \bm{\phi}.
	\]
\end{lemma}

\begin{proof}
	For each \(i\), let \(f_i = \sqrt{\phi_i}\) and \(g_i = x_i f_i\). Then
	\[
		f_i \nabla g_i - g_i \nabla f_i = f_i^2 \nabla \left( \frac{g_i}{f_i} \right) = \phi_i \bm{e}_i,
	\]
	where \(\bm{e}_i = (0, 0, \dots, 1, 0, \dots, 0) \in \mathbb{R}^d\) has 1 in the \(i\)-th position and 0 elsewhere. Thus, \(\bm{\phi} = \sum_{i=1}^d f_i \nabla g_i - g_i \nabla f_i\).
\end{proof}

\begin{proof}[Proof of Theorem \ref{thm:revULE}]
	\textbf{Step 1.} We prove \((i) \implies (ii)\).

	Because the process is reversible (and therefore stationary), \((\mb{q}^*_0, \mb{p}^*_0)\) and \((\mb{q}_0, \mb{p}_0)\) have the same distribution. Thus \((\mb{q}^*_0, \mb{p}^*_0) = (\mb{q}_T, -\mb{p}_T)\) has the same distribution as \((\mb{q}(0), \mb{p}(0))\), meaning that \(\rho_0(\mb{q}, \mb{p}) = \rho_0(\mb{q}, -\mb{p})\).

	From the definition of reversibility, we know that (i) is equivalent to 
	\begin{equation} \label{eq:exp2}
		\E[\varphi_1(\mb{q}_0) \varphi_2(\mb{p}_0) \psi_1(\mb{q}_t) \psi_2(\mb{p}_t) \mid (\mb{q}_0, \mb{p}_0) \sim \rho_0 ] 
		= \E[\varphi_1(\mb{q}_t) \varphi_2(-\mb{p}_t) \psi_1(\mb{q}_0) \psi_2(-\mb{p}_0) \mid (\mb{q}_0, \mb{p}_0) \sim \rho_0 ]
	\end{equation}
	for any \(\varphi_1(\mb{q}), \varphi_2(\mb{p}), \psi_1(\mb{q}), \psi_2(\mb{p}) \in C_0^{\infty}(\mathbb{R}^d)\).
	
	By similar calculations to those in \eqref{conEE}, we know
	\begin{align*}
		\iint & \varphi_1(\mb{q}) \varphi_2(\mb{p}) \frac{\E\left[ \psi_1(\mb{q}_t) \psi_2(\mb{p}_t) - \psi_1(\mb{q}_0) \psi_2(\mb{p}_0) \mid \mb{q}_0 = \mb{q}, \mb{p}_0 = \mb{p} \right]}{t} \rho_0(\mb{q}, \mb{p}) \, d\mb{q} \, d\mb{p} \\
		=& \iint \psi_1(\mb{q}) \psi_2(-\mb{p}) \frac{\E\left[ \varphi_1(\mb{q}_t) \varphi_2(-\mb{p}_t) - \varphi_1(\mb{q}_0) \varphi_2(-\mb{p}_0) \mid \mb{q}_0 = \mb{q}, \mb{p}_0 = \mb{p} \right]}{t} \rho_0(\mb{q}, \mb{p}) \, d\mb{q} \, d\mb{p}.
	\end{align*}
	Taking the limit as \(t \to 0^+\), by the definition of the generator for the process \(\mb{q}_t, \mb{p}_t\), we conclude that (i) implies
	\[
		\iint \varphi_1(\mb{q}) \varphi_2(\mb{p}) \mathcal{L} \left[ \psi_1(\mb{q}) \psi_2(\mb{p}) \right] \rho_0(\mb{q}, \mb{p}) \, d\mb{q} \, d\mb{p}
		= \iint \mathcal{L} \left[ \varphi_1(\mb{q}) \widetilde{\varphi}_2(\mb{p}) \right] \psi_1(\mb{q}) \widetilde{\psi}_2(\mb{p}) \rho_0(\mb{q}, \mb{p}) \, d\mb{q} \, d\mb{p},
	\]
	for any \(\varphi_1(\mb{q}), \varphi_2(\mb{p}), \psi_1(\mb{q}), \psi_2(\mb{p}) \in C_0^{\infty}(\mathbb{R}^d)\), where \(\tilde{\varphi}_2(\mb{p}) := \varphi_2(-\mb{p})\) and \(\tilde{\psi}_2(\mb{p}) := \psi_2(-\mb{p})\).

	As a consequence, taking \(\psi_1 \equiv \psi_2 \equiv 1\) in this equation, we have
	\[
		\iint \mathcal{L} \left[ \varphi_1(\mb{q}) \widetilde{\varphi}_2(\mb{p}) \right] \rho_0(\mb{q}, \mb{p}) \, d\mb{q} \, d\mb{p} = 0, \quad \forall \varphi_1, \varphi_2 \in C_0^{\infty}(\mathbb{R}^d),
	\]
	which implies that \(\rho_0(\mb{q}, \mb{p})\), satisfying \(\mathcal{L}^* \rho_0 = 0\), is an invariant measure. Therefore, we conclude that \((i) \implies (ii)\).

	\textbf{Step 2.} We prove \((ii)\) is equivalent to \((iii)\).

	Let \(H(\mb{q}, \mb{p}) = -\frac{1}{\beta} \ln \rho_0\). Then \((ii)\) is equivalent to
	\[
		\iint \varphi_1(\mb{q}) \varphi_2(\mb{p}) \mathcal{L} \left[ \psi_1(\mb{q}) \psi_2(\mb{p}) \right] e^{-\beta H(\mb{q}, \mb{p})} \, d\mb{q} \, d\mb{p}
		= \iint \mathcal{L} \left[ \varphi_1(\mb{q}) \widetilde{\varphi}_2(\mb{p}) \right] \psi_1(\mb{q}) \widetilde{\psi}_2(\mb{p}) e^{-\beta H(\mb{q}, \mb{p})} \, d\mb{q} \, d\mb{p},
	\]
	for any \(\varphi_1(\mb{q}), \varphi_2(\mb{p}), \psi_1(\mb{q}), \psi_2(\mb{p}) \in C_0^{\infty}(\mathbb{R}^d)\).
	Substituting the operator \(\mathcal{L}\) as defined in \eqref{eq:BULE} and using \(\widetilde{\varphi}_2(\mb{p}) = \varphi_2(-\mb{p})\) and \(\widetilde{\psi}_2(\mb{p}) = \psi_2(-\mb{p})\), we obtain 
	\[
		\begin{aligned}
			& \iint \varphi_1(\mb{q}) \varphi_2(\mb{p}) \Bigg[ \psi_2(\mb{p}) \mb{p} \cdot \nabla_{\mb{q}} \psi_1(\mb{q}) + \psi_1(\mb{q}) \mb{b} \cdot \nabla_{\mb{p}} \psi_2(\mb{p}) \\
			& \quad + \psi_1(\mb{q}) \left( - \dfrac{1}{2} (\bm{\sigma\sigma}^T \mb{p}) \cdot \nabla_{\mb{p}} \psi_2(\mb{p}) + \dfrac{1}{2\beta} (\bm{\sigma\sigma}^T) : \nabla^2_{\mb{p}} \psi_2(\mb{p}) \right) \Bigg] e^{-\beta H(\mb{q}, \mb{p})} \, d\mb{q} \, d\mb{p} \\
			&= \iint \psi_1(\mb{q}) \psi_2(\mb{p}) \Bigg[ - \varphi_2(\mb{p}) \mb{p} \cdot \nabla_{\mb{q}} \varphi_1(\mb{q}) - \varphi_1(\mb{q}) \mb{b} \cdot \nabla_{\mb{p}} \varphi_2(\mb{p}) \\
			& \quad + \varphi_1(\mb{q}) \left( - \dfrac{1}{2} (\bm{\sigma\sigma}^T \mb{p}) \cdot \nabla_{\mb{p}} \varphi_2(\mb{p}) + \dfrac{1}{2\beta} (\bm{\sigma\sigma}^T) : \nabla^2_{\mb{p}} \varphi_2(\mb{p}) \right) \Bigg] e^{-\beta H(\mb{q}, \mb{p})} \, d\mb{q} \, d\mb{p}.
		\end{aligned}
	\]

	We now separate the Hamiltonian flow terms on the LHS from the Fokker-Planck terms on the RHS. Then, we get 
	\[
		\begin{aligned}
			& I_1 := \iint \Big[ \varphi_2(\mb{p}) \psi_2(\mb{p}) \mb{p} \cdot \nabla_{\mb{q}} \left( \psi_1(\mb{q}) \varphi_1(\mb{q}) \right) + \varphi_1(\mb{q}) \psi_1(\mb{q}) \mb{b} \cdot \nabla_{\mb{p}} \left( \varphi_2(\mb{p}) \psi_2(\mb{p}) \right) \Big] e^{-\beta H(\mb{q}, \mb{p})} \, d\mb{q} \, d\mb{p} \\
			=& I_2 := \iint \Bigg[ \psi_1(\mb{q}) \varphi_1(\mb{q}) \Bigg( \dfrac{1}{2} \varphi_2(\mb{p}) (\bm{\sigma\sigma}^T \mb{p}) \cdot \nabla_{\mb{p}} \psi_2(\mb{p}) - \dfrac{1}{2\beta} \varphi_2(\mb{p}) (\bm{\sigma\sigma}^T) : \nabla^2_{\mb{p}} \psi_2(\mb{p}) \\
			& \qquad - \dfrac{1}{2} \psi_2(\mb{p}) (\bm{\sigma\sigma}^T \mb{p}) \cdot \nabla_{\mb{p}} \varphi_2(\mb{p}) + \dfrac{1}{2\beta} \psi_2(\mb{p}) (\bm{\sigma\sigma}^T) : \nabla^2_{\mb{p}} \varphi_2(\mb{p}) \Bigg) \Bigg] e^{-\beta H(\mb{q}, \mb{p})} \, d\mb{q} \, d\mb{p}.
		\end{aligned}
	\]
Using integration by parts, \( I_1 \) simplifies as 
\begin{align} \label{tm_i1}
I_1 = \iint \varphi_1(\mb{q}) \psi_1(\mb{q}) \varphi_2(\mb{p}) \psi_2(\mb{p}) 
\left( - \mb{p} \cdot \nabla_{\mb{q}} e^{-\beta H} - \mb{b} \cdot \nabla_{\mb{p}} e^{-\beta H} \right) d\mb{q} d\mb{p}.
\end{align}

Similarly, \( I_2 \) simplifies as 
\begin{equation} \label{tm_i2}
\begin{aligned}
I_2 &= \frac{1}{2} \iint \varphi_1(\mb{q}) \psi_1(\mb{q}) 
\Big[ e^{-\beta H} (\bm{\sigma\sigma}^T \mb{p}) \cdot 
\left( \varphi_2(\mb{p}) \nabla_{\mb{p}} \psi_2(\mb{p}) - \psi_2(\mb{p}) \nabla_{\mb{p}} \varphi_2(\mb{p}) \right) \\
&\quad + \frac{1}{\beta} \nabla_{\mb{p}} \left( \varphi_2(\mb{p}) e^{-\beta H} \right) (\bm{\sigma\sigma}^T) \nabla_{\mb{p}} \psi_2(\mb{p}) \\
&\quad - \frac{1}{\beta} \nabla_{\mb{p}} \left( \psi_2(\mb{p}) e^{-\beta H} \right) (\bm{\sigma\sigma}^T) \nabla_{\mb{p}} \varphi_2(\mb{p}) 
\Big] d\mb{q} d\mb{p} \\
&= \frac{1}{2} \iint \varphi_1(\mb{q}) \psi_1(\mb{q}) 
\left[ \left( e^{-\beta H} (\bm{\sigma\sigma}^T \mb{p}) + \frac{1}{\beta} \bm{\sigma\sigma}^T \nabla_{\mb{p}} e^{-\beta H} \right) 
\cdot \left( \varphi_2(\mb{p}) \nabla_{\mb{p}} \psi_2(\mb{p}) - \psi_2(\mb{p}) \nabla_{\mb{p}} \varphi_2(\mb{p}) \right) 
\right] d\mb{q} d\mb{p}.
\end{aligned}
\end{equation}

Now, let \( \varphi_2 = \psi_2 \), then \( I_2 = 0 \), and from \eqref{tm_i1}, we obtain 
\begin{align*}
I_1 = \iint \varphi_1(\mb{q}) \psi_1(\mb{q}) \varphi_2^2(\mb{p}) 
\left( - \mb{p} \cdot \nabla_{\mb{q}} e^{-\beta H} - \mb{b} \cdot \nabla_{\mb{p}} e^{-\beta H} \right) d\mb{q} d\mb{p} = 0
\end{align*}
for any \( \varphi_1(\mb{q}), \psi_1(\mb{q}) \), and \( \varphi_2(\mb{p}) \). Thus, we deduce 
\begin{align} \label{eq:sv_1}
- \mb{p} \cdot \nabla_{\mb{q}} e^{-\beta H} - \mb{b} \cdot \nabla_{\mb{p}} e^{-\beta H} = 0.
\end{align}

This implies that \( I_1 = I_2 = 0 \) for arbitrary test functions \( \varphi_1(\mb{p}), \varphi_2(\mb{p}), \psi_1(\mb{q}) \), and \( \psi_2(\mb{q}) \). By the auxiliary lemma, equation \eqref{tm_i2} yields 
\begin{align*}
e^{-\beta H} (\bm{\sigma\sigma}^T \mb{p}) + \frac{1}{\beta} \bm{\sigma\sigma}^T \nabla_{\mb{p}} e^{-\beta H} = 0.
\end{align*}
Since \( \bm{\sigma} \) is nonsingular, we conclude 
\begin{align} \label{eq:sv_2}
e^{-\beta H} \mb{p} + \frac{1}{\beta} \nabla_{\mb{p}} e^{-\beta H} = 0
\end{align}
holds for any \( \mb{p} \).

Next, we use \eqref{eq:sv_1} and \eqref{eq:sv_2} to derive \( H(\mb{q}, \mb{p}) = \frac{|\mb{p}|^2}{2} + U(\mb{q}) \) for some potential \( U(\mb{q}) \). Equation \eqref{eq:sv_2} implies 
\[
\nabla_{\mb{p}} e^{-\beta H + \beta \frac{|\mb{p}|^2}{2}} = 0,
\]
and thus there exists a potential \( U(\mb{q}) \) such that
\begin{equation} \label{eq:sv_p}
e^{-\beta H} = e^{-\beta \left( \frac{|\mb{p}|^2}{2} + U(\mb{q}) \right)}.
\end{equation}
Substituting \eqref{eq:sv_p} into \eqref{eq:sv_1}, we derive 
\begin{equation}
- \mb{p} \cdot \nabla_{\mb{q}} e^{-\beta H} - \mb{b} \cdot \nabla_{\mb{p}} e^{-\beta H} = 0,
\end{equation}
which simplifies to
\begin{equation}
\mb{p} \cdot \left( \mb{b} + \nabla_{\mb{q}} U(\mb{q}) \right) = 0, \quad \forall \, \mb{p}, \mb{q}.
\end{equation}
Therefore, we conclude 
\[
\mb{b} = - \nabla_{\mb{q}} U(\mb{q}),
\]
and \( U(\mb{q}) \) is unique up to a constant. Thus, the invariant measure is given by 
\begin{equation}
\rho_0(\mb{q}, \mb{p}) = \frac{1}{Z} e^{-\beta \left( \frac{|\mb{p}|^2}{2} + U(\mb{q}) \right)}, \quad Z = \iint e^{-\beta \left( \frac{|\mb{p}|^2}{2} + U(\mb{q}) \right)} d\mb{q} d\mb{p},
\end{equation}
which is the unique invariant measure, and hence (ii) is equivalent to (iii).

\textbf{Step 3.} We prove that (iii) implies (i).

First, from (iii) and for $\sL$ defined in \eqref{eq:BULE} with $\mb{b} = -\nabla_{\mb{q}} U(\mb{q})$, it is easy to verify \eqref{eq:thm2_ii}. Now, we use \eqref{eq:thm2_ii} and (iii) to derive (i).

From Lemma \ref{Lmm}, $-\sL$ is a maximal monotone operator in $L^2(\rho_0 \, d\mb{q}d\mb{p})$, where
\[
\rho_0(\mb{q}, \mb{p}) = \dfrac{1}{Z} e^{-\beta \left( \frac{|\mb{p}|^2}{2} + U(\mb{q}) \right)}.
\]
Recall the strongly continuous semigroup generated by $-\sL$, denoted by \( S(t) = e^{t\sL} \).

From \cite[p381, Cor 7.3.2]{buhler2018functional}, we know that the dual semigroup \( \widetilde{S}(t) = \widetilde{(e^{t\sL})} \) is generated by the dual operator \( -\widetilde{\sL} \), which is also a maximal monotone operator. Here, the dual operators are understood in \( L^2(\rho_0 \, d\mb{q}d\mb{p}) \). Therefore, we have 
\begin{align*}
\iint & \varphi_1(\mb{q}) \varphi_2(\mb{p}) e^{t\sL} \left( \psi_1(\mb{q}) \psi_2(\mb{p}) \right) \rho_0(\mb{q}, \mb{p}) \, d\mb{q} \, d\mb{p} \\
=& \iint \widetilde{(e^{t\sL})} \left( \varphi_1(\mb{q}) \varphi_2(\mb{p}) \right) \psi_1(\mb{q}) \psi_2(\mb{p}) \rho_0(\mb{q}, \mb{p}) \, d\mb{q} \, d\mb{p} \\
=& \iint \left( e^{t \widetilde{\sL}} \right) \left( \varphi_1(\mb{q}) \varphi_2(\mb{p}) \right) \psi_1(\mb{q}) \psi_2(\mb{p}) \rho_0(\mb{q}, \mb{p}) \, d\mb{q} \, d\mb{p},
\end{align*}
where \( \left( e^{t \widetilde{\sL}} \right) \left( \varphi_1(\mb{q}) \varphi_2(\mb{p}) \right) \) is the solution to
\begin{equation}
\frac{\partial}{\partial t} f = \widetilde{\sL} f := - \left( \mb{p} \cdot \nabla_{\mb{q}} f + \mb{b} \cdot \nabla_{\mb{p}} f \right) + \left( - \frac{1}{2} \left( \bm{\sigma\sigma}^T \mb{p} \right) \cdot \nabla_{\mb{p}} f + \frac{1}{2 \beta} \left( \bm{\sigma\sigma}^T \right) : \nabla^2_{\mb{p}} f \right).
\end{equation}
By changing variables from \( \mb{p} \) to \( -\mb{p} \), it is easy to verify 
\begin{equation}
\left( e^{t \widetilde{\sL}} \right) \left( \varphi_1(\mb{q}) \varphi_2(\mb{p}) \right) = \left( e^{t\sL} \right) \left( \varphi_1(\mb{q}) \varphi_2(-\mb{p}) \right) \Big|_{(\mb{q}, -\mb{p})} = \left( e^{t\sL} \right) \left( \varphi_1(\mb{q}) \tilde{\varphi}_2(\mb{p}) \right) \Big|_{(\mb{q}, -\mb{p})}.
\end{equation}
Thus, we have 
\begin{align*}
\iint &\varphi_1(\mb{q}) \varphi_2(\mb{p}) e^{t\sL} \left( \psi_1(\mb{q}) \psi_2(\mb{p}) \right) \rho_0(\mb{q}, \mb{p}) \, d\mb{q} \, d\mb{p} \\
=& \iint \left( e^{t\sL} \right) \left( \varphi_1(\mb{q}) \tilde{\varphi}_2(\mb{p}) \right) \Big|_{(\mb{q}, -\mb{p})} \psi_1(\mb{q}) \psi_2(\mb{p}) \rho_0(\mb{q}, \mb{p}) \, d\mb{q} \, d\mb{p} \\
=&   \iint (e^{t\sL})\bbs{\varphi_1(\mb q)  \tilde{\varphi}_2(\mb p)}  \psi_1(\mb q)  {\psi}_2(-\mb p)  \rho_0(\mb q, -\mb p) \ud \mb q \ud  \mb p\\
=& \iint \left( e^{t\sL} \right) \left( \varphi_1(\mb{q}) \tilde{\varphi}_2(\mb{p}) \right) \psi_1(\mb{q}) \tilde{\psi}_2(\mb{p}) \rho_0(\mb{q}, \mb{p}) \, d\mb{q} \, d\mb{p},
\end{align*}
and we conclude (i).

\textbf{Step 4.} We prove that (iii) and (iv) are equivalent.

It is obvious that (iii) implies (iv) since \( H(\mb{q}, \mb{p}) = H(\mb{q}, -\mb{p}) \), and the Gibbs measure solves \eqref{eq:FPOLE}. Thus, we just need to prove that (iv) implies (iii). Let \( \widetilde{\rho_0}(\mb{q}, \mb{p}) = \rho_0(\mb{q}, -\mb{p}) \). Since \( \rho(\mb{q}, \mb{p}) = \rho(\mb{q}, -\mb{p}) \), they both solve \eqref{eq:FPOLE}. Thus 
\begin{equation} \label{eq:solve}
\begin{aligned}
- \mb{p} \cdot \nabla_{\mb{q}} \rho_0 - \mb{b} \cdot \nabla_{\mb{p}} \rho_0 + \dfrac{1}{2} \nabla_{\mb{p}} \cdot \left( \rho_0 \bm{\sigma\sigma}^T \mb{p} + \frac{1}{\beta} \bm{\sigma\sigma}^T \nabla_{\mb{p}} \rho_0 \right) &= 0, \\
- \mb{p} \cdot \nabla_{\mb{q}} \widetilde{\rho_0} - \mb{b} \cdot \nabla_{\mb{p}} \widetilde{\rho_0} + \dfrac{1}{2} \nabla_{\mb{p}} \cdot \left( \widetilde{\rho_0} \bm{\sigma\sigma}^T \mb{p} + \frac{1}{\beta} \bm{\sigma\sigma}^T \nabla_{\mb{p}} \widetilde{\rho_0} \right) &= 0.
\end{aligned}
\end{equation}

Substituting \( \widetilde{\rho_0}(\mb{q}, \mb{p}) = \rho_0(\mb{q}, -\mb{p}) \) into \eqref{eq:solve}, we get 
\begin{equation}
\begin{aligned}
- \mb{p} \cdot \nabla_{\mb{q}} \rho_0 - \mb{b} \cdot \nabla_{\mb{p}} \rho_0 + \dfrac{1}{2} \nabla_{\mb{p}} \cdot \left( \rho_0 \bm{\sigma\sigma}^T \mb{p} + \frac{1}{\beta} \bm{\sigma\sigma}^T \nabla_{\mb{p}} \rho_0 \right) &= 0, \\
\mb{p} \cdot \nabla_{\mb{q}} \rho_0 + \mb{b} \cdot \nabla_{\mb{p}} \rho_0 + \dfrac{1}{2} \nabla_{\mb{p}} \cdot \left( \rho_0 \bm{\sigma\sigma}^T \mb{p} + \frac{1}{\beta} \bm{\sigma\sigma}^T \nabla_{\mb{p}} \rho_0 \right) &= 0,
\end{aligned}
\end{equation}
which implies 
\begin{align} \label{eq:solve2}
\mb{p} \cdot \nabla_{\mb{q}} \rho_0 + \mb{b} \cdot \nabla_{\mb{p}} \rho_0 = 0, \quad \nabla_{\mb{p}} \cdot \left( \rho_0 \bm{\sigma\sigma}^T \mb{p} + \frac{1}{\beta} \bm{\sigma\sigma}^T \nabla_{\mb{p}} \rho_0 \right) = 0.
\end{align}

Since \( \rho(\mb{q}, \mb{p}) > 0 \) and \( \rho_0(\mb{q}, \mb{p}) \in L^1(\R^{2d}) \), by Fubini's theorem, for almost every \( \mb{q} \in \R^d \), \( m(\mb{q}) := \int_{\R^d} \rho_0(\mb{q}, \mb{p}) \, d\mb{p} \in (0, \infty) \) exists, and \( m(\mb{q}) \in L^1(\R^d) \). Now, for these \( \mb{q} \)'s, from \eqref{eq:solve2}, \( \rho_0(\mb{q}, \mb{p}) \) solves the following Fokker-Planck equation (in the \( \mb{p} \)-variable) 
\begin{align*}
\nabla_{\mb{p}} \cdot \left( \rho_0 \bm{\sigma\sigma}^T \mb{p} + \frac{1}{\beta} \bm{\sigma\sigma}^T \nabla_{\mb{p}} \rho_0 \right) = 0,
\end{align*}
and \( \rho_0(\mb{q}, \mb{p}) \) is a positive \( L^1 \)-solution (in the \( \mb{p} \)-variable). Thus, by Proposition \ref{prop:uniqueness}, we know that \( e^{-\beta |\mb{p}|^2 / 2} \) is the unique non-zero solution (up to a constant). Hence, there exists \( c(\mb{q}) > 0 \) such that 
\begin{align} \label{rho_Q}
\rho_0 = c(\mb{q}) e^{-\beta |\mb{p}|^2 / 2}.
\end{align}
Substituting the above equality into \eqref{eq:solve2}, we get 
\begin{align*}
e^{-\beta |\mb{p}|^2 / 2} \mb{p} \cdot \left( \nabla_{\mb{q}} c(\mb{q}) - c \beta \mb{b} \right) = 0.
\end{align*}
This holds for arbitrary \( \mb{p} \), so \( \mb{b} = -\nabla_{\mb{q}} \log(c) / \beta \), and (iii) is proved.

\textbf{Step 5.} We prove that (v) and (iii) are equivalent. 

It is clear that (iii) implies (v), so we only need to prove that (v) implies (iii). First, we have 
\begin{align*}
    -U_2(\mb{p}) \, \mb{p} \cdot \nabla_{\mb{q}} U_1(\mb{q}) - U_1(\mb{q}) \, \mb{b} \cdot \nabla_{\mb{p}} U_2(\mb{p}) + \dfrac{U_1(\mb{q})}{2} \nabla_{\mb{p}} \cdot \left( U_2(\mb{p}) \bm{\sigma\sigma}^T \mb{p} + \frac{1}{\beta} \bm{\sigma\sigma}^T \nabla_{\mb{p}} U_2(\mb{p}) \right) = 0.
\end{align*} 
Dividing by \( U_1(\mb{q}) U_2(\mb{p}) \), we get 
\begin{align} \label{eq:help4}
    - \mb{p} \cdot \nabla_{\mb{q}} \log U_1(\mb{q}) - \mb{b} \cdot \nabla_{\mb{p}} \log U_2(\mb{p}) + \dfrac{1}{2 U_2(\mb{p})} \nabla_{\mb{p}} \cdot \left( U_2(\mb{p}) \bm{\sigma\sigma}^T \mb{p} + \frac{1}{\beta} \bm{\sigma\sigma}^T \nabla_{\mb{p}} U_2(\mb{p}) \right) = 0.
\end{align}
Notice that the last term only depends on \( \mb{p} \). Now, evaluating \eqref{eq:help4} at \( (\mb{q}_1, \mb{p}) \) and \( (\mb{q}_2, \mb{p}) \) and taking the difference yields 
\begin{align*}
    \mb{p} \cdot \nabla_{\mb{q}} [\log U_1(\mb{q}_1) - \log U_1(\mb{q}_2)] + [\mb{b}(\mb{q}_1) - \mb{b}(\mb{q}_2)] \cdot \nabla_{\mb{p}} \log U_2(\mb{p}) = 0.
\end{align*}
This holds for any \( \mb{p}, \mb{q}_1, \mb{q}_2 \in \R^d \). Taking the derivative with respect to \( \mb{p} \), we get 
\begin{align} \label{eq:help6}
    \nabla_{\mb{q}} [\log U_1(\mb{q}_1) - \log U_1(\mb{q}_2)] + \nabla^2_{\mb{p}} \log U_2(\mb{p}) [\mb{b}(\mb{q}_1) - \mb{b}(\mb{q}_2)] = 0.
\end{align}
Notice that the first term only depends on \( \mb{q}_1 \) and \( \mb{q}_2 \), and so does the second term. Thus, for any \( \mb{p}_1 \) and \( \mb{p}_2 \), we have 
\begin{align} \label{eq:help5}
    [\nabla^2_{\mb{p}} \log U_2(\mb{p}_1) - \nabla^2_{\mb{p}} \log U_2(\mb{p}_2)] [\mb{b}(\mb{q}_1) - \mb{b}(\mb{q}_2)] = \mb{0}.
\end{align}
Equation \eqref{eq:help5} implies that, for any \( \mb{p}_1 \) and \( \mb{p}_2 \), we must have 
\begin{align*}
    \nabla^2_{\mb{p}} \log U_2(\mb{p}_1) = \nabla^2_{\mb{p}} \log U_2(\mb{p}_2).
\end{align*} 
Thus, \( \nabla^2_{\mb{p}} \log U_2(\mb{p}) \) is a constant, meaning there exist \( \mb{A} \in \R^{d \times d}, \mb{a} \in \R^d \), and \( c \in \R \), all constant, such that 
\begin{align} \label{eq:help8}
    \log U_2(\mb{p}) = \mb{p}^T \mb{A} \mb{p} + \mb{a} \cdot \mb{p} + c.
\end{align}
We assume that \( \mb{A} \) is symmetric. Substituting this back into \eqref{eq:help6}, we get 
\begin{align*}
    \nabla_{\mb{q}} [\log U_1(\mb{q}_1) - \log U_1(\mb{q}_2)] + 2 \mb{A} [\mb{b}(\mb{q}_1) - \mb{b}(\mb{q}_2)] = 0,
\end{align*}
which holds for arbitrary \( \mb{q}_1 \) and \( \mb{q}_2 \). Thus, there exists a constant \( \mb{c}' \) such that, for any \( \mb{q} \in \R^d \), we have 
\begin{align} \label{eq:help7}
    \nabla_{\mb{q}} \log U_1(\mb{q}) + 2 \mb{A} \mb{b} = \mb{c}'.
\end{align}
Substituting \eqref{eq:help7} and \eqref{eq:help8} back into \eqref{eq:help4}, we get 
\begin{align*}
    -\mb{p} \cdot \mb{c}' - \mb{b} \cdot \mb{a} + \dfrac{1}{2} \left[ (2 \mb{A} \mb{p} + \mb{a})^T \bm{\sigma\sigma}^T \mb{p} + \mathrm{tr}(\bm{\sigma\sigma}^T) + \dfrac{1}{\beta} (2 \mb{A} \mb{p} + \mb{a})^T \bm{\sigma\sigma} (2 \mb{A} \mb{p} + \mb{a}) + \dfrac{2}{\beta} \mathrm{tr}(\bm{\sigma\sigma}^T \mb{A}) \right] = 0.
\end{align*}
Here, \( \mb{b} \) is the only term that depends on \( \mb{q} \). Thus, from the previous argument, we know that \( \mb{a} = 0 \). The above equation holds for all \( \mb{p} \in \R^d \), so the coefficients are all zero. Hence 
\begin{align*}
    \mb{c}' = \mb{0}, \quad \mb{A}^T \bm{\sigma\sigma}^T + \dfrac{2}{\beta} \mb{A}^T \bm{\sigma\sigma}^T \mb{A} = \mb{0}, \quad \mathrm{tr}(\bm{\sigma\sigma}^T (2 \mb{A} / \beta + \mb{I}_d)) = 0.
\end{align*}
Since \( \mb{A} = \mb{A}^T \), the second equality implies that \( \mb{A}^T \bm{\sigma\sigma}^T = -2 \mb{A}^T \bm{\sigma \sigma}^T \mb{A} / \beta \) is symmetric, so \( \mb{A} \) and \( \bm{\sigma\sigma}^T \) are commutative. Thus 
\begin{align} \label{tmAA}
    \bm{\sigma\sigma}^T (\mb{A} + 2 \mb{A}^2 / \beta) = 0.
\end{align}
Since \( \bm{\sigma\sigma}^T \) is invertible, it follows that \( \mb{A} + 2 \mb{A}^2 / \beta = \mb{0} \). The eigenvalues of \( \mb{A} \) are either 0 or \( -\beta / 2 \). Therefore, the eigenvalues of \( (2 \mb{A} / \beta + \mb{I}_d) \) are either 0 or 1. Suppose the eigenvectors of \( (2 \mb{A} / \beta + \mb{I}_d) \) are \( \mb{u}_i \), \( i = 1, 2, \dots, d \), satisfying \( \mb{u}_i \cdot \mb{u}_j = \delta_{ij} \), with corresponding eigenvalues \( \lambda_i \), \( i = 1, 2, \dots, d \). Then \( \lambda_i = 0 \) or 1, and 
\begin{align*}
    0 = \mathrm{tr}(\bm{\sigma\sigma}^T (2 \mb{A} / \beta + \mb{I}_d)) = \sum_{i=1}^d \mb{u}_i^T \bm{\sigma\sigma}^T (2 \mb{A} / \beta + \mb{I}_d) \mb{u}_i = \sum_{i=1}^d \lambda_i \mb{u}_i^T \bm{\sigma\sigma}^T \mb{u}_i \geq 0.
\end{align*}  
Thus, \( \lambda_i = 0 \) for all \( i = 1, 2, \dots, d \), since \( \bm{\sigma} \) is nonsingular. Hence, the only eigenvalue of \( \mb{A} \) is \( -\beta / 2 \), and therefore \( \mb{A} = -\beta \mb{I}_d / 2 \). Since \( \mb{c}' = \mb{0} \) from \eqref{eq:help7}, we have \( \mb{b} = \dfrac{1}{\beta} \nabla_{\mb{q}} \log U_1(\mb{q}) \), which is a gradient. Thus, (iii) is proved.
\end{proof}

\begin{lemma}\label{Lmm}
    Let \( \mb{b} = - \nabla_{\mb{q}} U(\mb{q}) \) in \eqref{eq:BULE}. Consider the operator 
    \begin{equation}
        \mathcal{L}f = \left( -\mb{p} \cdot \nabla_{\mb{q}} f - \mb{b} \cdot \nabla_{\mb{p}} f \right) + \left( -\frac{1}{2} (\bm{\sigma\sigma}^T \mb{p}) \cdot \nabla_{\mb{p}} f + \frac{1}{2 \beta} (\bm{\sigma\sigma}^T) : \nabla^2_{\mb{p}} f \right) 
        =: T f + L_v f.
    \end{equation}
    Then \( -\mathcal{L} \) is a maximal monotone operator in \( L^2({\rho_0} \, \dd \mb{q} \dd \mb{p}) \) with the weight \( \rho_0(\mb{q}, \mb{p}) \) defined in \eqref{eq:req}.
\end{lemma}
\begin{proof}
    First, observe that 
    \begin{equation}
        \langle -T f, f \rangle_{L^2({\rho_0} \, \dd \mb{q} \dd \mb{p})} = 0, \quad \langle -L_v f, f \rangle_{L^2({\rho_0} \, \dd \mb{q} \dd \mb{p})} \geq 0,
    \end{equation}
    which shows that both \( -T \) and \( -L_v \) are monotone in \( L^2({\rho_0} \, \dd \mb{q} \dd \mb{p}) \).

    Second, we prove that \( \text{Ran}(I - L_v) = L^2({\rho_0} \, \dd \mb{q} \dd \mb{p}) \), i.e., given any \( h \in L^2({\rho_0} \, \dd \mb{q} \dd \mb{p}) \), there exists a solution \( f \in D(-L_v) \subset L^2({\rho_0} \, \dd \mb{q} \dd \mb{p}) \) such that \( (I - L_v) f = h \).
    Indeed, define a weighted Hilbert space \( H^1(e^{-\frac{1}{2} |\mb{p}|^2} \dd \mb{p}) \) with the norm 
    \[
    \| f \|_1 := \int \rho_0 f^2 \, \dd \mb{p} + \int \rho_0 |\nabla f|^2 \, \dd \mb{p},
    \]
    and a bilinear form on \( H^1(e^{-\frac{1}{2} |\mb{p}|^2} \dd \mb{p}) \times H^1(e^{-\frac{1}{2} |\mb{p}|^2} \dd \mb{p}) \):
    \[
    a(u,v) := \langle (I - L_v) u, v \rangle_{L^2(\rho_0)}.
    \]
    It is clear that \( a(u,v) \) is a coercive and bounded bilinear form on \( H^1(e^{-\frac{1}{2} |\mb{p}|^2} \dd \mb{p}) \times H^1(e^{-\frac{1}{2} |\mb{p}|^2} \dd \mb{p}) \). Therefore, for any \( \mb{q} \), we conclude the existence of \( (I - L_v) f(\mb{q}, \cdot) = h(\mb{q}, \cdot) \) by the Lax-Milgram theorem.

    Third, we show that the sum of \( -T \) and \( -L_v \) is a maximal monotone operator by using \cite[Cor 2.6]{Barbu_2010}. From the above, we have already concluded that \( -L_v \) is a maximal monotone operator in \( L^2({\rho_0} \, \dd \mb{q} \dd \mb{p}) \) and \( -T \) is monotone. It remains to verify that \( -T \) is demicontinuous (see \cite{Kato_1964}), i.e., if \( f_n \in D(-T) \) converges strongly to \( f \in D(-T) \) in \( L^2(\rho_0 \, \dd \mb{q} \dd \mb{p}) \), then \( -T f_n \) converges weak* to \( -T f \) in \( (L^2({\rho_0} \, \dd \mb{q} \dd \mb{p}))^* \). It is clear that \( -T \) is a linear operator in \( L^2({\rho_0} \, \dd \mb{q} \dd \mb{p}) \) and thus demicontinuous.
\end{proof}

\begin{remark}
    Define the Hamiltonian \( \mathcal{H}(\rho, \xi) := \langle e^{-\xi} \mathcal{L} e^{\xi}, \rho \rangle \). Then the symmetry condition (ii) for \( \mathcal{L} \) is equivalent to 
    \begin{equation}\label{eq:Hcondition}
        \mathcal{H}(\rho,\xi) = \mathcal{H}(\rho, \log\frac{\rho}{\rho_0} - \xi), \quad \forall \xi, \rho \in C_c^\infty(\R^d).
    \end{equation}
        Indeed, if \( \mathcal{L} \) is symmetric with respect to \( \rho_0 \), then 
    \begin{align*}
        \mathcal{H}(\rho, \log\frac{\rho}{\rho_0} -\xi) &= \left\langle \frac{\rho_0}{\rho} e^{\xi} \mathcal{L}\left(\frac{\rho}{\rho_0} e^{-\xi}\right) , \rho \right\rangle \\
        &= \left\langle e^{\xi}, \mathcal{L}\left(\frac{\rho}{\rho_0} e^{-\xi}\right) \right\rangle_{\rho_0} = \langle \mathcal{L} e^{\xi}, \rho e^{-\xi} \rangle = \mathcal{H}(\rho,\xi).
    \end{align*}
    If \eqref{eq:Hcondition} holds, let 
    \[
    \xi = \log(\phi_2), \quad \rho = \phi_1 \phi_2 \rho_0,
    \]
    then:
    \[
    \mathcal{H}(\rho, \log\dfrac{\rho}{\rho_0} - \xi) = \langle \phi_2, \mathcal{L} \phi_1 \rangle_{\rho_0}, \quad \mathcal{H}(\rho, \xi) = \langle \phi_1, \mathcal{L} \phi_2 \rangle_{\rho_0}.
    \]
    Therefore, \eqref{eq:Hcondition} yields the symmetry of \( \mathcal{L} \) with respect to \( \langle \cdot, \cdot \rangle_{\rho_0} \).
\end{remark}

\begin{proof}[Proof of Lemma \ref{prop:contraction}]
    We first prove part (i). To start, we show that \( -\mathcal{L^*_{\eps}} \) is maximal accretive. Define
    \[
    H := \left\{ u \mid \frac{u}{\rho_{\infty}^{\eps}} \in H^1(\rho_{\infty}^{\eps}) \right\},
    \]
    and equip \( H \) with the following inner product 
    \[
    \langle u,v \rangle_H = \left\langle \frac{u}{\rho_{\infty}^{\eps}}, \frac{v}{\rho_{\infty}^{\eps}} \right\rangle_{H^1(\rho_{\infty}^{\eps})}.
    \]
    Clearly, \( H \subset L^2(1/\rho_{\infty}^{\eps}) \), and \( H \) is a Hilbert space. Let the domain of \( \mathcal{L^*_{\eps}} \) be 
    \[
    D(\mathcal{L^*_{\eps}}) := \{ u \mid u \in H, \ \mathcal{L^*_{\eps}} u \in L^2(1/\rho_\infty^\eps) \} \subset L^2(1/\rho_\infty^\eps).
    \]
    Define \( \mb{b} = -\nabla V + \eps \mb{M} \). Then, for any \( u \in D(\mathcal{L^*_{\eps}}) \), we have
    \begin{equation} \label{eq:calc}
    \begin{aligned}
        \langle u, \mathcal{L^*_{\eps}} u \rangle_{L^2(1/\rho_{\infty}^{\eps})}
        &= \int_{\R^d} \frac{u}{\rho_{\infty}^{\eps}} \mathcal{L^*_{\eps}} u \, \dd \mb{q} \\
        &= \int_{\R^d} \frac{u}{\rho_{\infty}^{\eps}} \left[ \nabla \cdot \left( \bm{\sigma} \bm{\sigma}^T \left( -\mb{b} \rho_\infty^\eps \frac{u}{\rho_\infty^\eps} + \nabla \left( \rho_\infty^\eps \cdot \frac{u}{\rho_\infty^\eps} \right) \right) \right) \right] \dd \mb{q} \\
        &= \int_{\R^d} \frac{u}{\rho_{\infty}^{\eps}} \left[ \nabla \cdot \left( \frac{u}{\rho_\infty^\eps} \bm{\sigma} \bm{\sigma}^T \left( -\mb{b} \rho_\infty^\eps + \nabla \rho_\infty^\eps \right) \right) \right] \dd \mb{q} \\
        & \quad + \int_{\R^d} \frac{u}{\rho_{\infty}^{\eps}} \left[ \nabla \cdot \left( \bm{\sigma} \bm{\sigma}^T \rho_\infty^\eps \nabla \left( \frac{u}{\rho_\infty^\eps} \right) \right) \right] \dd \mb{q} \\
        &= -\int_{\R^d} \rho_\infty^\eps \left| \bm{\sigma}^T \nabla \left( \frac{u}{\rho_{\infty}^{\eps}} \right) \right|^2 \dd \mb{q} - \frac{1}{2} \int_{\R^d} \left| \frac{u}{\rho_{\infty}^{\eps}} \right|^2 \mathcal{L^*_{\eps}} \rho_\infty^\eps \dd \mb{q} \\
        &= -\int_{\R^d} \rho_\infty^\eps \left| \bm{\sigma}^T \nabla \left( \frac{u}{\rho_{\infty}^{\eps}} \right) \right|^2 \dd \mb{q} \leq 0.
    \end{aligned}
    \end{equation}
    Thus, \( -\mathcal{L^*_{\eps}} \) is accretive.

    Next, we prove that \( R(I - \mathcal{L^*_{\eps}}) = L^2(1/\rho_\infty^\eps) \). This can be derived by Lax-Milgram's theorem. For arbitrary \( f \in L^2(1/\rho_\infty^\eps) \), define the linear functional 
    \[
    f(u) \colon L^2(1/\rho_\infty^\eps) \to \R, \quad u \mapsto f(u) := \langle u, f \rangle_{L^2(1/\rho_\infty^\eps)},
    \]
    which is continuous on \( H \). Now, consider the bilinear form:
    \[
    a(u,v) \colon H \times H \to \R, \quad (u,v) \mapsto a(u,v) := \langle (I - \mathcal{L^*_{\eps}}) u, v \rangle_{L^2(1/\rho_\infty^\eps)}.
    \]
    By direct calculation and Cauchy-Schwarz inequality, we obtain 
    \begin{align*}
    |a(u,v)| 
&= \left|\int_{\R^d}uv\dfrac{1}{\rho_\infty^\eps}\dd\mb q-\int_{\R^d}\dfrac{1}{\rho_\infty^\eps}u\mathcal{L^*_{\eps}}v\dd\mb q\right|\\
	&  \leq \| u \|_{L^2(1/\rho_\infty^\eps)} \| v \|_{L^2(1/\rho_\infty^\eps)} + \sqrt{\int_{\R^d} \rho_\infty^\eps \left| \bm{\sigma}^T \nabla \left( \frac{u}{\rho_{\infty}^{\eps}} \right) \right|^2 \dd \mb{q}} \cdot \sqrt{\int_{\R^d} \rho_\infty^\eps \left| \bm{\sigma}^T \nabla \left( \frac{v}{\rho_{\infty}^{\eps}} \right) \right|^2 \dd \mb{q}}.
    \end{align*}
    Therefore, we have 
    \[
    |a(u,v)| \leq c \| u \|_H \| v \|_H,
    \]
    where \( c > 0 \) is a constant depending on \( \bm{\sigma} \). Meanwhile,
    \[
    a(u,u) = \| u \|_{L^2(1/\rho_\infty^\eps)}^2 + \int_{\R^d} \rho_\infty^\eps \left| \bm{\sigma}^T \nabla \left( \frac{u}{\rho_{\infty}^{\eps}} \right) \right|^2 \dd \mb{q} \geq c' \| u \|_H^2,
    \]
    where \( c' > 0 \) is a constant depending on \( \bm{\sigma} \). Therefore, \( a \) is bounded and coercive on \( H \). By Lax-Milgram's theorem, there exists \( u \in H \) such that \( a(u,v) = f(v) \) holds for all \( v \in H \). Hence, \( (I - \mathcal{L^*_{\eps}}) u = f \), and thus \( R(I - \mathcal{L^*_{\eps}}) = L^2(1/\rho_\infty^\eps) \).

    Finally, by the Hille-Yosida theorem, we conclude that \( \mathcal{L^*_{\eps}} \) generates a strongly continuous semigroup of contractions in \( L^2(1/\rho_\infty^\eps) \).

    For part (ii), by \cite{buhler2018functional}, equation \eqref{eq:FPOLE} admits a unique solution \( \rho^{\varepsilon}(\mb{q},t) \in C^1([0,T],L^2(1/\rho_\infty^\varepsilon)) \) for any \( T > 0 \).
\end{proof}

\begin{proof}[Proof of Theorem \ref{thm:moR}]
    Notice that the structure of the generator in \eqref{eq:BGLE} is quite similar to the generator in \eqref{eq:BULE}, with the exception that only \( \mb{p} \) is an odd variable.

    \textbf{Step 1.} \( (i) \Rightarrow (ii) \) is exactly the same as Step 1 in the proof of Theorem \ref{thm:revULE}.

    \textbf{Step 2.} We prove \( (ii) \Rightarrow (iii) \).
    From the definition of the generator \( \sL \) in \eqref{eq:BGLE}, \eqref{symLm} implies
    \begin{equation}
    \begin{aligned}
        \iiint \Big[&-\varphi_2(\mb{p}) \psi_2(\mb{p}) \varphi_3(\mb{z}) \psi_3(\mb{z}) \, \mb{p} \cdot \nabla_{\mb{q}} \bbs{\psi_1(\mb{q}) \varphi_1(\mb{q})} \\
        &- \varphi_1(\mb{q}) \psi_1(\mb{q}) \varphi_3(\mb{z}) \psi_3(\mb{z}) (\mb{b} + \mb{z}) \cdot \nabla_{\mb{p}} \bbs{\varphi_2(\mb{p}) \psi_2(\mb{p})} \\
        &- \varphi_1(\mb{q}) \psi_1(\mb{q}) \varphi_2(\mb{p}) \psi_2(\mb{p}) \, \mb{p} \cdot \nabla_{\mb{z}} \bbs{\varphi_3(\mb{z}) \psi_3(\mb{z})} \Big] e^{-\beta H(\mb{q}, \mb{p}, \mb{z})} \, \dqpz \\
        = & \iiint \Bigg[ \psi_1(\mb{q}) \varphi_1(\mb{q}) \varphi_2(\mb{p}) \psi_2(\mb{p}) \alpha \Big( \psi_3(\mb{z}) \, \mb{z} \cdot \nabla_{\mb{z}} \varphi_3(\mb{z}) \\
        &- \varphi_3(\mb{z}) \, \mb{z} \cdot \nabla_{\mb{z}} \psi_3(\mb{z}) + \frac{1}{\beta} \bbs{\varphi_3(\mb{z}) \Delta_{\mb{z}} \psi_3(\mb{z}) - \psi_3(\mb{z}) \Delta_{\mb{z}} \varphi_3(\mb{z})} \Big) \Bigg] e^{-\beta H(\mb{q}, \mb{p}, \mb{z})} \, \dqpz.
    \end{aligned}
    \end{equation}
    Using integration by parts, this simplifies as 
    \begin{align} \label{tm_RH}
        \iiint \varphi_1 \psi_1 \varphi_2 \psi_2 \varphi_3 \psi_3 & \bbs{ \mb{p} \cdot \nabla_{\mb{q}} e^{-\beta H} + (\mb{b} + \mb{z}) \cdot \nabla_{\mb{p}} e^{-\beta H} - \mb{p} \cdot \nabla_{\mb{z}} e^{-\beta H}} \dqpz \nonumber \\
        &= \iiint \varphi_1 \psi_1 \varphi_2 \psi_2 \bbs{ \frac{1}{\beta} \nabla_{\mb{z}} e^{-\beta H} + \mb{z} e^{-\beta H}} \bbs{ \psi_3 \nabla_{\mb{z}} \varphi_3 - \varphi_3 \nabla_{\mb{z}} \psi_3 } \dqpz.
    \end{align}
    Taking \( \varphi_3(\mb{z}) = \psi_3(\mb{z}) \), the RHS of \eqref{tm_RH} becomes zero, yielding 
    \begin{equation} \label{tm_pbq}
        \mb{p} \cdot \nabla_{\mb{q}} e^{-\beta H} + (\mb{b} + \mb{z}) \cdot \nabla_{\mb{p}} e^{-\beta H} - \mb{p} \cdot \nabla_{\mb{z}} e^{-\beta H} = 0.
    \end{equation}
    Thus, by the auxiliary Lemma \ref{lmm:aux}, we obtain 
    \begin{equation}
        \frac{1}{\beta} \nabla_{\mb{z}} e^{-\beta H} + \mb{z} e^{-\beta H} = 0.
    \end{equation}
    This implies \( e^{-\beta H + \frac{\beta}{2}|\mb{z}|^2} \) is independent of \( \mb{z} \). Rewriting \eqref{tm_pbq} as 
    \begin{equation} \label{tm269}
        \mb{p} \cdot \nabla_{\mb{q}} e^{-\beta H + \frac{\beta}{2} |\mb{z}|^2} + \mb{b} \cdot \nabla_{\mb{p}} e^{-\beta H + \frac{\beta}{2} |\mb{z}|^2} + \mb{z} \cdot \bbs{ \nabla_{\mb{p}} e^{-\beta H + \frac{\beta}{2} |\mb{z}|^2} + \beta \mb{p} e^{-\beta H + \frac{\beta}{2} |\mb{z}|^2}} = 0,
    \end{equation}
    we conclude 
    \begin{equation} \label{tm270}
        \nabla_{\mb{p}} e^{-\beta H + \frac{\beta}{2} |\mb{z}|^2} + \beta \mb{p} e^{-\beta H + \frac{\beta}{2} |\mb{z}|^2} = 0.
    \end{equation}
    Therefore, \( e^{-\beta H + \frac{\beta}{2} |\mb{p}|^2 + \frac{\beta}{2} |\mb{z}|^2} \) depends only on \( \mb{q} \), and we conclude there exists a potential \( U(\mb{q}) \) such that \( H = U(\mb{q}) + \frac{1}{2}|\mb{p}|^2 + \frac{1}{2}|\mb{z}|^2 \). From:
    \[
    \mb{p} \cdot \nabla_{\mb{q}} e^{-\beta H + \frac{\beta}{2} |\mb{z}|^2} + \mb{b} \cdot \nabla_{\mb{p}} e^{-\beta H + \frac{\beta}{2} |\mb{z}|^2} = 0,
    \]
    we further conclude \( b(\mb{q}) = -\nabla_{\mb{q}} U(\mb{q}) \).

    \textbf{Step 3.} \( (iii) \Rightarrow (i) \) is exactly the same as Step 3 in the proof of Theorem \ref{thm:revULE}.

    \textbf{Step 4.} The statement \( (iii) \) implies both statements \( (iv) \) and \( (v) \), which are obvious.

    To show that \( (iv) \) implies \( (iii) \), similar to \eqref{rho_Q}, we use the evenness in \( \mb{p} \) to obtain that \( \rho_0(\mb{q}, \mb{p}, \mb{z}) \) solves 
    \[
    \nabla_{\mb{z}} \cdot \bbs{ \mb{z} \rho_0 + \frac{1}{\beta} \nabla_{\mb{z}} \rho_0 } = 0.
    \]
    By Proposition \ref{prop:uniqueness}, we conclude there exists \( c(\mb{q}, \mb{p}) > 0 \) such that 
    \[
    \rho_0 = c(\mb{q}, \mb{p}) e^{-\frac{\beta}{2} |\mb{z}|^2}.
    \]
    Then, by similar arguments to \eqref{tm269} and \eqref{tm270}, we can conclude \( (iii) \).

    To show that \( (v) \) implies \( (iii) \), we plug \( \rho_0 = U_1(\mb{q}) U_2(\mb{p}) U_3(\mb{z}) \) into \( \sL^* \rho_0 = 0 \) and obtain:
    \begin{align} \label{eq:tm273}
        -\mb{p} \cdot \nabla_{\mb{q}} \log U_1(\mb{q}) - \mb{b} \cdot \nabla_{\mb{p}} \log U_2(\mb{p}) &+ \mb{z} \cdot \nabla_{\mb{p}} \log U_2(\mb{p}) + \mb{p} \cdot \nabla_{\mb{z}} \log U_3(\mb{z}) \nonumber \\
        &+ \frac{\alpha}{U_3(\mb{z})} \nabla_{\mb{z}} \cdot \bbs{ U_3(\mb{z}) \mb{z} + \frac{1}{\beta} \nabla_{\mb{z}} U_2(\mb{z})} = 0.
    \end{align}
    By similar arguments as in \eqref{eq:help8}, we have:
    \[
    \log U_2(\mb{p}) = \mb{p}^T \mb{A} \mb{p} + \mb{a} \cdot \mb{p} + c, \quad \log U_3(\mb{z}) = \mb{z}^T \mb{B} \mb{z} + \mb{d} \cdot \mb{z} + e,
    \]
    for some constant matrices \( A, B \in \R^{d \times d} \), and vectors \( \mb{a}, \mb{d} \in \R^d \), and constants \( c, e \in \R \). By similar arguments as in \eqref{tmAA}, we conclude that \( A = -\frac{\beta}{2} I \), \( B = -\frac{\beta}{2} I \), \( \mb{a} = 0 \), and \( \mb{d} = 0 \). Finally, we obtain 
    \[
    \nabla_{\mb{q}} \log U_1(\mb{q}) = \beta \mb{b} = -\nabla_{\mb{q}} U(\mb{q}),
    \]
    for \( U(\mb{q}) = -\log U_1(\mb{q}) \).
\end{proof}

\section{Omitted Proofs for Exponential Convergence and Smoothness of Fokker-Planck Equations} \label{app:hy}

In this section, we present the omitted proofs for the exponential convergence and smoothness of the solution to the Fokker-Planck equation corresponding to overdamped Langevin dynamics, as well as the hypocoercivity and hypoellipticity results for underdamped Langevin dynamics.

\subsection{Exponential Convergence and Smoothness for the Overdamped Case}

We first provide the proof for the exponential convergence and smoothness in the overdamped case. The proof of Proposition \ref{prop:expTV} invokes the following version of Harris's theorem, as stated in \cite[Theorem 3.6]{hairer2010convergence}.

\begin{theorem}[Harris's theorem, Theorem 3.6 in \cite{hairer2010convergence}] \label{thm:Harris's theorem}
    Suppose that \( \mathcal{P} \) is the Markov operator of a discrete Markov semigroup \( \mathcal{P}^n \) for \( n = 1, 2, \dots \) on \( \mathbb{R}^d \). Let \( p(\mb{x}, \mb{y}) \), where \( \mb{x}, \mb{y} \in \mathbb{R}^d \), denote the transition probability of \( \mathcal{P} \). If the following two conditions hold 
    
    \begin{enumerate}[(i)]
        \item (Lyapunov Function) There exists a function \( U : \mathbb{R}^d \to [0, \infty) \) and constants \( K > 0 \) and \( \gamma \in (0,1) \) such that
        \begin{align} \label{cond:Lyapnov}
            \mathcal{P}U(\mb{x}) \leq \gamma U(\mb{x}) + K
        \end{align}
        for all \( \mb{x} \in \mathbb{R}^d \);
        
        \item (Minorisation) For every \( R > 0 \), there exists a constant \( \alpha > 0 \) such that
        \begin{align}
            \|p(\mb{x}, \cdot) - p(\mb{y}, \cdot)\|_{L^1(\mathbb{R}^d)} \leq 2(1 - \alpha),
        \end{align}
        for all \( \mb{x}, \mb{y} \) such that \( U(\mb{x}) + U(\mb{y}) \leq R \).
    \end{enumerate}
    
    Define the following weighted supremum norm:
    \begin{align}
        \|\varphi\|_{U} = \sup_{\mb{x} \in \mathbb{R}^d} \frac{|\varphi(\mb{x})|}{1 + U(\mb{x})}.
    \end{align}
    
    Then \( \mathcal{P} \) admits a unique invariant measure \( \mu^* \). Furthermore, there exist constants \( C > 0 \) and \( \rho \in (0,1) \) such that
    \begin{align}
        \left\|\mathcal{P}^n \varphi - \int_{\mathbb{R}^d} \varphi(\mb{x}) \, \mathrm{d} \mu^*\right\|_U 
        \leq C \rho^n \left\|\varphi - \int_{\mathbb{R}^d} \varphi(\mb{x}) \, \mathrm{d} \mu^*\right\|_U.
    \end{align}
\end{theorem}

In the proof of Proposition \ref{prop:expTV}, we take \( \mathcal{P} = e^{t_0 \mathcal{L}_\varepsilon} \) for some appropriately chosen \( t_0 \), and the Lyapunov function \( U(\mb{x}) = e^{V(\mb{x})/3} \), where \( V(\mb{x}) \) is the potential function.

\begin{proof}[Proof of Proposition \ref{prop:expTV}]
We directly employ Theorem B in \cite{ji2019convergence}. Consider \( U(\mb q) = e^{V(\mb q)/3} \), which is in \( L^1(\rho_0) \). Moreover, we have
\begin{align*}
    \mathcal{L}_\varepsilon U(\mb q) = \frac{1}{3} \left( \mathrm{tr}(\bm\sigma\bm\sigma^T\nabla^2V(\mb q)) - \frac{2}{3}|\bm\sigma^T\nabla V(\mb q)|^2 + \varepsilon \bm\sigma\bm\sigma^T\mb M(\mb q) \cdot \nabla V(\mb q) \right) e^{V(\mb q)/3}.
\end{align*} 
For \( \mb q \) that is not in the support of \( \mb M \) and sufficiently large, by Assumptions (I) and (II), we have
\begin{align*}
    \mathcal{L}_\varepsilon U(\mb q) \leq -\frac{1}{18}|\bm\sigma^T\nabla V(\mb q)|^2 e^{V(\mb q)/3} \leq -C_1 e^{V(\mb q)/3}.
\end{align*}
Thus, for all \( \mb q \in \mathbb{R}^d \), we have
\begin{align} \label{eq:lyap}
    \mathcal{L}_\varepsilon U(\mb q) \leq C_2 - C_1 U(\mb q)
\end{align}
for some constants \( C_1 > 0 \) and \( C_2 > 0 \). Therefore, \( U(\mb q) \) is a strongly unbounded Lyapunov function. By Theorem B in \cite{ji2019convergence}, we conclude exponential convergence.

Next, we prove that \( C_\varepsilon \) and \( r_\varepsilon \) can be selected uniformly. Indeed, the proof of Theorem B in \cite{ji2019convergence} relies on Harris's theorem \cite{hairer2010convergence}, which is based on the following two conditions:

1. Lyapunov condition:\quad For some \( t_0 > 0 \), there exist constants \( \gamma \in (0, 1) \), \( K > 0 \), and a function \( U : \mathbb{R}^d \to \mathbb{R}^+ \) such that
\begin{align} \label{eq:contraction}
    e^{t_0 \mathcal{L}_\varepsilon} U(\mb q) \leq \gamma U(\mb q) + K
\end{align}
holds for all \( \mb q \in \mathbb{R}^d \). According to \cite[Lemma 3.4]{ji2019convergence}, \eqref{eq:contraction} holds by choosing \( U(\mb q) = e^{V(\mb q)/3} \), i.e., the unbounded strong Lyapunov function. Moreover, \( \gamma \) and \( K \) depend on \( \mb M \), \( V \), \( \bm\sigma \), and constants \( C_1 \) and \( C_2 \) in \eqref{eq:lyap}, which are all uniform in \( \varepsilon \).

2. Minorisation condition:\quad It is required to verify the following local minorisation condition: for all \( R > 0 \), there exists \( \alpha > 0 \) such that
\begin{align} \label{eq:contraction2}
    \| p^\varepsilon(\mb q_1, \cdot, t_0) - p^\varepsilon(\mb q_2, \cdot, t_0) \|_{L^1(\mathbb{R}^d)} \leq 2(1 - \alpha)
\end{align} 
holds for all \( \mb q_1 \) and \( \mb q_2 \) in the set \( \{ (\mb q_1, \mb q_2) : U(\mb q_1) + U(\mb q_2) \leq R \} \). Here, \( p^\varepsilon(\mb q_1, \mb q, t_0) \) is the transition probability. By \cite[Lemma 3.3]{ji2019convergence}, we know that \eqref{eq:contraction2} holds for \( \varepsilon = 0 \) with some \( \alpha = \alpha_0 \). Then, by \cite[Corollary 9.8.26]{bogachev2015fokker}, there exists a constant \( C > 0 \) (depending on \( \mb M \), \( V \), and \( \bm\sigma \)) such that for all \( \mb q_1 \) with \( U(\mb q_1) \leq R \), we have
\begin{align*}
    \| p^\varepsilon(\mb q_1, \cdot, t_0) - p^0(\mb q_1, \cdot, t_0) \|_{L^1(\mathbb{R}^d)} \leq C \varepsilon.
\end{align*}
Thus, there exists \( \varepsilon_0 > 0 \) such that for all \( \varepsilon \in (0, \varepsilon_0) \),
\begin{align}
    \| p^\varepsilon(\mb q_1, \cdot, t_0) - p^\varepsilon(\mb q_2, \cdot, t_0) \|_{L^1(\mathbb{R}^d)} \leq 2(1 - \alpha_0 / 2).
\end{align}
Therefore, for each \( R > 0 \), one can find \( \alpha > 0 \) such that \eqref{eq:contraction2} holds uniformly for sufficiently small \( \varepsilon \).

Finally, according to \cite[Remark 3.10]{hairer2010convergence}, the contraction constants \( C_\varepsilon \) and \( r_\varepsilon \) depend on \( \mb M \), \( V \), \( \bm\sigma \), and the following constants: by selecting \( R \) sufficiently large such that
\begin{align*}
    \gamma_0 := \gamma + \frac{2K}{R} < 1, \quad \beta := \frac{\alpha}{2K}, \quad \alpha' := \max \left\{ 1 - \frac{\alpha}{2}, \frac{2 + R \beta \gamma_0}{2 + R \beta} \right\}.
\end{align*}
Here, \( \gamma \) and \( K \) are from \eqref{eq:contraction}, and \( \alpha \) is from \eqref{eq:contraction2}, all of which are uniform for \( \varepsilon \in (0, \varepsilon_0) \). Thus, \( C_\varepsilon \) and \( r_\varepsilon \) are also uniform in \( \varepsilon \in (0, \varepsilon_0) \).
\end{proof}

\begin{proof}[Proof of Lemma \ref{lmm:smoothOL}]
Consider the coordinate transformation \( \mb q' = \bm\sigma^{-1} \mb q \). Then we have \( \nabla_{\mb q'} f = \bm\sigma^T \nabla_{\mb q} f \). Thus, one can reformulate \( \mathcal{L}_{\varepsilon}^* - \frac{\partial}{\partial t} \) as
\begin{align*}
    \mathcal{L}_{\varepsilon}^* - \frac{\partial}{\partial t} 
    &= \sum_{j=1}^d \left( \frac{\partial}{\partial \mb q_j'} \right)^2 
    + \left( \nabla_{\mb q_j'} V(\bm\sigma \mb q') - \varepsilon \bm\sigma^T \mb M(\bm\sigma \mb q') \right) \cdot \nabla_{\mb q_j'} - \frac{\partial}{\partial t} \\
    &\quad + \left( \Delta_{\mb q'} V(\bm\sigma \mb q') - \varepsilon \nabla_{\mb q_j'} \cdot \bm\sigma^T \mb M(\bm\sigma \mb q') \right).
\end{align*}
Let
\begin{align*}
    \mb X_j &= \mb e_j \in \mathbb{R}^{d+1}, \quad j = 1, 2, \dots, d, \quad \mb X_0 = -\mb e_{d+1} + \left[ \nabla_{\mb q_j'} V(\bm\sigma \mb q') - \varepsilon \bm\sigma^T \mb M(\bm\sigma \mb q'), 0 \right], \\
    c &= \Delta_{\mb q'}^2 V(\bm\sigma \mb q') - \varepsilon \nabla_{\mb q_j'} \cdot \bm\sigma^T \mb M(\bm\sigma \mb q').
\end{align*}
Then, we can express \( \mathcal{L}_{\varepsilon}^* - \frac{\partial}{\partial t} \) as
\begin{align*}
    \mathcal{L}_{\varepsilon}^* - \frac{\partial}{\partial t} = \sum_{i=1}^d (\mb X_j)^2 + \mb X_0 + c.
\end{align*}
Note that \( \mb X_i \) for \( i = 1, 2, \dots, d+1 \) are linearly independent at any point \( (\mb q', t) \in \mathbb{R}^d \times (0, +\infty) \). By Theorem \ref{thm:hypoellipticity}, we know that \( \mathcal{L}_{\varepsilon}^* - \frac{\partial}{\partial t} \) is hypoelliptic, and \( \rho^{\varepsilon}(\mb q, t) \) is smooth in both \( \mb q \) and \( t \).
\end{proof}

Now we provide the well-posedness of the Fokker-Planck equation for the overdamped case. The existence and uniqueness are obtained via the semigroup method, while the smoothness is established using Hörmander's hypoellipticity theorem \ref{thm:hypoellipticity}.

\begin{proof}[Proof of Lemma \ref{lem:4.1semi}]
We define a weighted Hilbert space \( L^2(\rho_\infty^\varepsilon) \) with the weight \( \rho_\infty^\varepsilon(\mb q, \mb p) > 0 \) and the norm
\[
    \|f\|_{L^2(\rho_\infty^\varepsilon)} := \iint \rho_\infty^\varepsilon f^2 \, \mathrm{d}\mb q \, \mathrm{d}\mb p.
\]
Additionally, we define a weighted Hilbert space \( H^1(\rho_\infty^\varepsilon) \) with the norm
\[
    \|f\|_{H^1(\rho_\infty^\varepsilon)} := \iint \rho_\infty^\varepsilon f^2 \, \mathrm{d}\mb q \, \mathrm{d}\mb p + \iint \rho_\infty^\varepsilon |\nabla f|^2 \, \mathrm{d}\mb q \, \mathrm{d}\mb p.
\]

For (i), we have
\[
    \frac{\rho_0(\mb q, \mb p)}{\rho_\infty^\varepsilon(\mb q, \mb p)} = \frac{e^{-\varepsilon W(\mb q)} Z_{\varepsilon}}{Z_0}.
\]
Since \( W \in C_c^\infty(\mathbb{R}^d; \mathbb{R}) \), for \( \mb q \) outside the support of \( W \), \( \frac{\rho_0(\mb q, \mb p)}{\rho_\infty^\varepsilon(\mb q, \mb p)} \) is a constant. Therefore, it is in \( L^2(\rho_\infty^\varepsilon) \), and any order derivative of \( \frac{\rho_0}{\rho_\infty^\varepsilon} \) is also in \( L^2(\rho_\infty^\varepsilon) \). Thus, (i) holds.

For (ii), by Lemma \ref{Lmm}, we know that if \( \mb b = -\nabla_{\mb q} U(\mb q) \), then \( -\mathcal{L}_{\varepsilon,1} \) is a maximal monotone operator in \( L^2(\rho_\infty^\varepsilon \, \mathrm{d}\mb q \, \mathrm{d}\mb p) \), with the weight defined in \eqref{eq:req}. By the Hille-Yosida theorem, \( \mathcal{L}_{\varepsilon,1} \) generates\footnote{Sometimes $-\sL_{\eps,1}$ is called the infinitesimal generator instead of $\sL$.} a strongly continuous semigroup on \( L^2(\rho_\infty^\varepsilon) \), denoted as \( S(t) = e^{t\mathcal{L}_{\varepsilon,1}} \).

For (iii), from the contraction semigroup in (ii), we know that the Cauchy problem
\[
    \frac{\partial p^\varepsilon}{\partial t} = \mathcal{L}_{\varepsilon,1} p, \quad p(\mb q, \mb p, 0) = \frac{\rho_0}{\rho_\infty^\varepsilon},
\]
is well-posed, and \( p^\varepsilon(\mb q, \mb p, t) \in C^1([0, T]; L^2(\rho_\infty^\varepsilon)) \). Thus, \eqref{eq:FPpUL} admits a unique solution \( \rho^\varepsilon(\mb q, \mb p, t) \in C^1([0, T]; L^2(\rho_\infty^\varepsilon)) \).
\end{proof}

\begin{proof}[Proof of Lemma \ref{lmm:smoothUL}]
We apply Theorem \ref{thm:hypoellipticity} again. For simplicity, we assume \( \bm \sigma = \mb I \). If this is not the case, a coordinate transformation similar to that in Lemma \ref{lmm:smoothOL} can be used.

The operator \( \mathcal{L}_{\eps}^* - \dfrac{\partial}{\partial t} \) can be expressed as
\begin{align*}
    \mathcal{L}_{\eps}^* - \dfrac{\partial}{\partial t} = \sum_{i=1}^d (\mb X_i)^2 + \mb X_0 + c,
\end{align*}
where $c=d$ and
\begin{align*}
    \mb X_i &= \dfrac{\partial}{\partial p_i}, \quad i = 1, 2, \dots, d, \\
    \mb X_0 &= -\mb p \cdot \nabla_{\mb q} + \left( \nabla_{\mb q} (V - \eps W) + \mb p \right) \cdot \nabla_{\mb p} - \dfrac{\partial}{\partial t}.
\end{align*}
Notice that 
\begin{align*}
    [\mb X_i, \mb X_0] = -\dfrac{\partial}{\partial q_i} + \dfrac{\partial}{\partial p_i}, \quad i = 1, 2, \dots, d.
\end{align*}
Thus, at any point \( (\mb q, \mb p, t) \in \mathbb{R}^{2d} \times (0, \infty) \), we have
\begin{align*}
    \mathrm{span} \{ \mb X_1, \mb X_2, \dots, \mb X_d, [\mb X_1, \mb X_0], \dots, [\mb X_d, \mb X_0], \mb X_0 \} = \mathbb{R}^{2d+1}.
\end{align*}
By Theorem \ref{thm:hypoellipticity}, we conclude that \( \rho^\eps(\mb q, \mb p, t) \) is smooth.
\end{proof}


\bibliographystyle{alpha}
\bibliography{osgbib}

\begin{thebibliography}{MPRV08}

\bibitem[Bar10]{Barbu_2010}
Viorel Barbu.
\newblock {\em Nonlinear differential equations of monotone types in Banach
  spaces}.
\newblock Springer monographs in mathematics. Springer, 2010.

\bibitem[BKRS15]{bogachev2015fokker}
Vladimir~I Bogachev, Nicolai~V Krylov, Michael R{\"o}ckner, and Stanislav~V
  Shaposhnikov.
\newblock {\em Fokker-Planck-Kolmogorov Equations}, volume 207.
\newblock American Mathematical Soc., 2015.

\bibitem[BS18]{buhler2018functional}
Theo B{\"u}hler and Dietmar~A Salamon.
\newblock {\em Functional analysis}, volume 191.
\newblock American Mathematical Soc., 2018.

\bibitem[CHT11]{campisi2011colloquium}
Michele Campisi, Peter H{\"a}nggi, and Peter Talkner.
\newblock Colloquium: Quantum fluctuation relations: Foundations and
  applications.
\newblock {\em Reviews of Modern Physics}, 83(3):771, 2011.

\bibitem[DBE19]{dabelow2019irreversibility}
Lennart Dabelow, Stefano Bo, and Ralf Eichhorn.
\newblock Irreversibility in active matter systems: Fluctuation theorem and
  mutual information.
\newblock {\em Physical Review X}, 9(2):021009, 2019.

\bibitem[DMS15]{dolbeault2015hypocoercivity}
Jean Dolbeault, Cl{\'e}ment Mouhot, and Christian Schmeiser.
\newblock Hypocoercivity for linear kinetic equations conserving mass.
\newblock {\em Transactions of the American Mathematical Society},
  367(6):3807--3828, 2015.

\bibitem[{Ein}05]{E1905}
A.~{Einstein}.
\newblock {{\"U}ber die von der molekularkinetischen Theorie der W{\"a}rme
  geforderte Bewegung von in ruhenden Fl{\"u}ssigkeiten suspendierten
  Teilchen}.
\newblock {\em Annalen der Physik}, 322(8):549--560, January 1905.

\bibitem[GC95]{gallavotti1995dynamical}
Giovanni Gallavotti and Ezechiel Godert~David Cohen.
\newblock Dynamical ensembles in stationary states.
\newblock {\em Journal of Statistical Physics}, 80:931--970, 1995.

\bibitem[GT15]{gilbarg2015elliptic}
David Gilbarg and Neil~S Trudinger.
\newblock {\em Elliptic partial differential equations of second order}, volume
  224.
\newblock springer, 2015.

\bibitem[Hai10]{hairer2010convergence}
Martin Hairer.
\newblock Convergence of markov processes.
\newblock {\em Lecture notes}, 2010.

\bibitem[HM10]{hairer2010simple}
Martin Hairer and Andrew~J Majda.
\newblock A simple framework to justify linear response theory.
\newblock {\em Nonlinearity}, 23(4):909, 2010.

\bibitem[H{\"o}r68]{hormander1967hypoelliptic}
Lars H{\"o}rmander.
\newblock Hypoelliptic second order differential equations.
\newblock 1968.

\bibitem[HP08]{hairer2008ballistic}
Martin Hairer and Grigorios~A Pavliotis.
\newblock From ballistic to diffusive behavior in periodic potentials.
\newblock {\em Journal of Statistical Physics}, 131:175--202, 2008.

\bibitem[JSY19]{ji2019convergence}
Min Ji, Zhongwei Shen, and Yingfei Yi.
\newblock Convergence to equilibrium in fokker--planck equations.
\newblock {\em Journal of Dynamics and Differential Equations}, 31:1591--1615,
  2019.

\bibitem[Kat64]{Kato_1964}
Tosio Kato.
\newblock Demicontinuity, hemicontinuity and monotonicity.
\newblock {\em Bulletin of the American Mathematical Society}, 70(4):548–551,
  Jul 1964.

\bibitem[Kub66]{kubo1966fluctuation}
Rep Kubo.
\newblock The fluctuation-dissipation theorem.
\newblock {\em Reports on progress in physics}, 29(1):255, 1966.

\bibitem[Kun82]{kunitha1982backward}
Hiroshi Kunitha.
\newblock On backward stochastic differential equations.
\newblock {\em Stochastics}, 6(3-4):293--313, 1982.

\bibitem[Kup04]{kupferman2004fractional}
Raz Kupferman.
\newblock Fractional kinetics in kac--zwanzig heat bath models.
\newblock {\em Journal of statistical physics}, 114:291--326, 2004.

\bibitem[KX04]{kou2004generalized}
Samuel~C Kou and X~Sunney Xie.
\newblock Generalized langevin equation with fractional gaussian noise:
  subdiffusion within a single protein molecule.
\newblock {\em Physical review letters}, 93(18):180603, 2004.

\bibitem[LLL17]{li2017fractional}
Lei Li, Jian-Guo Liu, and Jianfeng Lu.
\newblock Fractional stochastic differential equations satisfying
  fluctuation-dissipation theorem.
\newblock {\em Journal of Statistical Physics}, 169(2):316--339, 2017.

\bibitem[Mor65]{mori1965continued}
Hazime Mori.
\newblock A continued-fraction representation of the time-correlation
  functions.
\newblock {\em Progress of Theoretical Physics}, 34(3):399--416, 1965.

\bibitem[MPRV08]{marconi2008fluctuation}
Umberto Marini~Bettolo Marconi, Andrea Puglisi, Lamberto Rondoni, and Angelo
  Vulpiani.
\newblock Fluctuation--dissipation: response theory in statistical physics.
\newblock {\em Physics reports}, 461(4-6):111--195, 2008.

\bibitem[Ons31a]{onsager1931reciprocal1}
Lars Onsager.
\newblock Reciprocal relations in irreversible processes. i.
\newblock {\em Physical review}, 37(4):405, 1931.

\bibitem[Ons31b]{onsager1931reciprocal2}
Lars Onsager.
\newblock Reciprocal relations in irreversible processes. ii.
\newblock {\em Physical review}, 38(12):2265, 1931.

\bibitem[Pav]{Pavliotis}
G~A Pavliotis.
\newblock Stochastic processes and applications.
\newblock page 155.

\bibitem[Pav10]{pavliotis2010asymptotic}
GA~Pavliotis.
\newblock Asymptotic analysis of the green--kubo formula.
\newblock {\em IMA journal of applied mathematics}, 75(6):951--967, 2010.

\bibitem[Rue09]{ruelle2009review}
David Ruelle.
\newblock A review of linear response theory for general differentiable
  dynamical systems.
\newblock {\em Nonlinearity}, 22(4):855, 2009.

\bibitem[Sei12]{seifert2012stochastic}
Udo Seifert.
\newblock Stochastic thermodynamics, fluctuation theorems and molecular
  machines.
\newblock {\em Reports on progress in physics}, 75(12):126001, 2012.

\bibitem[Vil09]{villani2009hypocoercivity}
C{\'e}dric Villani.
\newblock {\em Hypocoercivity}.
\newblock American Mathematical Society, 2009.

\end{thebibliography}
\end{document}